\documentclass{amsart}
\usepackage{amssymb}
\usepackage{amsthm}
\usepackage{graphicx}

\newtheorem{theorem}{Theorem}[section]
\newtheorem{lemma}[theorem]{Lemma}
\newtheorem{corollary}[theorem]{Corollary}

\theoremstyle{definition}

\newtheorem{example}[theorem]{Example}

\theoremstyle{remark}
\newtheorem{remark}[theorem]{Remark}

\theoremstyle{question}

\numberwithin{equation}{section}

\begin{document}
\parskip.3em

\title{Circular handle decompositions of free genus one knots}

\author{Fabiola Manjarrez-Guti\'errez}
\address{CIMAT, A.P. 402, Guanajuato 36000, Gto., M\'EXICO}
\email{fabiola@cimat.mx}

\author{V\'\i ctor N\'u\~nez}
\email{victor@cimat.mx}

\author{Enrique Ram\'\i rez-Losada}
\email{kikis@cimat.mx}


\maketitle

\begin{abstract}
We determine the structure of the circular handle decompositions of
the family of free genus one knots. Namely, if $k$ is a free genus one
knot, then the handle number $h(k)=$ 0, 1 or 2, and, if $k$ is not
fibered (that is, if~$h(k)>0$), then $k$ is almost fibered.
For this, we develop \emph{practical} techniques to construct circular handle
decompositions of knots with free Seifert surfaces in the 3--sphere
(and compute handle numbers of many knots), and, also, we characterize
the free genus one knots with more than one Seifert surface.
These results are obtained through analysis
of spines of surfaces on handlebodies. Also we show that there are
infinite families of free genus one knots with either $h(k)=1$ or~$h(k)=2$.

\end{abstract}

\section{Introduction}

In the study of the topology of a given 3--manifold, $M$, it has been
useful to consider
regular real-valued Morse functions $f:M\rightarrow\mathbb{R}$
where~$M$ has some smooth structure. A regular real-valued Morse
function on $M$  
corresponds to a handle decomposition of $M$ of the form $M= b_0 \cup B_1
\cup P_1\cup\cdots \cup B_r \cup P_r \cup b_3$ where~$b_0$ is a collection
of 0-handles, $B_j$ is a collection of 1--handles, $P_j$ is a
collection of 2--handles, and $b_3$ is a collection of 3--handles, in such
a way that
the  $i$--handles of the decomposition are neighbourhoods of the critical
points of index $i$ of the Morse function ($j=1,\dots, r$, and $i=0,1,2,3$).
In a celebrated paper (\cite{ST}),  
it is introduced  the concept of 
\textit{thin position} for 3--manifolds; the idea is to build the
manifold as described above (that is, step by step: adding to the
set $b_0$ the set 
$B_1$, and then adding $P_1$, and then adding $B_2$, and so on) with a
sequence of 
selected sets of 1--handles and
sets of 2--handles chosen to keep the boundaries of the intermediate steps as
simple as possible. 

Now if a 3--manifold $M$ satisfies $H^1(M;\mathbb{Q})\neq0$, then there are
essential (non-nulhomotopic) regular Morse functions $f:M\rightarrow S^1$, and 
one can always find this kind of functions having only critical points
of index~1 and~2 (see Section~\ref{sec22}). 
Such a function 
corresponds to a \emph{circular handle decomposition}  
$M= F\times [0,1] \cup B_1 \cup P_1 \cup\cdots \cup B_r \cup P_r $
where  $F$ is a properly embedded surface in $M$, $B_j$ is a collection of
1--handles, and $P_j$ is a 
collection of 2--handles (the handles are glued along,
say,~$F\times\{1\}$), and, as above,  
the set of $i$--handles of the decomposition corresponds to the critical
points of index $i$ of the Morse function. With this kind of 
circular handle decompositions we may also require that the intermediate steps
be as simple as possible: that requirement acquires the notion of thin
position for circular handle decompositions. The existence of
these decompositions gives rise to numerical topological invariants
such as the (circular) handle number, $h(M)=\sum_{1=1}^r\#(B_i)$ where
the sum
$\sum\#(B_i)$ is minimal among all  circular handle decompositions; also, when
the decomposition is in thin position, we obtain the circular width,
$cw(M)$ (See Section~\ref{sec24}).

Outstanding examples of manifolds that admit circular handle
decompositions, are the exteriors of links in $S^3$. In this case the
interesting intermediate surfaces in the decomposition are Seifert
surfaces for the given link (these intermediate surfaces have no
closed components, and, if the decomposition is in thin position,
they
are a sequence of Seifert surfaces which are alternately
incompressible and weakly 
incompressible. See \cite{fabiux}, Theorem~3.2, where there is a statement for
knots, but its proof works verbatim for links).

If the exterior of a link $\ell$ in $S^3$ admits a circular
decomposition of the form~$E(\ell)= F\times [0,1] \cup B_1 \cup P_1 $,
and this decomposition is in 
thin position, we say that $\ell$ is an \emph{almost fibered
  link}. One may regard the set of almost fibered knots as the set of knots with
non-trivial simplest circular handle structure.

Thus, an interesting problem of this theory
is: Determine the set of all
almost fibered knots. We solve this problem for the family of free
genus one knots. In fact we show that all free genus one knots are
almost fibered (Theorem~\ref{thm67}).

Also it is interesting to find explicit constructions of circular
handle decompositions of the exterior of a given link which are
minimal (that is, that realize the handle number), or that are in thin
position. In \cite{Go1}, although in other context,
explicit minimal circular handle decompositions of the exterior of the
250 knots in Rolfsen's table are given. Of these knots, 117 are
fibered, and 132 have handle number one. 
As far as we know, there
are no other 
previously published explicit constructions of circular handle decompositions of
exteriors of links in the 3--sphere.

As mentioned above, in this paper we are interested mainly in the
circular handle 
structures of the family of free genus one knots:

In the first part of this work (Section~\ref{sec3}) we develop techniques to
construct explicit circular decompositions of link exteriors for
links that admit a free Seifert surface; these decompositions are
interesting, of 
course, when the free Seifert surface used in the construction is of
minimal genus for the link. The information needed to construct these
decompositions for the exterior of a given link is encoded in some
spine of a free Seifert surface of the link. In this sense, the
techniques developed in Section~\ref{sec3} (and through all this paper) could be
regarded as elements for a possible theory of spines of surfaces on
handlebodies that might be worthy of consideration. As applications we
construct minimal circular 
decompositions for all rational knots and links and, also, for a
family of pretzel knots, 
namely, pretzel knots of the form~$P(\pm3,q,r)$ with
$|q|,|r|$ odd integers $\geq3$. These circular decompositions for both
families of links are all minimal and
have handle number~one; they are also in thin position, giving also the
circular width of each link considered. This last family gives examples
of 
non-fibered
knots whose handle number is strictly less than their tunnel number
(Remark~\ref{remark311}). Also, it is shown that free genus one knots have handle
number $\leq 2$ (Corollary~\ref{coro36}).

Secondly (Section~\ref{sec4}), we construct circular
handle decompositions for the exteriors of all pretzel knots of the
form $P(p,q,r)$ with $|p|,|q|,|r|$ odd integers $\geq5$, and we show that these
decompositions are minimal with handle number two (Theorem~\ref{thm41}), and are
also in thin position, giving the circular width equal to 6 for each of these
knots. 
These examples answer a question posed in
\cite{PRW} (Remark~\ref{remark45}). 

Next, in Section~\ref{sec5}, we give a characterization
of the free genus one knots that admit at least two different
(non-parallel) Seifert
surfaces of genus one. This characterization is given in terms of the
existence of a special spine for the given genus one free Seifert surface of
the knot (see Theorem~\ref{thm52}).

Using the characterization given in  Section~\ref{sec5} we show, in the final
part of this work, 
that all (non-fibered) free
genus one knots are almost fibered (Theorem~\ref{thm67}).

It follows from the proof of Theorem~\ref{thm67}, that the free genus one
knots with handle number two have a unique minimal genus Seifert
surface (that is, free genus one knots with at least two genus one
Seifert surfaces have handle number one). It is an interesting open
problem to determine the family of free genus one knots with handle
number two. 


\section{Preliminaries}
Unless explicitly stated, we will use the word `knot' for a knot or a
link in $S^3$. That is, we will emphasize connectedness if
needed. Otherwise, we will admit non-connected knots.

Let $X$ be a manifold and let $Y\subset X$ be a sub-complex. We write
$E(Y)=\overline{X-\mathcal{N}(Y)}$ for the \emph{exterior} of $Y$ in
$X$ where $\mathcal{N}(Y)$ is a regular neighbourhood of $Y$ in $X$. 

Let $X$ be a manifold and let $Y\subset X$ be a properly embedded
submanifold. $Y$ is called \emph{$\partial$--parallel} in $X$, or
\emph{parallel into} $\partial X$, if there is an embedding
$e:(Y,\partial Y)\times I\rightarrow (X,\partial X)$, such that
$e_0:Y\rightarrow Y$ is the 
identity, and $e_1(Y)\subset \partial X$. If $Y$ is
$\partial$--parallel in $X$ with embedding $e:(Y,\partial Y)\times
I\rightarrow (X,\partial X)$, then the submanifold~$e(Y\times I)$ is
called a \emph{$\partial$--parallelism for} $Y$. Notice that if $Y$ is
disconnected with components~$Y_1,\dots,Y_n$, and $Y$ is
$\partial$--parallel in $X$ with a $\partial$--parallelism $W$,
then $W$ is a disjoint union of $\partial$--parallelisms
$W_1,\dots,W_n$ for $Y_1,\dots,Y_n$, respectively.

\subsection{Seifert Surfaces}
Let $k\subset S^3$ be a knot, and let $F$ be a Seifert surface for~$k$;
that is, $F$ is an orientable surface and $\partial F=k$. Then, by
drilling out a small neighbourhood, $\mathcal{N}(k)$, of $k$, the
surface $\widehat F=F\cap E(k)$ is a properly embedded surface in
$E(k)$, the exterior of $k$ in $S^3$, and one may assume that
$\partial\widehat F$ is parallel to $k$ in $\mathcal{N}(k)$. Usually,
we identify $F$  
with $\widehat F$; but, more appropriately, we start with~$F\subset E(k)$ a
Seifert surface for $k$. Seifert surfaces may be disconnected, but
they are not allowed to contain closed components. The \emph{genus}~$g(k)$ of a knot $k$ is the minimal genus among all Seifert surfaces for $k$.

A surface $F\subset S^3$ is called \emph{free} if $E(F)$ is a
handlebody. The \emph{free genus} of a knot $k$, $g_f(k)$, is
the minimal genus among all free Seifert surfaces for $k$.

In this work we will be interested mainly in free genus one knots.

\subsection{Handle decompositions of  rel~$\partial$ Cobordisms}
\label{sec22}
Let $W$ be a cobordism rel~$\partial$ between
surfaces with no closed components, $\partial_+W$ and $\partial_-W$.
A \emph{moderate handle decomposition of $W$} is a decomposition of
the form 
$W\cong \partial_+W\times I\cup(\textrm{1--handles})\cup(\textrm{2--handles})$. Given $W$, a cobordism rel~$\partial$ between
surfaces with no closed components, $\partial_+W$ and $\partial_-W$,
 it is easy to find
a moderate decomposition as above by
considering a triangulation of the exterior
$E(\partial_+W)=\overline{W-\mathcal{N}(\partial_+W)}$.

Given a cobordism $W$ and a moderate handle decomposition for $W$, one
can find a regular Morse function $f:W\rightarrow I$ which realizes
the handle decomposition of~$W$. That is, $f$ only has critical points
of index 1 and 2, and neighbourhoods of the critical points of $f$
correspond to the 1 and 2--handles of $W$, and the preimage of each
regular value of $f$ is a properly embedded surface in $W$. We will call
such a Morse function a \emph{moderate Morse function} (see \cite{PRW}).

\subsection{Circular decompositions}

Let $k$ be a knot in $S^3$. Since $H_1(E(k))$ is a free Abelian group
of positive 
rank, we can always find an essential (non-nulhomotopic) 
moderate Morse
function~$f:E(k)\rightarrow S^1$. Any
such Morse function, as in Subsection~\ref{sec22}, induces a 
decomposition
$$E(k)=(F\times I)\cup B\cup P$$
where $F\subset E(k)$ is a Seifert surface for $k$, $B$ is a set of
$n$ 1--handles glued along, say, 
$F\times\{1\}$, and $P$ is a set of the same number, $n$, of 2--handles
glued along the same side.

We call such a decomposition a \emph{circular handle decomposition of
  $E(k)$ based on~$F$}, and write $h(F)=n$, the \emph{handle number of
  $F$}, where $n$ is the minimal number of
1--handles among all circular handle decompositions of $E(k)$ based on $F$. The
\emph{circular handle 
  number of~$k$}, or simply the \emph{handle 
  number of~$k$},~$h(k)$, is the minimal $h(F)$ among all Seifert
surfaces $F\subset E(k)$.
Notice that $h(k)=0$ if and only if $k$ is a fibered knot.

By rearranging the critical points of a moderate Morse
function $f:E(k)\rightarrow S^1$, we can 
thin a circular handle
decomposition of $E(k)$:
$$E(k)=(F\times I)\cup B_1\cup P_1\cup B_2\cup P_2\cup\cdots\cup
B_\ell\cup P_\ell$$
where $B_i$ is a set of 1--handles glued along $F\times\{1\}$, and $P_i$ is
a set of 2--handles,~$i=1,\dots,\ell$ (of course, it is not always possible
to thin a given circular handle decomposition). 

For $i=1,\dots,\ell$, the set $W_i=(F\times [\frac12,1])\cup B_1\cup
P_1\cup\cdots\cup B_i$ gives a moderate handle decomposition for the
rel $\partial$ cobordism $W_i$  with $\partial_+W_i=F\times\{\frac12\}$.
Write~$S_i=\partial_-W_i$. Now we define
$$\displaystyle c(S_i)=\sum_{j=1}^{n_i}\left(1-\chi(G_{i,j})\right)$$ 
where $\chi$ stands for Euler characteristic, and
$G_{i,1},\dots,G_{i,n_i}$ are the components of~$S_i$ (Notice that
there are no closed components of $S_i$ for, $F$ has no closed
components and the handle decomposition is moderate). Order the surfaces
$S_{\sigma(1)},S_{\sigma(2)},\dots,S_{\sigma(\ell)}$ in such a way that
$c(S_{\sigma(i)})\geq c(S_{\sigma(i+1)})$ for $i=1,\dots,\ell-1$, where
$\sigma$ is a permutation in the symbols $1,\dots,\ell$. Then
the \emph{circular width} of this decomposition is the tuple
$(c(S_{\sigma(1)}),c(S_{\sigma(2)}),\dots,c(S_{\sigma(\ell)}))$. The
\emph{circular width of $k$}, $cw(k)$, is the minimal circular width,
with respect to
lexicographic order, among all thinned circular decompositions of
$E(k)$ based on all possible Seifert surfaces for $k$.

Let $k\subset S^3$ be a knot such that its circular width has the form
$cw(k)=(n)$. Then we write $cw(k)=n$, or $cw(k)\in\mathbb{Z}$. If $k$
is a non-fibered knot and $cw(k)\in\mathbb{Z}$, then $k$ is said to be an
\emph{almost fibered} knot. 

\begin{remark}
\label{remark22}
\textbf{Equivalence of knots}.
Let $k,\ell\subset S^3$ be two knots. If the pairs~$(S^3,k)$ and
$(S^3,\ell)$ are homeomorphic, then their exteriors also are
homeomorphic, $E(k)\cong E(\ell)$; and
therefore, the exteriors of $k$ and $\ell$ have homeomorphic handle
decompositions. We 
regard two knots as being \emph{equivalent} if their
corresponding pairs are homeomorphic.
\end{remark}

\begin{remark}
\label{remark23}
\textbf{Construction of circular decompositions}.
To describe or, rather, to actually \emph{construct} a decomposition
$$E(k)=(F\times I)\cup B\cup P$$
where $B$ is a set of 1--handles, and $P$ is a set of 2--handles,
it is convenient to write
$$\textstyle E(k)=(F\times [\frac12,1])\cup B\cup P\cup(F\times[0,\frac12]).$$
Then to obtain (describe) this circular decomposition we can either
\begin{enumerate}
\item \label{uno} Start with a regular neighbourhood $\mathcal{N}(F)$ of $F$ in
  $E(k)$. Then add a
  number of 1--handles to $\mathcal{N}(F)$ (the elements of $B$) on one side, 
  say $F\times\{1\}$, and then add the same number of 2--handles (the
  elements of $P$) on the
  same side.

\noindent
  The complement of the union above is a regular neighbourhood of
  $F\times\{0\}$ in $E(k)$.
Or

\item \label{dos} Start with  $E(F)$, the exterior of $F$ in $E(k)$. Then drill
  a number of 2--handles (the elements of $B$) out of $E(F)$. Now drill
  the same number of 
  1--handles (the elements of $P$) out of $E(F)$.

\noindent
  Here one should be careful that the drilled out 2--handles intersect
  $\partial E(F)$ on the same side, say $F\times\{1\}$, and that the
  following drilled out 1--handles intersect the remaining boundary of
  $E(F)$ on the same side.

\noindent
  The result of this drilling is a regular neighbourhood of
  $F\times\{0\}$ in $E(k)$.
\end{enumerate}

Of course, in (\ref{uno}) above, `$\mathcal{N}(F)$' stands for
$F\times [\frac12,1]$, and in (\ref{dos}), `$E(F)$' stands for the
exterior $\overline{E(k)-F\times[\frac12,1]}$. To describe a thinned
circular decomposition, one proceeds similarly, but now there will be
several steps. Note that, in this kind of decomposition, a thinned
decomposition, the number of
1--handles and the number of 2--handles at each step are not necessarily the
same.

We emphasize that the  main use of the program outlined in (\ref{uno}) is to
describe an explicit circular handle decomposition of some given
example. 
\end{remark}

\begin{remark}
\label{remark24}
\textbf{Decompositions of non almost fibered knots}.
Now start with a circular decomposition
$$\textstyle E(k)=(F\times [\frac{1}{2},1])\cup B_1\cup P_1\cup B_2\cup P_2\cup\cdots\cup
B_\ell\cup P_\ell\cup(F\times[0,\frac12])$$
which realizes $cw(k)$, the circular width of $k$. For $i=1,\dots,\ell$, the set $V_i=(F\times [\frac12,1])\cup B_1\cup
P_1\cup\cdots\cup B_i\cup P_i$ gives a moderate handle decomposition for the
rel~$\partial$ cobordism $V_i$  with $\partial_+V_i=F\times\{\frac12\}$.
Write $T_i=\partial_-V_i$. Then the $\ell$ disjoint surfaces $T_1,T_2,\dots,T_{\ell}=F$
are incompressible in $E(k)$ and are non-parallel by pairs
(see~\cite{fabiux}, Theorem~3.2. As noted in the Introduction, the
theorem also holds for non-connected knots). That is, 

\emph{If $k$ is non fibered and not an almost fibered knot, then $k$
  has at least two non-parallel 
  incompressible Seifert surfaces.}
\end{remark}

\begin{remark}
\label{remark25}
\textbf{Decompositions of pairs}.
Let $k\subset S^3$ be a knot with Seifert surface~$F\subset
E(k)$. There is a copy of $F$, $F_0\subset \partial E(F)$, such that $E(F)$ is
a cobordism 
rel~$\partial$ between $F_0=\partial_+E(F)$ and $\partial_-E(F)$. We
commit an abuse of notation by
identifying $F$ with $F_0$. To
find a circular decomposition of $E(k)$ based
on $F$ is the same as finding a moderate handle decomposition of the
rel~$\partial$ cobordism~$E(F)$. A \emph{handle decomposition of the
  pair $(E(F),F)$} is, by definition, a handle decomposition of the
rel~$\partial$ cobordism~$E(F)$.

Now let $\ell\subset S^3$ be another knot with Seifert surface $G\subset
E(\ell)$. If there is a homeomorphism of pairs $(E(F),F)\cong(E(G),G)$, then
the handle decompositions of the pairs $(E(F),F)$ and $(E(G),G)$ (as
well as those of $E(F)$ and $E(G)$ as rel~$\partial$ cobordisms) are in 1-1 correspondence via the given
homeomorphism. That is:

\emph{To find circular decompositions of $E(k)$ based on $F$, we need
  only to construct moderate handle decompositions of the homeomorphism class
  of the pair $(E(F),F)$. In particular, it is not necessary to regard
  $E(F)$ as embedded in $S^3$.}

This remark is very helpful in the search of circular decompositions.
\end{remark}


\subsection{Spines}
\label{sec24}
Let $X$ be either a handlebody or a surface with boundary. A \emph{spine} of
$X$ is a graph $\Gamma\subset X$ such that $X$ is a regular
neighbourhood of $\Gamma$. In this work we mainly consider
spines of the form $\Gamma\cong\bigvee_{i=1}^n S^1$, a wedge of
circles. 
We
write~$\Gamma=a_1\vee\cdots\vee a_n$ to emphasize the circles
involved, and we assume that the curves $a_i$ carry a given
orientation. Notice that 
it is allowed for $\Gamma$ to be a single simple closed curve.

Let $k\subset S^3$ be a knot, and let $F\subset E(k)$ be a Seifert
surface for $k$. A regular neighbourhood $\mathcal{N}(F)$ of $F$ in
$E(k)$ admits a product structure $\mathcal{N}(F)=F\times I$
where~$\partial F\times I=\mathcal{N}(k)\cap\mathcal{N}(F)$. A spine 
$\Gamma\subset F\times\{0\}$, $\Gamma\cong\bigvee_{i=1}^n S^1$, is
also a spine for $\mathcal{N}(F)$, and the 
graph $\Gamma$ induces a product structure $\mathcal{N}(F)=G\times
I$, where, say,~$G\times\{0\}$ is a regular neighbourhood of $\Gamma$
in $\partial\mathcal{N}(F)$ (here, of course, $G$ is isotopic to $F$
in $\partial\mathcal{N}(F)$). A
spine~$\Gamma\subset F\times\{0\}$ is also a graph
$\Gamma\subset\partial E(F)$. A spine for
$F$,~$\Gamma\subset F\times\{0\}$ (or $\Gamma\subset F\times\{1\}$),
is 
called a \emph{spine for $F$ on  $\partial\mathcal{N}(F)$}. Also, we
say that $\Gamma$ is a \emph{spine for $F$ on  $\partial E(F)$}.

If $\Gamma$ is a spine for $F$ on $\partial E(F)$, and $G$ is a regular
neighbourhood of $\Gamma$ in~$\partial E(F)$, then a \emph{handle
  decomposition for the pair $(E(F),\Gamma)$} is, by definition, a handle
decomposition for the pair $(E(F),G)$. 

Let $\Gamma=a_1\vee \cdots\vee a_n$ be a spine for $F$ on
$\partial E(F)$, and let $t(a_i)$ be a Dehn twist on~$F$ along the curve
$a_i$. If $\widetilde\Gamma$ is the graph obtained from $\Gamma$ by
replacing the curve~$a_j$ by the curve
$t(a_i)(a_j)$, then $\widetilde\Gamma$ is also a spine for $F$. The graph
$\widetilde\Gamma$ is called \emph{the spine for $F$ obtained from
  $\Gamma$ by sliding $a_j$ along
  $a_i^{\pm1}$} ($i,j\in\{1,\dots,g\}$).

\begin{remark}
\label{remark26}
Notice that if $\widetilde\Gamma$ is
another spine for $F$ on $\partial E(F)$, and $\widetilde G$ is a regular
neighbourhood of $\widetilde\Gamma$ in $\partial E(F)$, then the pairs
$(E(F),\Gamma)$ and $(E(F),\widetilde\Gamma)$ usually are not homeomorphic, but
the pairs $(E(F),F)$ and $(E(F),\widetilde G)$ are homeomorphic.
Thus:

\emph{To find circular decompositions of $E(k)$ based on $F$, we need
  only to construct moderate handle decompositions of the homeomorphism class
  of a pair 
  $(E(F),\Gamma)$ for some spine $\Gamma$ for $F$ on $\partial E(F)$.}
\end{remark}

\begin{remark}
\label{remark27}
Let $F\subset S^3$ be a connected orientable surface with
boundary~$k=\partial F$. If a spine $\Gamma$ for $F$ on $\partial
\mathcal{N}(F)$ is also a spine for~$E(F)$, then 
$k$ is a fibered knot with fiber $F$. Indeed, $E(F)$ is a handlebody (for it is an irreducible
3--manifold with connected boundary, and with free fundamental group), and
both $\mathcal{N}(F)$ and $E(F)$ admit a product structure of the form
$G\times I$, where $G$ is a regular
neighbourhood of $\Gamma$ in~$\partial\mathcal{N}(F)=\partial E(F)$. 
\end{remark}

\subsection{Whitehead diagrams}

Let $H$ be a genus $g$ handlebody, and let $x_1,\dots,x_g$ be a system of
meridional disks for $H$. The exterior $E(x_1\cup\cdots\cup x_g)$ is a
3--ball with~$2g$ fat vertices $x_1,\bar x_1,\dots,x_g,\bar x_g$ on its
boundary, where $x_i=x_i\times\{0\}$ and~$\bar x_i=x_i\times\{1\}$ are
the copies of~$x_i$ in the product structure $\mathcal{N}(x_i)=
x_i\times I\subset H$, $i=1,\dots,g$.

There is a 1-1
correspondence between isotopy classes of systems of meridional disks
$\{x_1,\dots,x_g\}$ for $H$, and homotopy classes of spines
of the form $a_1\vee\cdots,\vee a_g\subset H$ such that $\#(a_i\cap
x_i)=1$, and $a_i\cap x_j=\emptyset$ for~$i\neq j$, $i,j=1,\dots,g$. It
is convenient to commit an abuse of notation, and write both $\{x_1,\dots,x_g\}$
for a meridional system of disks for~$H$, and $\{x_1,\dots,x_g\}$ for the
corresponding basis of $\pi_1(H)$ represented by the curves
$a_1,\dots,a_g$ in the 1-1 correspondence above. Throughout this paper we
adhere to this abuse of notation.

A graph~$\Gamma=a_1\vee\cdots\vee a_n\subset \partial H$
intersects $E(x_1\cup\cdots \cup x_g)$ in
a set of sub-arcs of the curves $a_i$; some of these arcs intersect
in the base point of~$\Gamma$. 
These arcs together with $x_1,\bar x_1,\dots,x_g,\bar x_g$ form a
graph $G$ with $2g$ fat vertices immersed on~$\partial E(x_1\cup\cdots
\cup x_g)$. The  
base point of~$\Gamma$ appears in the drawing on $\partial E(x_1\cup\cdots
\cup x_g)$ as the intersection of
some edges of~$G$, but the base point of~$\Gamma$ is not considered a
vertex of $G$. We require that the graph~$G$ has no
loops, that is, that there are no edges with ends
in the same fat vertex of $G$. In our examples, we will be able to realize
this assumption ---no loops in $G$--- through the use of some
isotopies of~$H$. For each $i$ we number the ends of the arcs in $x_i$
and   
$\bar x_i$ in such a way that the gluing homeomorphisms, which
recover~$H$ from~$E(x_1\cup\cdots\cup x_g)$, identify equally
numbered points. The immersion of the graph~$G$ in $\partial E(x_1\cup\cdots\cup
x_g)$, together with these numberings, is called
\emph{the Whitehead
  diagram of the pair 
  $(H,\Gamma)$} associated to
the system of meridional disks $x_1,\dots,x_g\subset H$  (see Figure~\ref{figure1}). The graph~$G$
is called the \emph{Whitehead graph} of the corresponding Whitehead diagram.
\begin{figure}
\centering
\includegraphics[height=50mm]{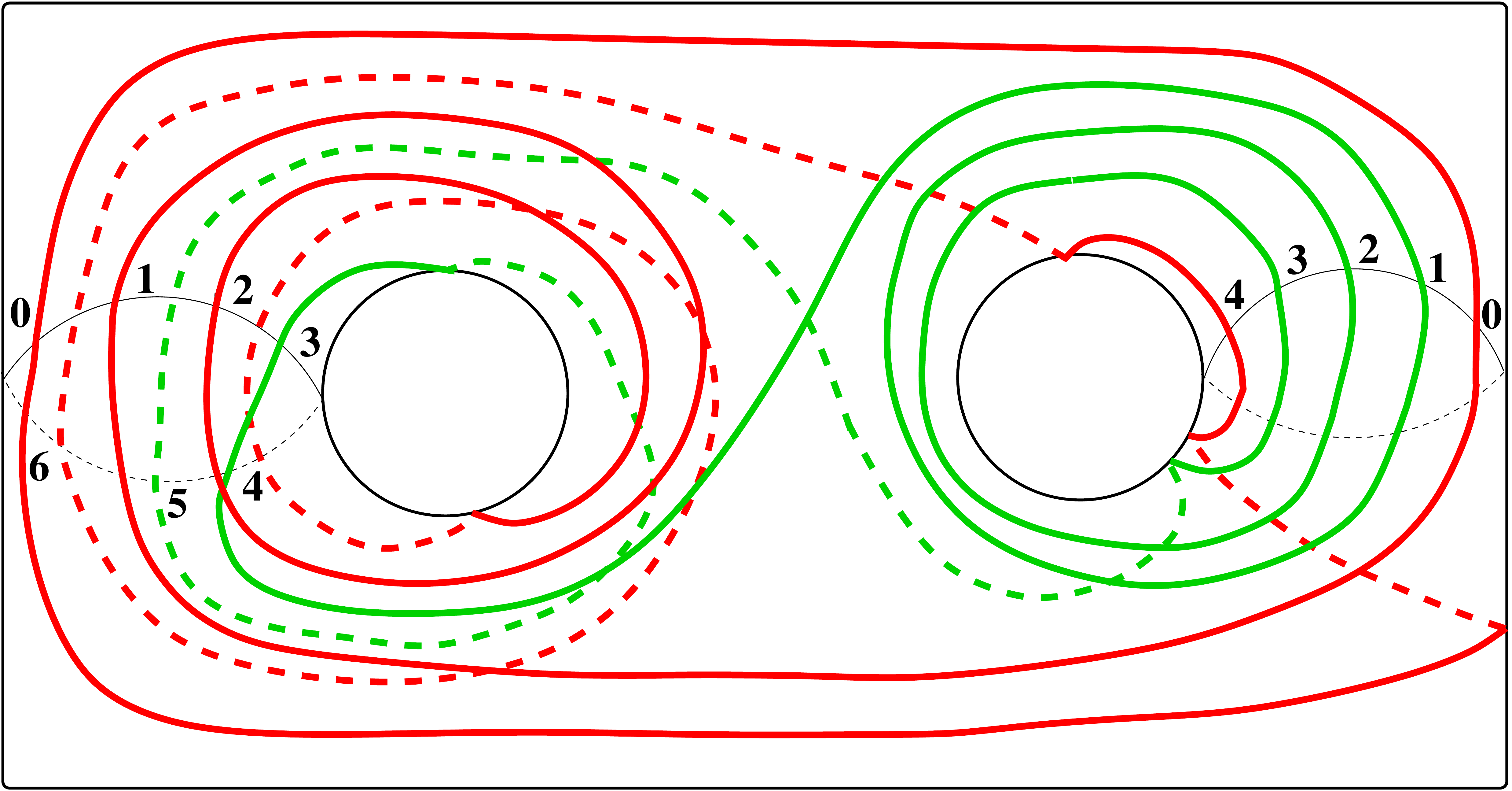}
\vskip.2true cm
\includegraphics[height=70mm]{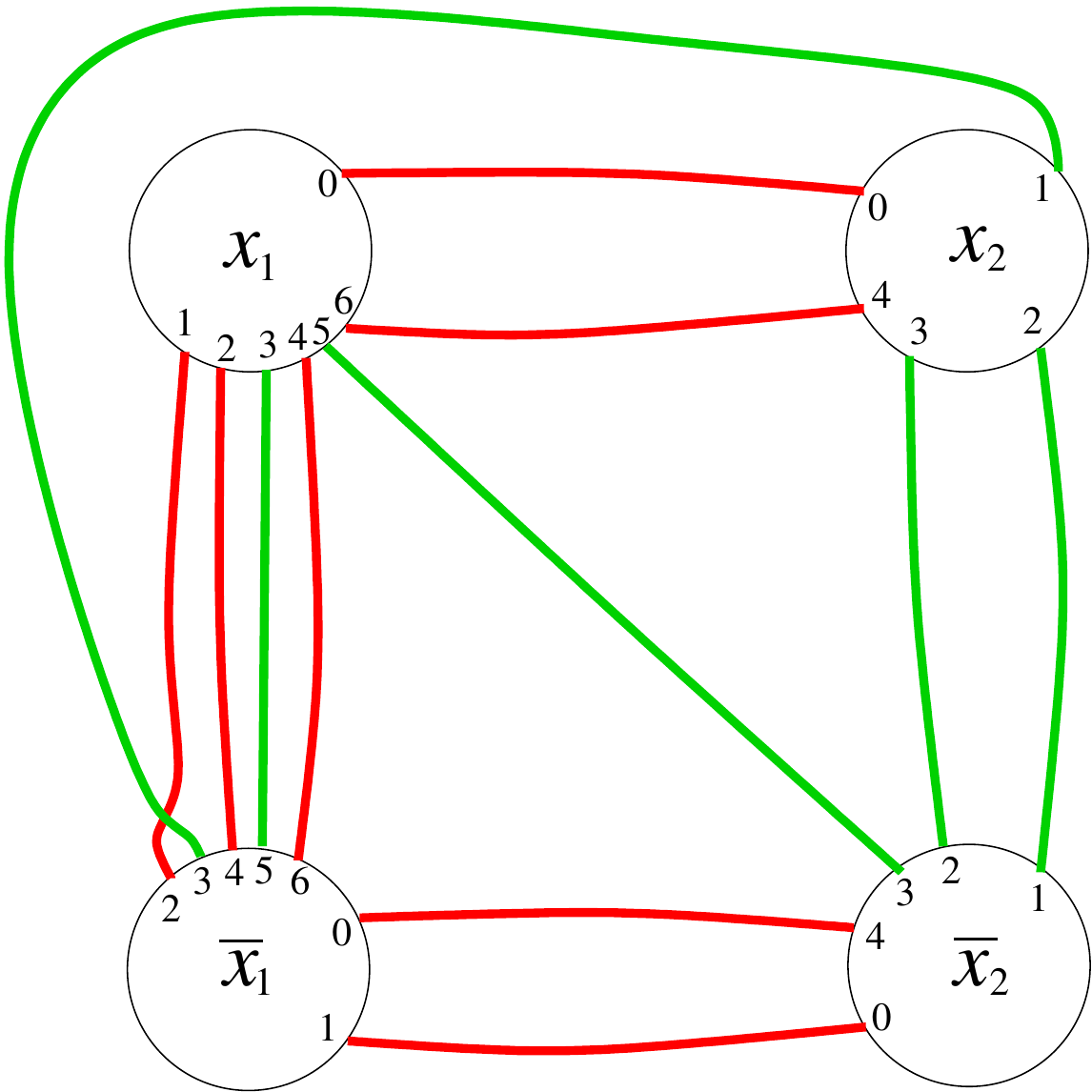}
\caption{A Whitehead diagram associated to the exterior of the pretzel knot $p(5,5,5)$.}
\label{figure1}
\end{figure}

Let $X$ be a graph, and let $e$, $f$ be two edges of $X$; we say that
$e$ and $f$ are \emph{parallel} if they connect the same pair of
vertices of $X$. The
\emph{simple graph associated to~$X$} is the graph obtained
from $X$ by 
replacing each parallelism class of edges of $X$ by a single edge,
and deleting each loop in $X$ (if any).

If $X$ is a connected graph, a vertex $v$ of~$X$ is called a \emph{cut
  vertex} of $X$ if $X-\{v\}$ is not connected. Notice that a loop-less
graph
$X$ contains a cut vertex if and only if the simple graph associated
to $X$ contains a cut vertex.

Let $\mathcal{F}$ be a free group with basis $Y$, and let $A$ be a set
of cyclically reduced words on $Y\cup Y^{-1}$, regarded as elements of
$\mathcal{F}$. The 
\emph{genuine 
  Whitehead graph of~$A$} is the graph, $\Gamma$,  with vertex set
$Y\cup Y^{-1}$, 
and if~$\alpha\in A$, when cyclically $\alpha$ contains the word
of length two $v_1v_2$, then there is an edge in~$\Gamma$ from $v_1$
to $v_2^{-1}$ for,~$v_1,v_2\in Y\cup Y^{-1}$. If~$\alpha$ is of length 1,
$\alpha=v$, then there is an edge from $v$ to~$v^{-1}$. If $A$ is a
set of elements of $\mathcal{F}$, we can replace 
$A$ with a set $A'$ of cyclically reduced words representing
 the conjugacy classes of the elements of $A$, and then the
\emph{genuine Whitehead graph} of~$A$ is, by definition, the genuine Whitehead
graph of $A'$. The genuine Whitehead graph of a set of elements of
$\mathcal{F}$ is regarded as being embedded in 3--space and also contains no loops.

Let $\mathcal{F}$ be a free group and let
$A$ be a set of elements of
$\mathcal{F}$. Then~$A$ is called \emph{separable} if there
exists a non-trivial splitting 
$\mathcal{F}\cong \mathcal{F}_1*\mathcal{F}_2$ such that each
$\alpha\in A$ represents, up to conjugacy, 
an element of $\mathcal{F}_j$ for some $j$.

\begin{theorem}[Theorem~2.4 of~\cite{stallings}] 
\label{thm24stallings}
Let $A$ be a set of elements
of a free group $\mathcal{F}$ with genuine Whitehead graph $\Gamma$.
If $\Gamma$ is connected and if $A$ is
separable in~$\mathcal{F}$, then there is a cut vertex in $\Gamma$.
\end{theorem}

The following result follows from
Theorem~\ref{thm24stallings} and is included here for future reference.

\begin{corollary}
\label{coro29}
Let $\Gamma=a_1\vee\cdots\vee a_n$ be a wedge of $n$ simple closed curves
 embedded in the boundary of a handlebody~$H$.
Assume that
 for some Whitehead diagram of the pair $(H,\Gamma)$,
the Whitehead graph of this diagram is
connected and has no cut
vertex.
Then $\Gamma$
intersects every essential disk of~$H$. 
\end{corollary}
\begin{proof} 
  Let $G$ be the Whitehead graph of the pair
  $(H,\Gamma)$ with respect to some system of meridional disks
  $\{x_1,\dots,x_g\}$, such that $G$ has no cut vertex and is
  connected. In particular $G$ has no loops.
  If we regard $G$ as a graph $G'$
  embedded in 3--space so that the base point of $\Gamma$ vanishes,
  then $G'$ is the genuine Whitehead graph of the set of elements of
  $\pi_1(H)$ 
  represented by $\{a_1,\dots,a_n\}$ with respect to the
  basis~$\{x_1,\dots,x_g\}$. Since $G$ is connected and has no cut
  vertex, 
  it follows that $G'$ is also connected and has no cut vertex (recall
  that the base point of $\Gamma$ is not part of~$G$; then $G$ and
  $G'$ are isomorphic graphs). If
  there is an essential disk in $H$ disjoint with $\Gamma$, then the 
  set of elements of $\pi_1(H)$ represented by $\{a_1,\dots,a_n\}$ 
  clearly is separable, and by Theorem~\ref{thm24stallings}, $G'$ has
  a cut vertex or is disconnected. Since $G'$ is connected and has no
  cut vertex, it follows that  $\Gamma$ intersects
  every essential disk of~$H$. 

\end{proof}


\subsection{Handle slides}
\label{sec27}
Handle slides in a handlebody are conveniently visualized
when `translated'
into a Whitehead diagram. Figure~\ref{fig01} shows the
effect of sliding the handle corresponding to the disk $x_2$ along the
handle corresponding to $x_1$.
\begin{figure}
\centering
\includegraphics[height=120mm]{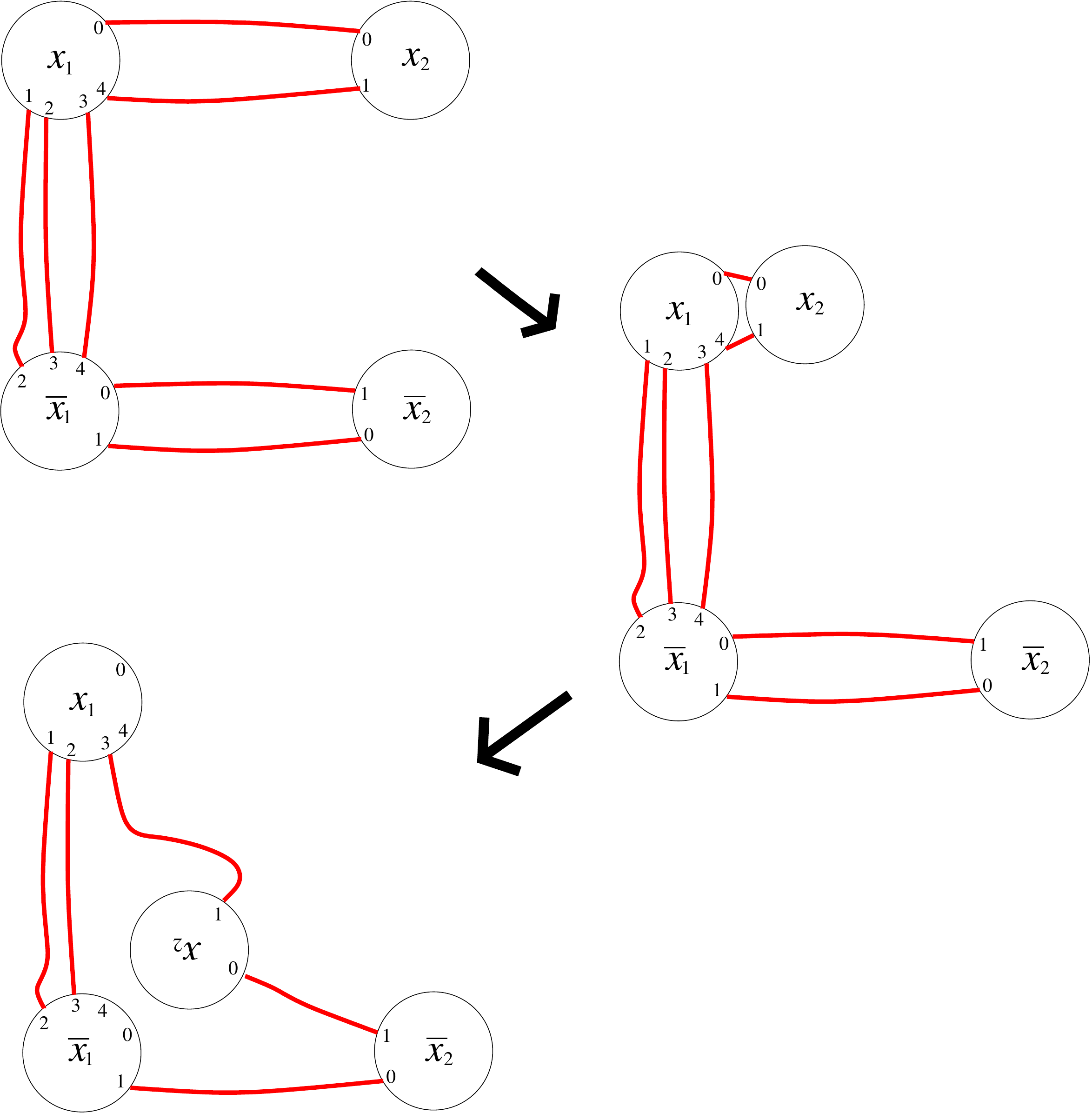}
\caption{A handle slide.}
\label{fig01}
\end{figure}
But, of course, in the final step, the meridional disks $x_1,\bar
x_1,x_2,\bar x_2$ in the drawing are no longer the same disks, but
are their images after the handle slide in the handlebody (The effect of
such a handle slide in the fundamental group of the handlebody is a
\emph{Whitehead automorphism}. See~\cite{stallings}).

\subsubsection{$\partial$--parallel arcs in handlebodies}
\label{sec271}

Let $k\subset S^3$ be a knot, and let $F\subset E(k)$ be a free Seifert
surface for~$k$. Also let $\Gamma$ be a spine for $F$ on $\partial
E(F)$. In Remark~\ref{remark23}~(2) a program is outlined   to
construct a circular decomposition for $E(k)$. It starts by drilling
some 2--handles out of~$E(F)$ disjoint with $F$. A
2--handle $P\subset E(F)$ is a product~$P=D^2\times I$ 
such that~$(D^2\times I)\cap \partial E(F)=D^2\times\{0,1\}$,
and it is determined by its `co-core' 
$\gamma=\{0\}\times I$. This co-core, $\gamma$, can be
visualized in $E(F)$  as a properly embedded arc with ends disjoint with
$\Gamma$.

Given two 
properly embedded arcs $\gamma$ and $\gamma'$ in $E(F)$ 
disjoint with $\Gamma$, if the triples $(E(F),\Gamma,\gamma)$ and
$(E(F),\Gamma,\gamma')$ are homeomorphic, then the pairs
$(E(\gamma),\Gamma)$ and $(E(\gamma'),\Gamma)$ are homeomorphic, and,
therefore, have homeomorphic handle decompositions. In this sense, we
say that $\gamma$ and $\gamma'$ \emph{induce homeomorphic handle
  decompositions} of $(E(F),\Gamma)$. Also we say, as an abuse of
language, that $\gamma$ and $\gamma'$ are \emph{equivalent 2--handles}.

Let $k$ be a
knot with $h(k)=1$, and let $F\subset E(k)$ be a free Seifert
surface for~$k$ which realizes a one-handled circular decomposition
of $E(F)$. 
Let $\gamma\subset E(F)$ be a properly embedded arc disjoint
with~$F\times\{0\}$. If the arc~$\gamma$ is the co-core of the single
2--handle of the 
one-handled circular decomposition
of $E(F)$,
then $\gamma$ is called
\emph{the arc of the handle decomposition}. Note that in this case, we know
that $\gamma$ is parallel into~$\partial E(F)$ (see
Corollary~\ref{coro43} below).

\subsubsection{Criterion for one-handledness}
We will establish a criterion to determine if an arc is the arc of some
one-handled decomposition.

Let $k$ be a
knot with $h(k)=1$, and let $F\subset E(k)$ be a free Seifert
surface for~$k$ which realizes a one-handled circular decomposition
of $E(F)$. 
Let $\gamma\subset E(F)$ be a $\partial$--parallel properly embedded
arc disjoint 
with~$F\times\{0\}$.

Consider  a system of meridional disks $x_1,\dots,x_g\subset E(F)$.
Let $z$
be a $\partial$--parallelism disk for $\gamma$.
After an isotopy of $E(F)$ which keeps $\Gamma$ fixed point-wise, we
may assume that $z$ is disjoint with the disks~$x_1,\dots,x_g$.
Then
$\gamma$ can be visualized 
in the Whitehead diagram of $(E(F),\Gamma)$, with respect
to~$x_1,\dots,x_g\subset E(F)$, 
as a properly embedded arc in $E(x_1\cup\cdots\cup x_g)$ disjoint
with~$G$, where $G$ is the corresponding Whitehead graph. 
After drilling out
the 2--handle, which is a regular neighbourhood of $\gamma$, we are
`adding a new handle' to~$E(F)$; that is, the exterior
$E(\gamma)\subset E(F)$ is homeomorphic to $E(F)$ plus one 1--handle.
We obtain
a new Whitehead diagram for $(E(\gamma), \Gamma)$ with respect to
$x_1,\dots,x_g,z$,
adding two fat vertices $z$ and $\bar z$ as in Figure~\ref{fig02}. 
\begin{figure}
\centering
\includegraphics[height=60mm]{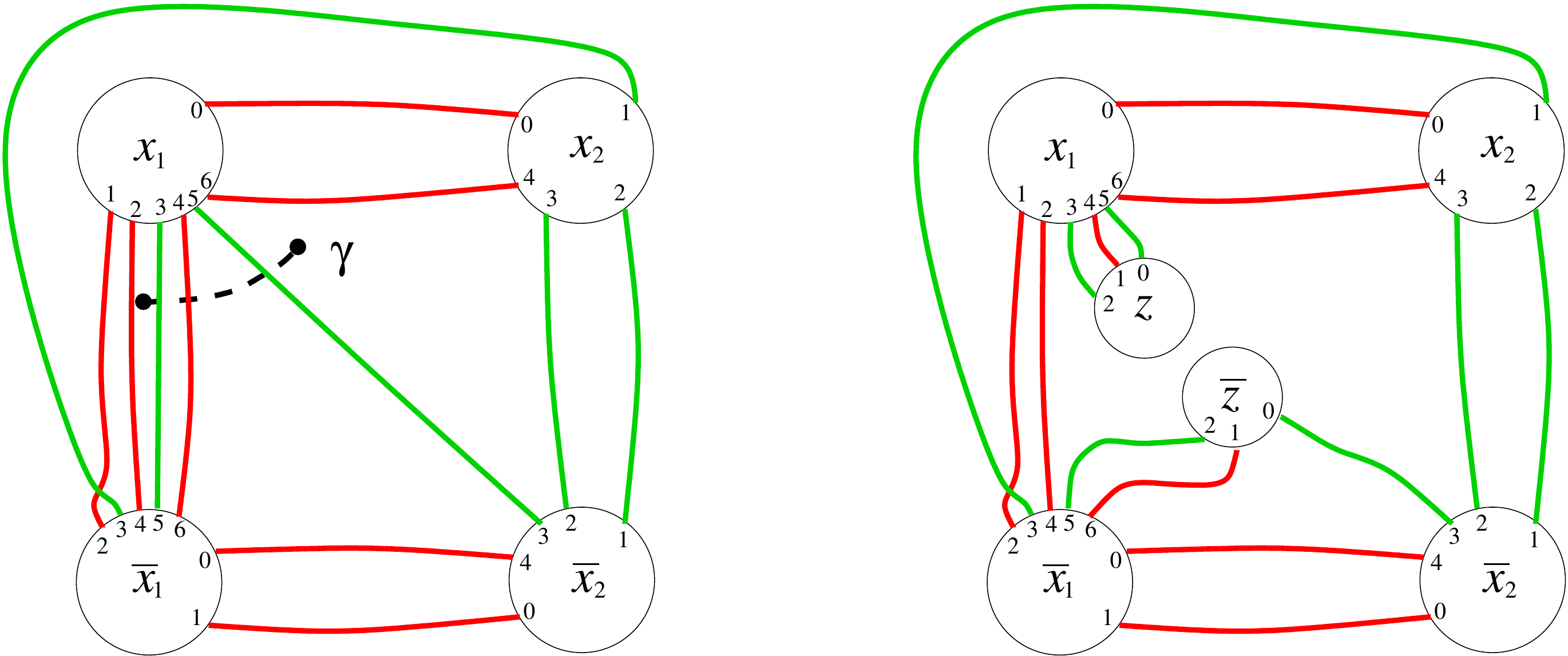}
\caption{Drilling out a 2--handle.}
\label{fig02}
\end{figure}

Define the complexity of a Whitehead graph as the sum of all valences
of the fat vertices of the graph.
The new Whitehead diagram obtained in the last paragraph may contain a
cut vertex 
$v$. For example, $v=x_1$ in Figure~\ref{fig02}. When there is a cut 
vertex $v$ in $G$, this vertex decomposes the graph $G$ 
into two non-trivial graphs $X_1$ and $X_2$. One of these graphs, say
$X_1$, does not contain $\bar v$. Then we can slide the
part corresponding to graph $X_1$ along the handle defined by disk~$v$.
If, after sliding, there appear cut vertices, we
continue sliding along some cut vertex on and on.
See Figures~\ref{regneigh} and~\ref{regneighslide}.
Since each such handle slide
lowers the complexity of the graph, eventually
we end up with, either:
\begin{enumerate}

\item \emph{A disconnected diagram}, or

\item \emph{A connected diagram with no cut vertices}.
  %
\end{enumerate}

In case (1),  see the last drawing of
  Figure~\ref{regneighslide}, there are obvious essential
  disks in $E(\gamma)$ disjoint with $\Gamma$ (more precisely, disjoint with
  the \emph{image} of $\Gamma$  on the diagram
  after the slides); the boundary of these
  essential disks are curves that separate the components of the
  current Whitehead
  graph. Assume a neighbourhood of one of these disks is a 1--handle~$B$ inside
  $E(\gamma)$ such that, after drilling out $B$,
  $E(\gamma\cup B)$  is a regular neighbourhood of
  $F=F\times\{0\}$. See the last drawing in 
  Figure~\ref{regneighslide} where the disk labeled $x_1$ corresponds
  to $B$. Then we have found a circular
  one-handled  decomposition of $E(k)$ based on $F$ according to the
  program outlined in 
   Remark~\ref{remark23}~(2), and $\gamma$ is the arc of this handle
   decomposition.  
Otherwise, we have to
  restart the program choosing a different arc to drill out.
\begin{figure}
\centering
\includegraphics[height=60mm]{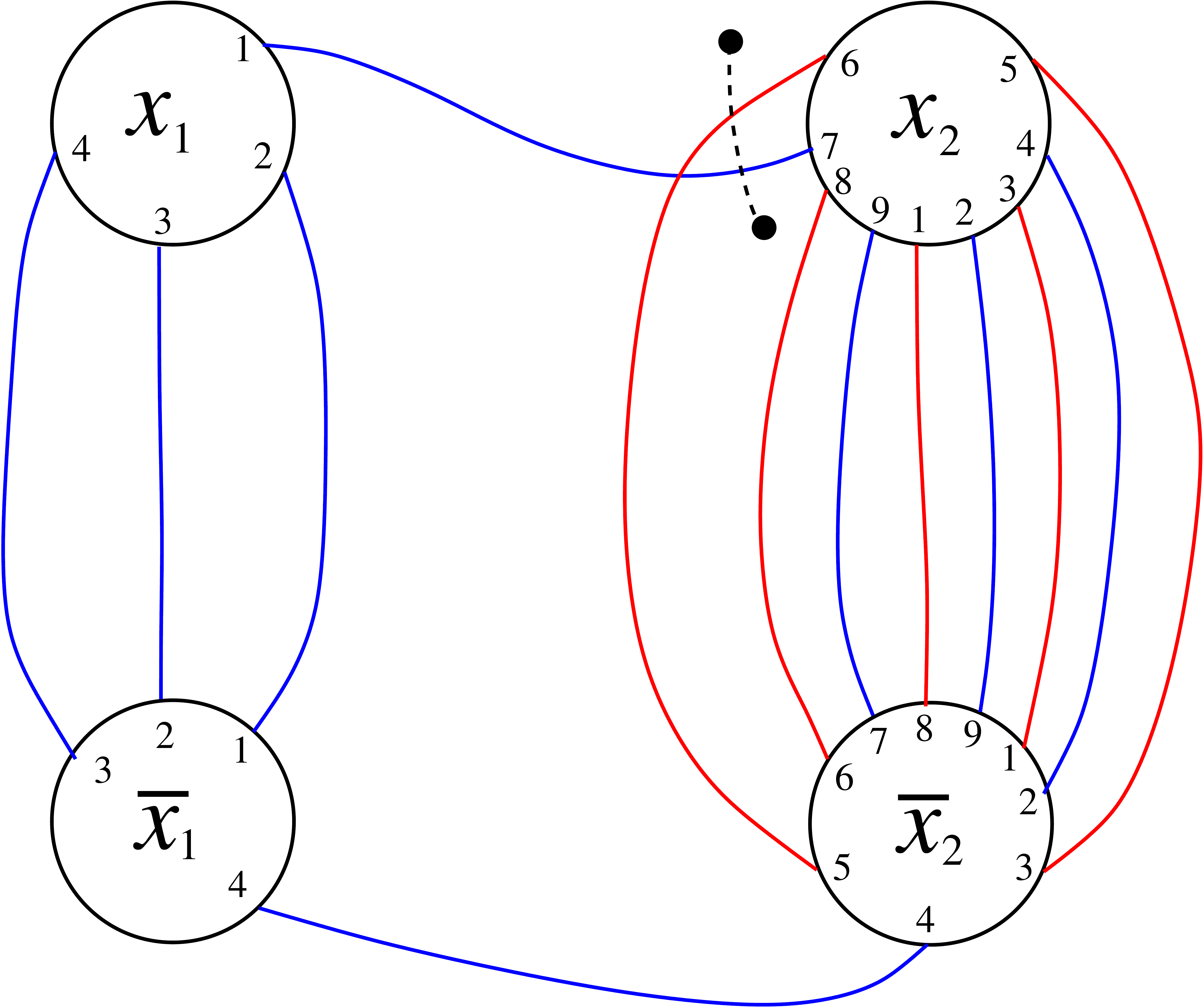}
\caption{}
\label{regneigh}
\end{figure}
\begin{figure}
\centering
\includegraphics[height=170mm]{example27slide.pdf}
\caption{}
\label{regneighslide}
\end{figure}

In Case~(2),
  by Corollary~\ref{coro29}, the chosen arc is not part of a one-handled
  circular decomposition.
Again, we have to
  restart the program choosing a different arc to drill out.

\subsubsection{Some definitions}
Now let $\gamma$ and $\gamma'$ be two $\partial$--parallel properly
embedded arcs in $E(F)$ disjoint with~$\Gamma$, with
$\partial$--parallelism disks $z$ and $z'$, respectively; let
$\{x_1,\dots,x_g\}$ be a meridional system of disks for $E(F)$, and let
$G$ be the corresponding Whitehead graph with respect to this system
of disks. Then,
by an isotopy of $E(F)$, we may assume that $z$ and $z'$ are
contained in $E(x_1\cup\cdots\cup x_g)$ and (the images of) $\gamma$ and
$\gamma'$ are disjoint with $G$. 

Assume that for two faces of $G$, that is, two connected components
$A,B\subset \partial E(x_1\cup\cdots\cup x_g)-G$, the face $A$ contains
an endpoint of $\gamma$ and one of $\gamma'$, and the face $B$ contains
the other two endpoints of $\gamma$ and $\gamma'$. Then there is an
isotopy of~$E(x_1\cup\cdots\cup x_g)$ that fixes point-wise $G$ and sends
$\gamma$ onto 
$\gamma'$. Such an isotopy exists for, being $\gamma$ and $\gamma'$
$\partial$--parallel, they are unknotted properly embedded arcs in the
3--ball $E(x_1\cup\cdots\cup x_g)$, and the isotopy can be chosen to fix $G$,
for the endpoints of the arcs are, by pairs, in components of $\partial
E(x_1\cup\dots\cup x_g)-G$. Then we see that a class of `equivalent'
2--handles in the Whitehead diagram of $(E(F),\Gamma)$ with respect
to~$x_1,\dots,x_g$ is determined by pairs of faces of $G$ in $\partial 
E(x_1\cup\cdots\cup x_g)$ (and conversely).
That is, for $\partial$--parallel properly embedded arcs
$\gamma,\gamma'\subset E(x_1\cup\cdots\cup x_g)$, the triples
$(E(x_1\cup\cdots\cup x_g),G,\gamma)$ and $(E(x_1\cup\cdots\cup
x_g),G,\gamma')$ are homeomorphic if and only if $\gamma$ and
$\gamma'$ connect the same pair of faces of $G$.

This is a very useful fact. To search for a one-handled decomposition,
one must only test a finite number of $\partial$--parallel arcs in
some Whitehead diagram, and analyze as above: there are as many
$\partial$--parallel arcs to check as pairs of faces
of the corresponding Whitehead graph.

We end this section with some definitions. 
Assume the arc~$\gamma$ is
boundary parallel into~$\partial E(F)$. 
Let $z$ be a $\partial$--parallelism disk for $\gamma$ such
that~$\partial z= \gamma \cup \gamma^B_z$, where~$\gamma^B_z$ is an
arc in $\partial E(F)$. 
Then, after a small isotopy of $z$, if necessary, $\gamma^B_z$
intersects the edges of $\Gamma$ transversely in a finite number of
points. If
$e_1,\dots,e_n$ are the edges of $\Gamma$ that intersect $\gamma^B_z$
and each $e_i$ intersects only once with $\gamma^B_z$, we say that
$\gamma$ \emph{encircles the edges} $e_1,\dots,e_n$. If $\gamma$ encircles the
edges $e_1,\dots,e_n$, and all $e_i$ are incident in the vertex $\xi$
of $\Gamma$, we say that the arc $\gamma$ \emph{is around the vertex $\xi$}. 
Notice that if $e_1,\dots,e_n,e_{n+1},\dots,e_{n+m}$ are all the edges
incident in the vertex $\xi$ of $\Gamma$, and~$\gamma$ is around
vertex $\xi$ encircling the edges $e_1,\dots,e_n$, then $\gamma$ also
encircles the edges~$e_{n+1},\dots,e_{n+m}$.  
The
\emph{length} of $\gamma$ in~$\Gamma$ is
the minimal number of  intersection points of $\gamma^B_z$ and
$\Gamma$ among all $\partial$--parallelism disks $z$ for $\gamma$. 

\section{Primitive elements in spines}
\label{sec3}

Let $\mathcal{F}$ be a free group. An element $x\in \mathcal{F}$ is called
\emph{primitive} if $x$ is part of some basis of $\mathcal{F}$. A set of
elements $x_1,x_2,\dots,x_k\in \mathcal{F}$ are called \emph{associated
  primitive elements} if they are contained in some basis of $\mathcal{F}$.

Let $H$ be a genus $g$ handlebody. 
A simple closed curve
$\alpha\subset H$ represents a primitive element in $\pi_1(H)$ if and
only if there is an essential properly embedded disk $D\subset H$ such
that $\alpha\cap D$ consists of a single point. A set of simple closed
curves $\alpha_1,\dots,\alpha_k\subset H$ represent a set of
associated primitive elements in $\pi_1(H)$ if and only if there is a system of
meridional disks $D_1,D_2,\dots,D_g\subset
H$ such that, up to renumbering, $\alpha_i\cap D_i$ consists of a
single point, and $\alpha_i\cap D_j=\emptyset$ for $i\neq
j$,~$i=1,\dots,k$, and $j=1,\dots,g$.

\begin{theorem}
Let $k\subset S^3$ be a knot, and let $F\subset E(k)$ be a free
Seifert surface for $k$. Assume $E(F)$ is a handlebody of genus $g$.  

If there exists a graph $\Gamma=a_1\vee\cdots\vee
a_g$ such that $\Gamma$ is
a spine for $F$ on $\partial E(F)$, and the $\ell$ curves $a_1,\dots,a_\ell$
represent associated 
primitive elements of $\pi_1(E(F))$,
then the handle number $h(F)\leq g-\ell$.
\label{thm31}

\end{theorem}

\begin{proof} 
We follow the plan in
  Remark~\ref{remark23}~(\ref{dos}): we will exhibit a system of
  properly embedded arcs (the arcs $\beta_j^I$, below) which are the
  co-cores of $g-\ell$ 2--handles to be drilled out of $E(F)$, and
  a system of $g-\ell$ 2--disks ($D_{\ell+1},\dots, D_g$, below) which
  define the co-cores of $g-\ell$ 1--handles to be drilled out of
  $E(F\cup \bigcup_j \beta_j^I)$

Let $D_1,D_2,\dots,D_g\subset E(F)$ be a system of meridional disks
for $E(F)$ such that~$|a_i\cap D_i|=1$, and $a_i\cap D_j=\emptyset$
for $i\neq j$, $i=1,\dots,\ell$, and $j=1,\dots,g$.
This system of meridional disks exists, since $a_1,\dots,a_\ell$
represent associated 
primitive elements of~$\pi_1(E(F))$. 

Let $P\subset E(F)$ be a regular
neighbourhood of the base point $x_0\in\partial E(F)$ ($x_0$ is also
the base point of the graph $\Gamma$). We visualize
$P$ as a $2g$--gonal prism. See Figure~\ref{fig1}. For $i=1,\dots,g$,
let $T_i$ be a regular neighbourhood of $a_i$ in $E(F)$ such that $T_i\cap
T_j=P$ if $i\neq j$.
Write $\widehat T_i=\overline{T_i-P}$; then $\widehat T_i$ is a
3--ball. The intersection,~$\widehat T_i\cap P=d_i^+\cup d_i^-$, is the
disjoint union of two 2--disks 
$d_i^+$ and $d_i^-$ (see Figure~\ref{fig1}). Also write $\partial
d_i^+=\beta_i^B\cup\beta_i^I$ where $\beta_i^B$ is an arc in $\partial
E(F)$, and $\beta_i^I$ is a properly embedded arc in $E(F)$. Finally,
write $A_i=\overline{\partial T_i -(d_i^+\cup d_i^-\cup\partial
  E(F))}$ which is a 2--disk.

The arcs $\beta^I_{\ell+1},\dots,\beta^I_{g}$ are the co-cores of
2--handles in $E(F)$ to be drilled out, according to the plan in
Remark~\ref{remark23}~(2): 

Notice that the exterior of each $\beta_i^I$,
$E(\beta_i^I)=\overline{E(F)-\mathcal{N}(\beta_i^I)}\cong\overline{ 
  E(F)-\mathcal{N}(A_i)}$, and this homeomorphism is the identity map
outside a small neighbourhood of $A_i$.

Consider $V=\overline{E(F)-(\widehat T_{\ell+1}\cup\widehat
  T_{\ell+2}\cup\cdots\cup\widehat T_g)}$. Then $V$ is a genus $g$
handlebody and $E(F)$ is a regular neighbourhood of $V$. We see that
$$\overline{E(F)-\cup_{\ell+1}^g
  \mathcal{N}(\beta_i^I)}\cong\overline{E(F)-\cup_{\ell+1}^g
  \mathcal{N}(A_i)}\cong V\cup (g-\ell \text{ 1--handles })$$
where the $g-\ell$ 1--handles are the $g-\ell$ balls $\widehat T_i$
attached along the disks $d_i^+$,~$d_i^-$,~$i=\ell+1,\dots,g$.

By the choice of the disks $\{D_i\}$, we see that
$\overline{V-\cup_{\ell+1}^g\mathcal{N}(D_i\cap V)}$ is a regular
neighbourhood of $a_1\vee\cdots\vee a_\ell$. Then 
$\overline{E(F)-(\cup_{\ell+1}^g\mathcal{N}(\beta_i^I)+\cup_{\ell+1}^g\mathcal{N}(D_i\cap V))}$
is a regular neighbourhood of $\Gamma$. In other words,
$\mathcal{N}(F)\cup\{\mathcal{N}(\beta_i^I)|i=\ell+1,\dots,g\}\cup\{\mathcal{N}(D_i\cap
V)|i=\ell+1,\dots,g\}$ determines a circular handle decomposition of $E(k)$
based on~$F$, as in Remark~\ref{remark23}~(\ref{dos}). 
Therefore,~$h(F)\leq g-\ell$.

\end{proof}
\begin{figure}
\centering
\includegraphics[height=60mm]{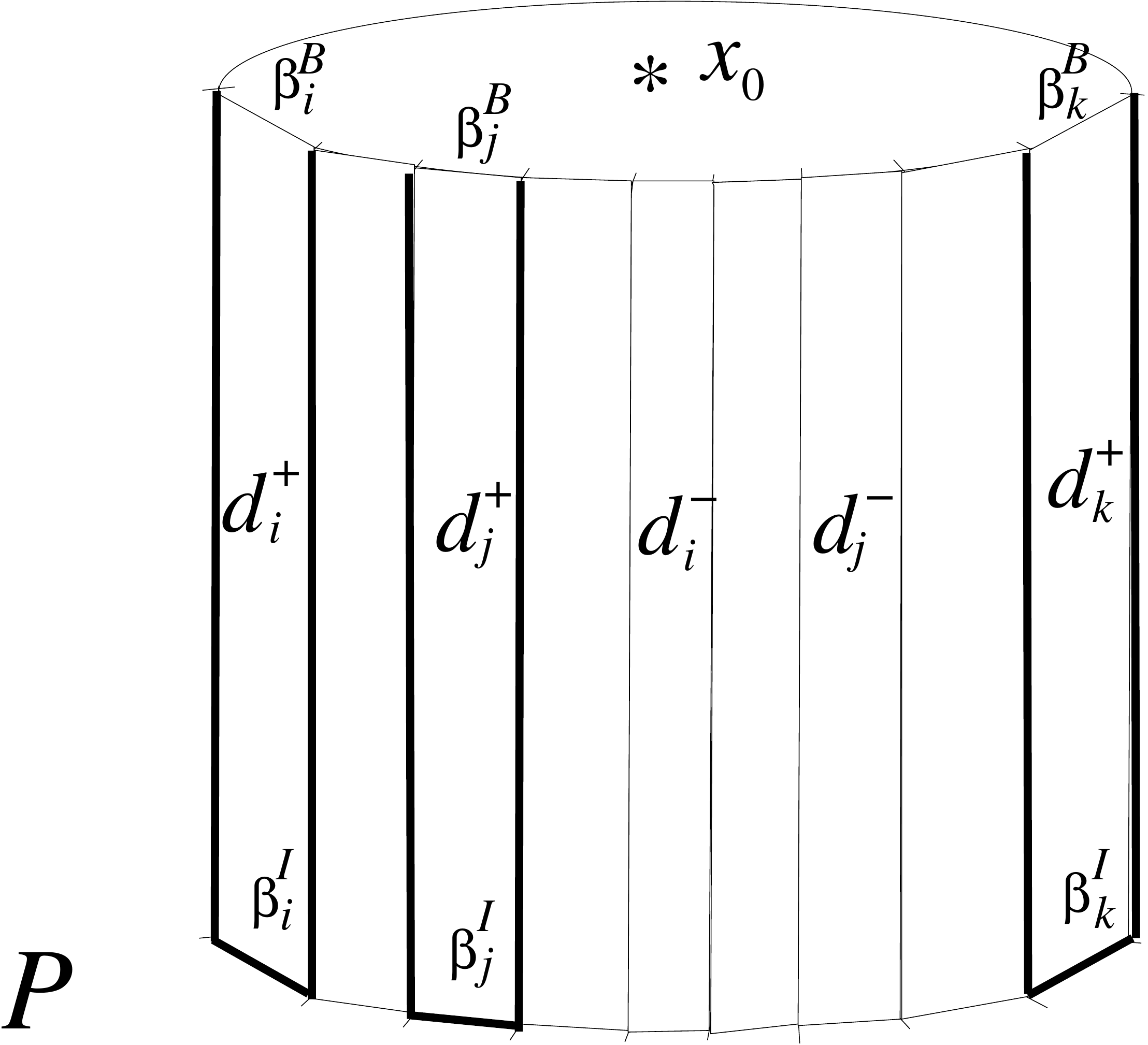}
\caption{The neighbourhood of $x_0$.}
\label{fig1}
\end{figure}

\begin{corollary}[The case ``$\ell=g$'']

Let $k\subset S^3$ be a knot, and let $F$ be a free
Seifert surface for $k$. Assume $E(F)$ is a handlebody of genus $g$.  

If there exists a graph $\Gamma=a_1\vee a_2\vee\cdots\vee
a_g$ such that $\Gamma$ is
a spine for $F$ on~$\partial E(F)$, and the curves $a_1,\dots,a_g$
form a basis of 
$\pi_1(E(F))$, then $k$ is a fibered knot with fiber $F$.
\label{coro33}

\end{corollary}

\begin{proof}
In this case $h(F)=0$, therefore, $E(F)$ admits a product structure
$E(F)=F\times I$
induced by $\Gamma$, and $k$ is fibered with fiber $F$.
\end{proof}

\begin{corollary}[The case ``$\ell=0$'']

Let $k\subset S^3$ be a knot, and let $F\subset E(k)$ be a free
Seifert surface for $k$. Assume $E(F)$ is a handlebody of genus $g$.  

Then $h(k)\leq g$.
\label{coro34}

\end{corollary}

\begin{proof}
By Theorem~\ref{thm31}, considering $\ell=0$, we have $h(F)\leq g$.
Therefore, $h(k)\leq g$. 
\end{proof}

\begin{remark}
Corollary~\ref{coro34} asserts that for a
connected knot $k$, $h(k)\leq 2g_f(k)$. See~\cite{rurru} for
another proof of this fact (a fact called the `Free
Genus Estimate'  in~\cite{rurru}). 
\end{remark}

\begin{corollary}
If $k$ is a connected free genus one knot, then $h(k)=0,1$, or $2$.
\label{coro36}
\end{corollary}
\hfill $\Box$

\begin{remark} 
\label{remark37}
Let $k$ be a connected free genus one knot in~$S^3$ such that $k$ is not
fibered (that is, $k\neq3_1,4_1$). At this point we can give some
estimates for $cw(k)$:

If $k$ is almost fibered, it
follows from Corollary~\ref{coro36} that $cw(k)=4$, or
$cw(k)=6$. In any case, that is, if $k$ is almost fibered or not,
$cw(k)\leq6$.

If $k$ is not almost fibered, consider a circular decomposition
$E(k)=(F\times I)\cup B_1\cup P_1\cup B_2\cup P_2\cup\cdots\cup 
B_n\cup P_n$, $n>1$, and $B_i,P_i\neq\emptyset$, which realizes
$cw(k)$. Then there are Seifert surfaces for $k$,
$T_1,\dots,T_n=F,S_1,\dots,S_n\subset E(k)$,
such that~$S_i$ is obtained from $T_{i-1}$ by adding the
1-handles $B_i$, and $T_i$ is obtained from~$S_i$ by adding the
2-handles $P_i$, and
$cw(k)=(c(S_{\sigma(1)}),\dots,c(S_{\sigma(n)}))$
with~$c(S_{\sigma(1)})\geq\cdots\geq c(S_{\sigma(n)})$ where
$c(S)=1-\chi(S)$.  

Now all $T_i$ are incompressible (Remark~\ref{remark24}), and of genus one. For, if some~$T_j$ is of genus at least two, then $S_j$ is
of genus at least three, and the complexity~$c(S_{j})\geq6$. But then, since $n>1$,
$cw(k)=(c(S_{\sigma(1)}),\dots,c(S_{\sigma(n)}))>6$, a
contradiction. It follows that $cw(k)=(4,\dots,4)$.

That is, \emph{if $k$ is a connected non-fibered free genus one knot,
  then $cw(k)=4,6$, or $(4,\dots,4)$.}

\end{remark}

As it was mentioned in the Introduction, a connected non-fibered free genus one
knot in~$S^3$ is almost
fibered (Theorem~\ref{thm67} below). It follows that $cw(k)\in\{4,6\}$.

\begin{example}\textbf{Rational knots.} 
\label{ex38}
\emph{If $k\subset S^3$ is a non-fibered rational knot,
  then~$h(k)=1$. Also $cw(k)=4g(k)$ if $k$ is connected, and
  $cw(k)=4g(k)+1$ otherwise.}

Let $k\subset S^3$ be a
  rational knot. Then $k$ is encoded with a continued fraction of the form
  $[2b_1,2b_2,\dots,2b_g]$ where $g$ is even or odd if $k$ is connected
  or not, respectively. 
Here $b_1,\dots,b_g$ are non-zero integers .
  Now $k$ has a minimal genus Seifert surface~$F$ as
  in Figure~\ref{figRat1} (see \cite{gabai}, Answer~1.19).  This
  surface is free. Note that 
  $g(F)=g/2$ if $k$ is connected, and~$g(F)=(g-1)/2$ otherwise.
\begin{figure}
\centering
\includegraphics[height=50mm]{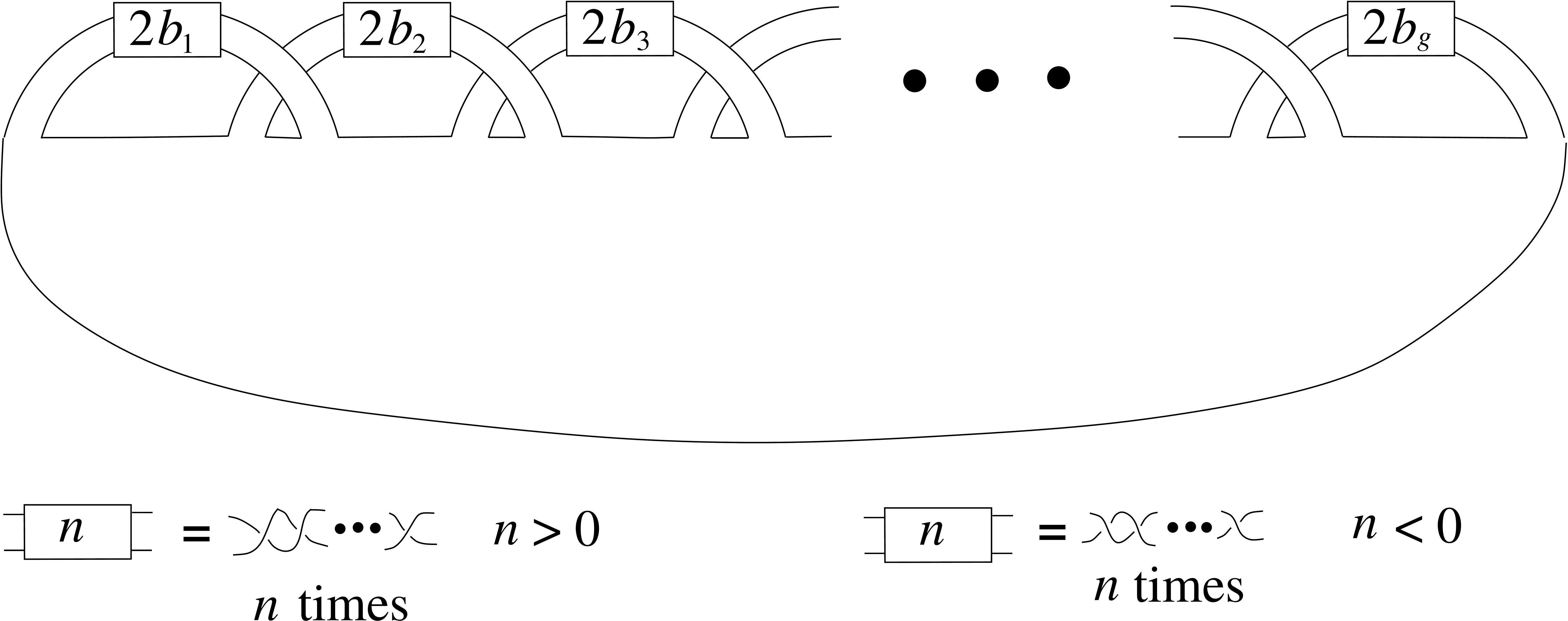}
\caption{A minimal genus Seifert surface for the knot $k=[2b_1,2b_2,\dots,2b_n]$.}
\label{figRat1}
\end{figure}

In a neighbourhood $V$ of this surface we can find a
spine~$\Gamma\subset F\times\{0\}\subset\partial V$ with~$\Gamma=a_1\vee 
a_2\vee\cdots\vee a_g$, as in
Figure~\ref{figRat2}. 
\begin{figure}
\centering
\includegraphics[height=50mm]{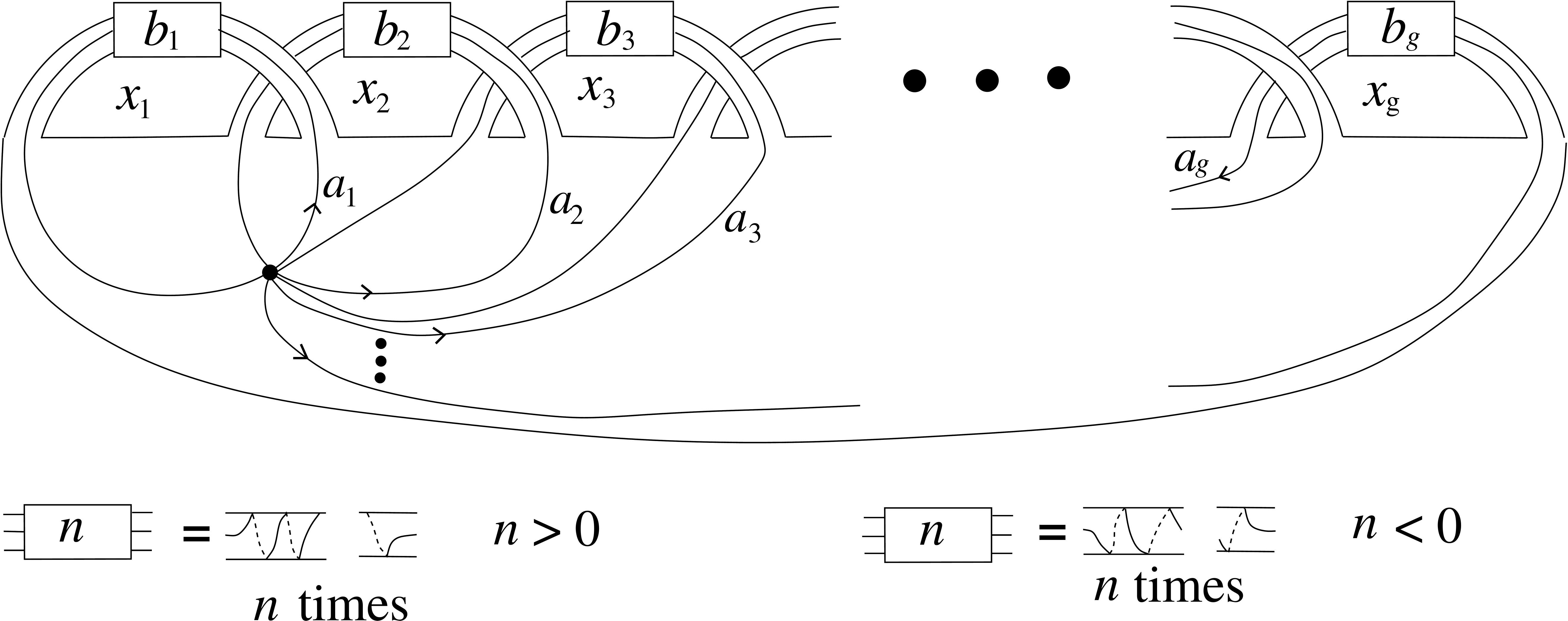}
\caption{A spine for $k=[2b_1,2b_2,\dots,2b_g]$ in $\partial\mathcal{N}(F)$.}
\label{figRat2}
\end{figure}
For the obvious meridional disks,
$x_1,x_2,\dots,x_g$, of the handlebody $E(F)$, corresponding to a basis
$\{x_1,x_2,\dots,x_g\}$ of~$\pi_1(E(F))$, the curves
$a_1,a_2,\dots,a_g$ represent the elements
$x_1^{b_1}$, $x_2^{b_2}x_1$, $x_3^{b_3}x_2\dots$,
$x_{g-1}^{b_{g-1}}x_{g-2}$, $x_g^{b_g}x_{g-1}$
of $\pi_1(E(F))$, respectively. 

If each $|b_i|=1$, then
$a_1,a_2,\dots,a_g$ represent a basis of~$\pi_1(E(F))$, and, by
Corollary~\ref{coro33}, $k$ is fibered with fiber $F$.

If some $|b_i|\geq2$, then
$\{x_g,x_2^{b_2}x_1,x_3^{b_3}x_2\dots,x_{g-1}^{b_{g-1}}x_{g-2},x_g^{b_g}x_{g-1}\}$
is a basis for $\pi_1(E(F))$; it follows that the curves
$a_2,a_3,\dots,a_g\subset \Gamma$
represent associated primitive elements of $\pi_1(E(F))$, and, by
Theorem~\ref{thm31}, $h(k)\leq h(F)=1$.
By the second part of the statement of Answer~1.19 of~\cite{gabai}, $k$
is not fibered. Therefore, $0<h(k)=h(F)=1$, and $cw(k)=2g$ if $k$ is
connected, and $cw(k)=2g+1$ otherwise.

\begin{remark}
In Theorem~3.21 of \cite{lagoda} it is claimed that the 
result in Example~\ref{ex38}, the one-handledness of rational knots, is
known, but unpublished. 
\end{remark}

\end{example}

\begin{example} 
\label{exPretzel}
  \textbf{Pretzel knots.} 
  \emph{The pretzel knot $k=P(\pm3,q,r)$ with $|q|,|r|$ odd integers
    $\geq3$, has $h(k)=1$ and, therefore, $cw(k)=4$.} 

  Let $k$ be the pretzel knot
  $P(p,q,r)$ with $p,q,r$ odd integers. Then $k$ is a connected knot,
  and the `black surface' $F$ of a standard projection of $k$ is a free
  genus one Seifert surface for $k$. See Figure~\ref{fig31teta}.
\begin{figure}
\centering
\includegraphics[height=60mm]{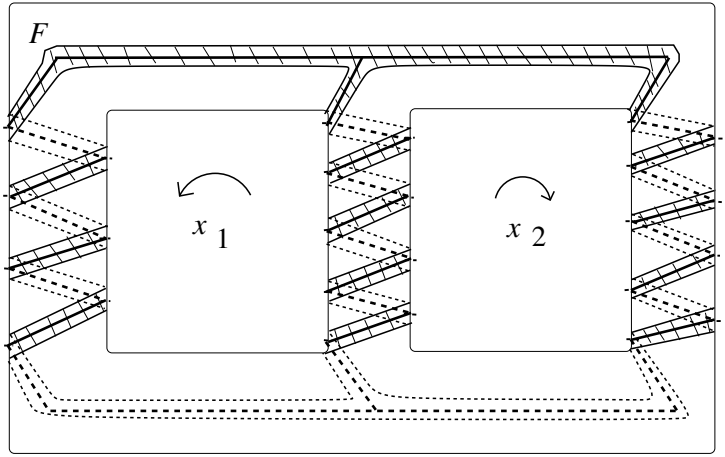}
\caption{Black surface  for $P(7,9,9)$.}
\label{fig31teta}
\end{figure}
If $|p|,|q|,|r|\geq3$, it is known that
(1)~$k$ has a unique incompressible Seifert surface (see
\cite{godaishi}), namely, the free black surface $F$
of genus one;
(2)~$k$ has tunnel number two (see \cite{saku-klimen}); 
(3)~$h(k)\leq2$ (see Corollary~\ref{coro36});
(4)~since~$t(k)\neq1$,~$k$ is not a
  rational knot;
(5)~ also $k$
  is not fibered (that is, $k\neq3_1,4_1$).

For any
  permutation $s,t,u$ of  
  $p,q,r$, the pair $(S^3, k)$ is homeomorphic to a pair
  $(S^3, \ell)$ where $\ell$ is a pretzel knot $P(s,t,u)$. 
Also, by a reflection, $P(p,q,r)$ is equivalent to $P(-p,-q,-r)$.
Then, by Remark~\ref{remark22}, we may assume that it holds either,
  \emph{Case~1}: ``$p,q,r>0$'', or \emph{Case~2}:~``$p<0$ and $q,r>0$''.

  There is  a spine shown in Figure~\ref{fig31teta} for the surface
  $F\times\{0\}\subset\partial\mathcal{N}(F)$. This spine is a
  $\theta$--graph. To obtain a wedge of circles as a spine
$\Gamma=a_1\vee a_2\subset F\times\{0\}\subset\partial\mathcal{N}(F)$,
 we slide the middle edge of the
$\theta$--graph to the left. In \emph{Case~1},~``$p,q,r>0$'', we obtain
the upper part of Figure~\ref{fig31}; 
and, in \emph{Case 2},~``$p<0$ and $q,r>0$'', after using an isotopy to avoid unnecessary
intersections of the curve 
$a_2$ with the disk $x_1$, we obtain the lower part of
Figure~\ref{fig31}.
\begin{figure}
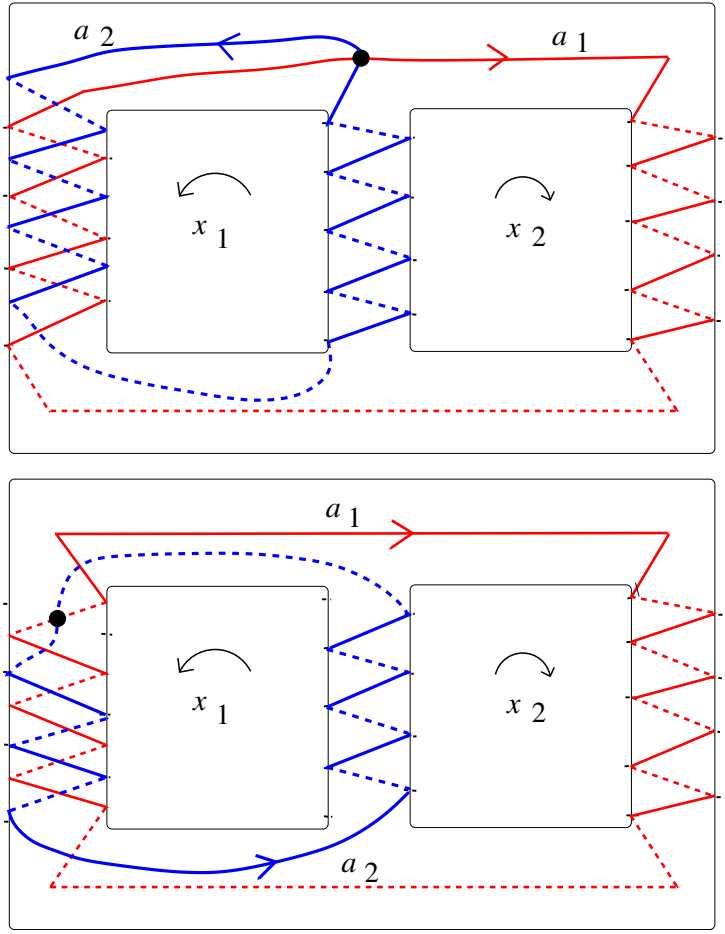

\centering
\includegraphics[height=60mm]{espinapos.pdf}
\vskip3true mm
\includegraphics[height=60mm]{espinaneg.pdf}
\caption{Spines for $P(p,q,r)$.}
\label{fig31}
\end{figure}
We see that, writing $\pi_1(E(F))\cong\langle x_1,x_2:-\rangle$:

  \emph{Case 1}, ($p,q,r>0$),
  the curves $a_1$ and $a_2$ 
  represent the elements
$x_2^{(r+1)/2}x_1^{-(p-1)/2}$ and $x_1^{(p+1)/2}(x_2x_1)^{(q-1)/2}$,
  respectively, in $\pi_1(E(F))$, or, 

  \emph{Case 2}, ($p<0$ and $q,r>0$), the curves $a_1$ and $a_2$ 
  represent the ele\-ments $x_2^{(r+1)/2}x_1^{(|p|+1)/2}$ and
 $x_1^{-(|p|-3)/2}(x_2x_1)^{(q-3)/2}x_2$,
res\-pec\-ti\-ve\-ly,
  in $\pi_1(E(F))$.

  Assume the number $3\in \{|p|,q,r\}$. 

  In \emph{Case~1},~``$p,q,r>0$'',
  using a homeomorphism of $S^3$, we 
  may assume $p=3$. In this case the curve $a_1\simeq
  x_2^{(r+1)/2}x_1^{-1}$ represents a primitive element of
  $\pi_1(E(F))$ for, the set $\{x_2^{(r+1)/2}x_1^{-1}, x_2\}$ is a
  basis of $\pi_1(E(F))$. Therefore, by Theorem~\ref{thm31},~$h(k)=h(F)=1$, and
  $cw(k)=4$. 

  In \emph{Case~2},~``$p<0$ and $q,r>0$'', if
  $p=-3$, then the curve $a_2\simeq (x_2x_1)^{(q-3)/2}x_2$
  represents a primitive element of  $\pi_1(E(F))$ for, the set
  $\{(x_2x_1)^{(q-3)/2}x_2, x_2x_1\}$ is a basis of
  $\pi_1(E(F))$. If $q=3$ or $r=3$, we may assume that
  $q=3$, and then the curve $a_2\simeq x_1^{(|p|-3)/2}x_2$
  represents a primitive element of  $\pi_1(E(F))$ for, the
  set~$\{x_1^{-(|p|-3)/2}x_2, x_1\}$ is a basis of
  $\pi_1(E(F))$.  

  In
  both cases, $p=-3$, or $q$ or $r=3$, we conclude by Theorem~\ref{thm31},
  $h(k)=h(F)=1$, and  $cw(k)=4$.

\end{example}

\begin{remark}
\label{remark311}
If $|q|, |r|$ are odd integers $\geq 3$, 
then $k=P(\pm3, q, r)$ has
tunnel number two.  
Then the family of pretzel knots $\{P(\pm3, q, r) : |q|,|r|\text{ odd
  integers }\geq3 \}$ is
a family of examples of non-fibered knots~$k$
for which the 
strict inequality $h(k)< t(k)$ holds (compare with \cite{pajilon},
where it is proved that $h(k)\leq t(k)$). 
\end{remark}

\section{Pretzel knots: the case $|p|,|q|,|r|\geq5$}
\label{sec4}

In this section we show:

\begin{theorem}
\label{thm41}
The free genus one Seifert surface for a pretzel knot $P(p,q,r)$ with
$|p|, |q|, |r| \geq5$ has handle number two.
\end{theorem}

As noted in Example~\ref{exPretzel}, when dealing with the pretzel
knot $k=P(p,q,r)$ we may assume:
\emph{Case 1:} ``$p,q,r >0$'', or
\emph{Case 2:} ``$p<0$ and $q, r >0$''.

\subsection{Handle decompositions of $E(P(p,q,r))$}

\begin{lemma}
\label{lema42}
Let $V$ be a handlebody and let $\alpha \subset V$ be a properly embedded
arc. If the exterior~$E(\alpha)\subset V$ is a handlebody, then
$\alpha$ is parallel into $\partial V$. 
\end{lemma}
\begin{proof}
By hypothesis, $\pi_1(E(\alpha))$ is a finitely generated free group.
If $\mathcal{N}(\alpha)= D^2\times I$ is a regular
neighbourhood of $\alpha$ in $V$, let $\mu=\partial D^2\times\{1/2\}$ be
a meridian of~$\mathcal{N}(\alpha)$. If~$N\langle\mu\rangle$ denotes
the normal closure of the element represented by $\mu$ 
in $\pi_1(E(\alpha))$, then
$\pi_1(E(\alpha))/N\langle\mu\rangle$ is isomorphic to the fundamental
group of the space obtained from $E(\alpha)$ by adding a 2--handle
along $\mu$. Then
$\pi_1(E(\alpha))/N\langle\mu\rangle\cong\pi_1(V)$
is a free group. It follows that $\mu$ represents a primitive element
in $\pi_1(E(\alpha))$ (see~\cite{whitehead}, Theorem~4). Thus, there is an
essential disk $\delta\subset E(\alpha)$  such that the number of
points $\#(\delta\cap
\mu)=1$. After an isotopy, we may assume that $\partial \delta
\cap \partial N(\alpha)=\gamma$ is an arc, and $\partial \delta= \beta \cup
\gamma$ where $\beta$ is an arc  contained in $\partial V$. 

There is a product 2--disk $Z=(\textrm{radius of } D^2)\times I$ between $\gamma$ and $\alpha$, with
$Z\subset \mathcal{N}(\alpha)$ for some product structure $D^2\times
I$ of $\mathcal{N}(\alpha)$. Then~$\delta$ can be extended to a disk 
$\delta'=Z\cup \delta$ 
whose boundary is a union of arcs $\alpha \cup \beta'$ with~$\beta'
\subset \partial V$ (and $\beta\subset\beta'$). Therefore, $\alpha$ is parallel into
$\partial V$.
\end{proof}

\begin{corollary}
\label{coro43}
Let $F$ be a free 
Seifert surface for a knot
$k$. Suppose~$F$ has handle number one and let $\alpha$ be the core of
the 
1--handle of a one-handled circular decomposition of $E(k)$ based on
$F$. Then $\alpha$ is parallel into 
$\partial E(F)$. 
\end{corollary}
\begin{proof} 
As in Remark~\ref{remark23}~(\ref{dos}), the one-handled decomposition
of the pair $(E(F),F)$ is constructed by, first, drilling a 2--handle out of
$E(F)$ disjoint with, say, $F\times\{1\}$.
This 2--handle has as co-core the arc~$\alpha$ of the statement (cf.
Remark~\ref{sec271}). After drilling out~$\alpha$, we, secondly,
drill one  
1--handle $B$ out of the exterior $E(\alpha)\subset E(F)$ with $B$ disjoint
with $F\times\{1\}$.
The result of this drilling is a regular
neighbourhood of the surface $F \times \{0\}$ in $E(k)$ which is a
handlebody. Therefore, the exterior $E(\alpha)$ in~$E(F)$ is the union
of the
neighbourhood of $F \times \{0\}$ and the 1--handle~$B$; that is,~$E(\alpha)$ is a
handlebody. By Lemma~\ref{lema42} we conclude that $\alpha$ is parallel
into $\partial E(F)$.  

\end{proof}

\begin{proof}[Proof of Theorem~\ref{thm41}]
Let $F$ be the free genus one Seifert surface for $k=P(p,q,r)$ with
$|p|, |q|, |r|$ odd integers $\geq5$. 

For the 
sake of contradiction, we assume that $F$ has handle number one. By
Corollary~\ref{coro43}, the core $\gamma$ of the 1--handle of the
circular decomposition of
$E(k)$ based on $F$ is
parallel into $\partial E(F)$. By assumption, there is also a 2--handle
$B\cong I\times D^2$ that completes the decomposition, such that the exterior
$E(\gamma\cup B)\subset E(\gamma)$ is a
regular neighbourhood of $F$ in $E(k)$, and
$\partial B$ is disjoint with $F$. In particular the core,
$\{1/2\}\times D^2$, of
$B$ is an essential disk in $E(\gamma)$ disjoint with $F$.
\begin{figure}[htp]
\centering
\includegraphics[width=12true cm]{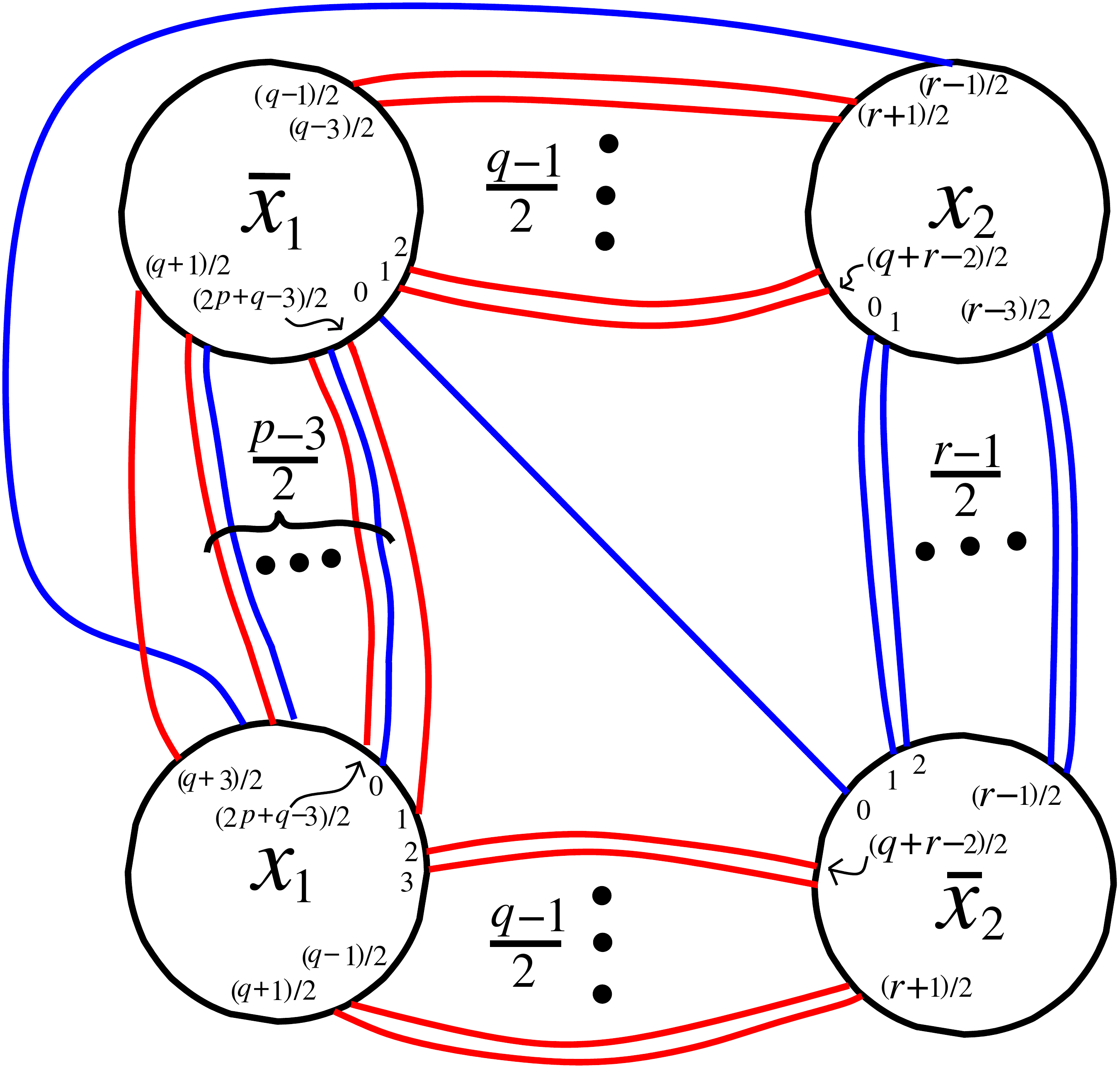}
\caption{}
\label{Ppqr1}
\end{figure}
We will show that any essential disk in $E(\gamma)$ intersects~$F$,
obtaining the desired contradiction.

\emph{Case 1:} ``$p,q,r >0$''.
Let $\Gamma= a_1\vee a_2 $ be the spine for $F$ given in
Example~\ref{exPretzel}. By Remark~\ref{remark26}, we only need to analyze the
handle decompositions of $(E(F),\Gamma)$. There is an obvious system of
meridional disks 
$x_1,x_2\subset E(F)$ as depicted in the upper part of
Figure~\ref{fig31}. The Whitehead 
diagram for $(E(F),\Gamma)$ with respect to~$x_1,x_2$ looks like
Figure~\ref{Ppqr1}. 

In the corresponding
Whitehead graph $G$ we see:

\begin{itemize}
\item  Four fat vertices corresponding to the
  meridional disks $x_1$ and $x_2$. 

\item There are $(q-1)/2$ horizontal edges connecting $\bar x_1$ and 
  $x_2$, and $(q-1)/2$ horizontal edges connecting $x_1$ and $\bar x_2$;
  all these horizontal arcs belong to the curve $a_2$.

\item There are $(r-1)/2$ vertical edges connecting $x_2$ and $\bar x_2$;
  one diagonal edge connecting $x_1$ and $x_2$, and one diagonal edge
  connecting $\bar x_1$ and $\bar x_2$; all these vertical and diagonal
  edges belong to the curve $a_1$. 

\item Finally, connecting $x_1$ with 
  $\bar x_1$, we find, going from right to left in Figure~\ref{Ppqr1},
  first an arc
  belonging to $a_2$, and then we find
$(p-3)/2$
pairs of arcs belonging
  consecutively to $a_1$ and $a_2$; and a last arc belonging to $a_2$
  which crosses with the diagonal arc from $x_1$ to $x_2$ on the base
  point of $\Gamma$.

\end{itemize}

\textbf{Claim 0:} Let $z$ be a $\partial$--parallelism disk for the arc $\gamma$ in
$E(F)$. 
Then the disk $z$ contains at
least one point of $a_1$ and one point of $a_2$.
\begin{proof}
Let $G_i$ be the Whitehead graph
of the pair $(E(F),a_i)$ with respect to $x_1,x_2$ ($i=1,2$). See
Figure~\ref{a1a2}. After sliding the handle defined by the disk $x_2$
along the handle defined by $\bar x_1$ on the right side of Figure~\ref{a1a2}, 
the image of the graph~$G_2$ looks like Figure~\ref{Ppqr4}. Since
these graphs are connected and contain no cut vertex, it follows from
Corollary~\ref{coro29} 
 that
any essential disk in $E(F)$ intersects $a_i$ ($i=1,2$). 
Now, the exterior
$E(\gamma)$ can be regarded as a copy of $E(F)$ plus one 1--handle defined by the
disk $z$. Assume $z\cap a_2\neq\emptyset$. If $z\cap a_1=\emptyset$,
then~$a_1$ is contained in the copy of~$E(F)\subset E(\gamma)$. By hypothesis there is
an essential disk $\Delta\subset 
E(\gamma)$ such that~$\Delta\cap(a_1\cup a_2)=\emptyset$. Now,
$\Delta\cap z\neq\emptyset$, otherwise $\Delta$ is a subset of the
copy of~$E(F)\subset E(\gamma)$ missing the extra 1--handle, and
$\Delta\cap a_1=\emptyset$, contradicting that 
any essential disk in~$E(F)$ intersects~$a_1$.
Through isotopies, we may assume
that $\Delta\cap z$ is a set of disjoint arcs. 
Then 
the intersection of $\Delta$ with the copy of~$E(F)\subset E(\gamma)$,
that is, the set~$\Delta\cap\overline{(E(\gamma)-\mathcal{N}(z))}$,
is a set of disjoint properly embedded
disks~$\Delta_1,\dots,\Delta_n\subset E(F)$. Since $\Delta$ is not
parallel to $z$ in $E(\gamma)$, at least one $\Delta_i$
is 
essential in $E(F)$, otherwise $\Delta$ would be parallel
into~$\partial E(\gamma)$. We 
obtain again an essential disk in $E(F)$ disjoint with $a_1$, which is
a contradiction as above, and, therefore, $z\cap
a_1\neq\emptyset$. 
\end{proof}
\begin{figure}[htp]
\centering
\includegraphics[width=12true cm]{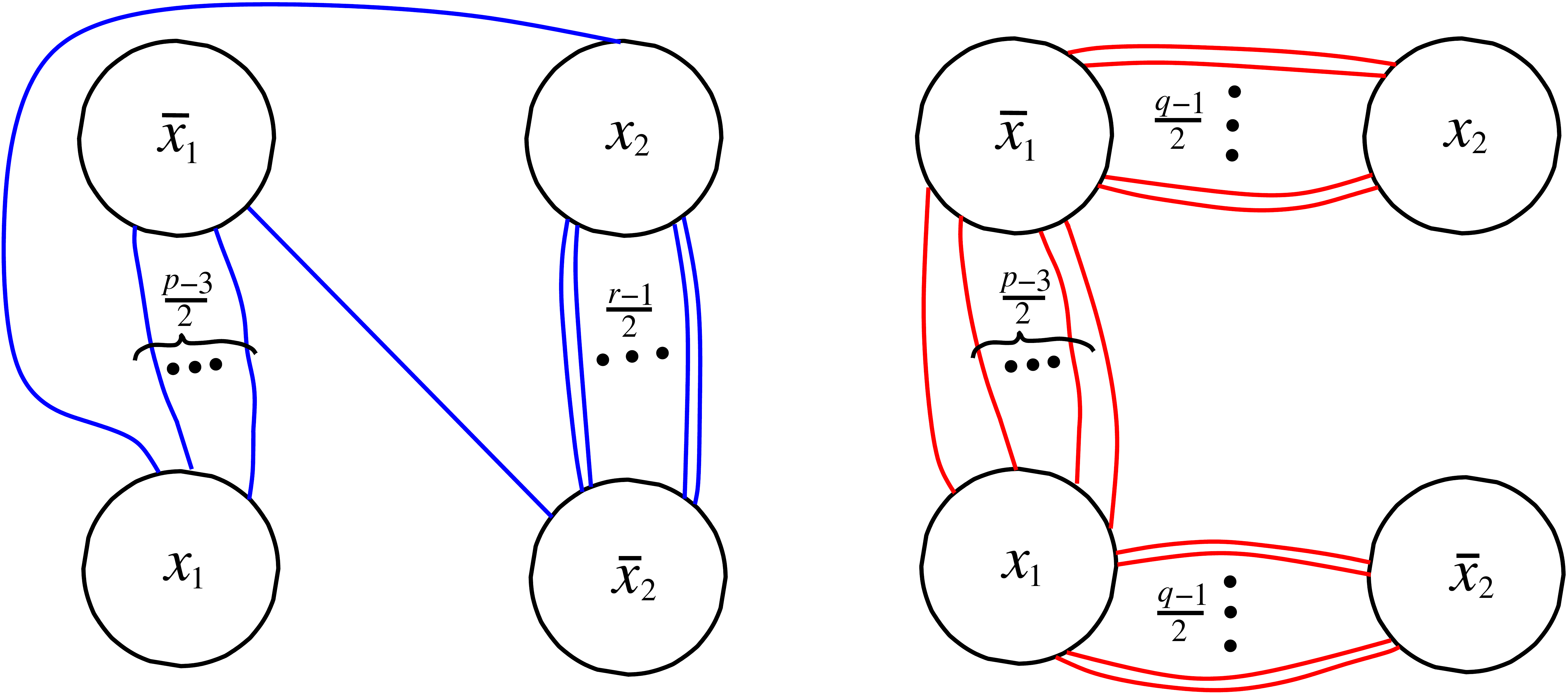}
\caption{The graphs of curves $a_1$ and $a_2$}
\label{a1a2}
\end{figure}

The arc $\gamma$, being $\partial$--parallel in $E(F)$ by
Corollary~\ref{coro43}, can 
be isotoped into this Whitehead diagram as a properly embedded arc with
ends disjoint with $G$ (that is, after an isotopy of $E(F)$, we may
assume that $\gamma$ is disjoint with the system of disks $x_1$
and~$x_2$). 
%
Recall that we are assuming that $\gamma$ is the core of a 1--handle
  of a one-handled circular decomposition of $E(k)$ based on $F$.
  Therefore, after drilling out $\gamma$, there is an essential disk
  in $E(\gamma)$ disjoint with $\Gamma$; that is, 
after drilling out~$\gamma$, and 
obtaining a new Whitehead diagram with six fat vertices with Whitehead
graph~$G'$, 
there is a sequence of handle slides of $E(\gamma)$ that 
disconnect the graph $G'$, giving an essential disk in $E(F)$ disjoint
with~$\Gamma$ (see Section~\ref{sec27}).

Let $G_i$ be the Whitehead graph
of the pair $(E(F),a_i)$ with respect to $x_1,x_2$. See
Figure~\ref{a1a2}. After drilling out the arc $\gamma$ from the diagram
of~$G_i$, 
we obtain a new Whitehead diagram for $(E(\gamma),a_i)$ with six fat
vertices, corresponding to $x_1,x_2$, and~$z$, and with Whitehead
graph $G_i'$. Performing the handle slides of $E(\gamma)$ as 
above, the image of 
the graph~$G_i'$ will be also disconnected, giving an
essential disk in $E(\gamma)$ disjoint with $a_i$ ($i=1,2$).

Notice that if we drill out an arc of length one in $G_i$ and perform
handle slides, the image of $G_i$ is disconnected (it contains four
isolated fat vertices), $i=1,2$. We deal with this kind of arcs after
Claims~1 and~2.

\textbf{Claim 1:} 
Let $\alpha$ be a properly embedded arc in
$(E(F),a_2)$, disjoint with $a_2$, such that $\alpha$ is parallel
into $\partial E(F)$, and $\alpha$ has length at least two in
$G_2$. 
Then any essential disk in $E(\alpha)$ intersects $a_2$.

\begin{proof}
The arc $\alpha$ \emph{minimally} encircles a number of edges of the
graph $G_2$. For example, the arc that encircles the two diagonal
edges in Figure~\ref{Ppqr4} actually has length 0. 

Now, after sliding the handle defined by the disk $x_2$ along the handle
defined by~$\bar x_1$ on the right side of Figure~\ref{a1a2}, 
the image of the graph
$G_2$ looks like Figure~\ref{Ppqr4}. The fat vertices of this graph
are also obtained from the images of the disks $x_1$ and $x_2$ after the
slide. We still call this new graph and new disks $G_2$, and $x_1$,  $x_2$,
respectively. 
This graph has $(q-3)/2$ vertical edges
connecting $x_2$ with $\bar x_2$, one diagonal edge connecting $x_2$
with $\bar x_1$, one diagonal edge connecting $x_1$
with $\bar x_2$, 
and there are $(p-1)/2$ 
vertical arcs connecting $x_1$ with $\bar x_1$.

Let $z$ be a minimal $\partial$--parallelism disk for $\alpha$ in $E(F)$,
and let $G$ be the Whitehead graph of $(E(\alpha),a_2)$ with respect
to $x_1,x_2$, and $z$, which is
obtained from~$G_2$, by cutting
along $z$ and adding two fat vertices $z$ and $\bar z$.

\noindent
\emph{Case ``Length of $\alpha=2$''}:
Since $p\geq5$, there are at least two vertical edges connecting~$x_1$
and $\bar x_1$. Then there are two 
types of arcs of length two for the
edges of~$G_2$ around 
$x_1$ as in Figure~\ref{Ppqr4},
for, any arc encircling two consecutive edges of~$G_2$
connecting~$x_1$ and $\bar x_1$ can be slid in $E(F)$  
into an arc of type~1 or type~2. See Figure~\ref{type2} where the arcs that
can be slid in $E(F)$ into an arc of type~2 are shown.
\begin{figure}
\centering
\includegraphics[width=8cm]{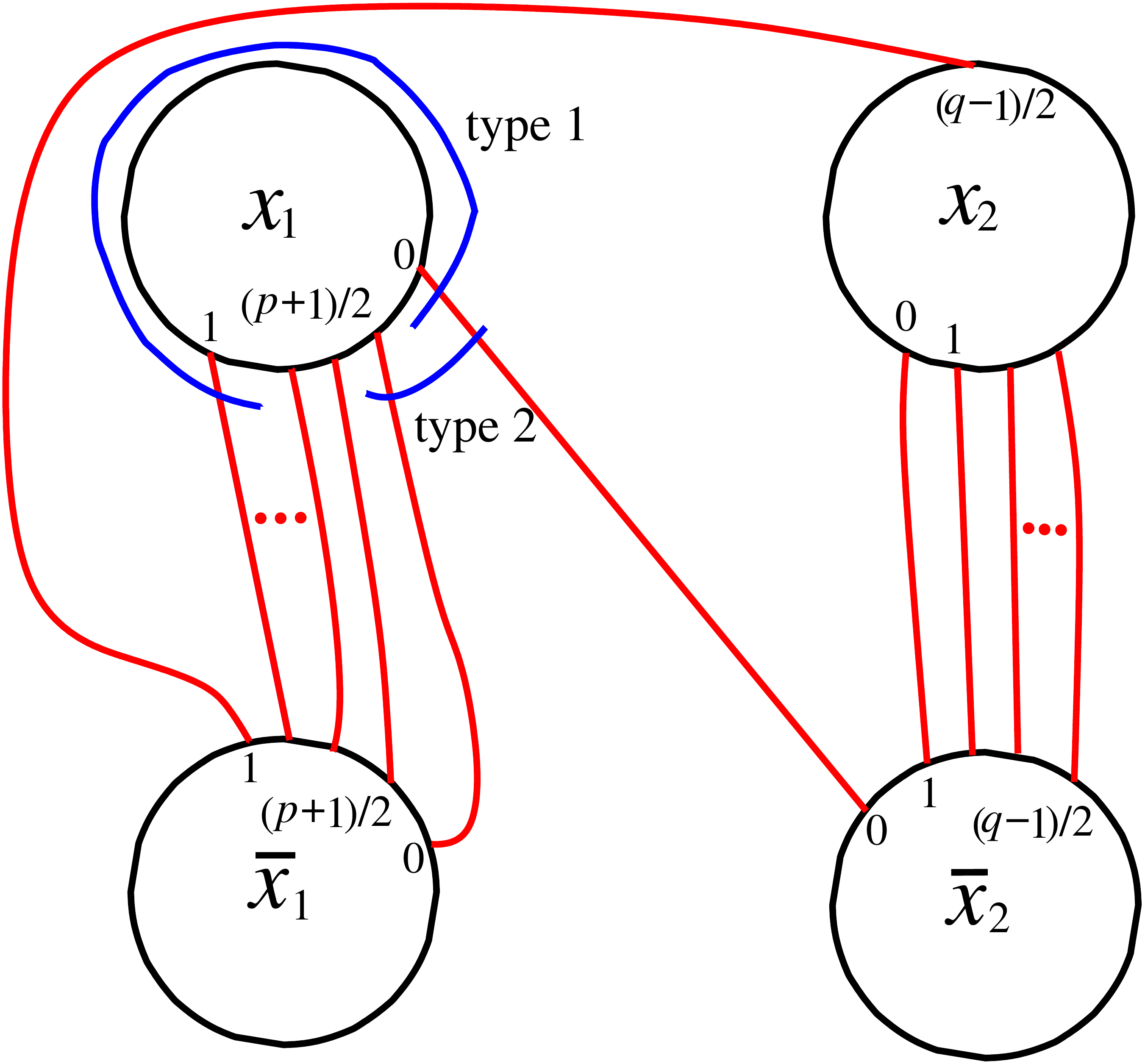}
\caption{}
\label{Ppqr4}
\end{figure}
\begin{figure}
\centering
\includegraphics[width=8cm]{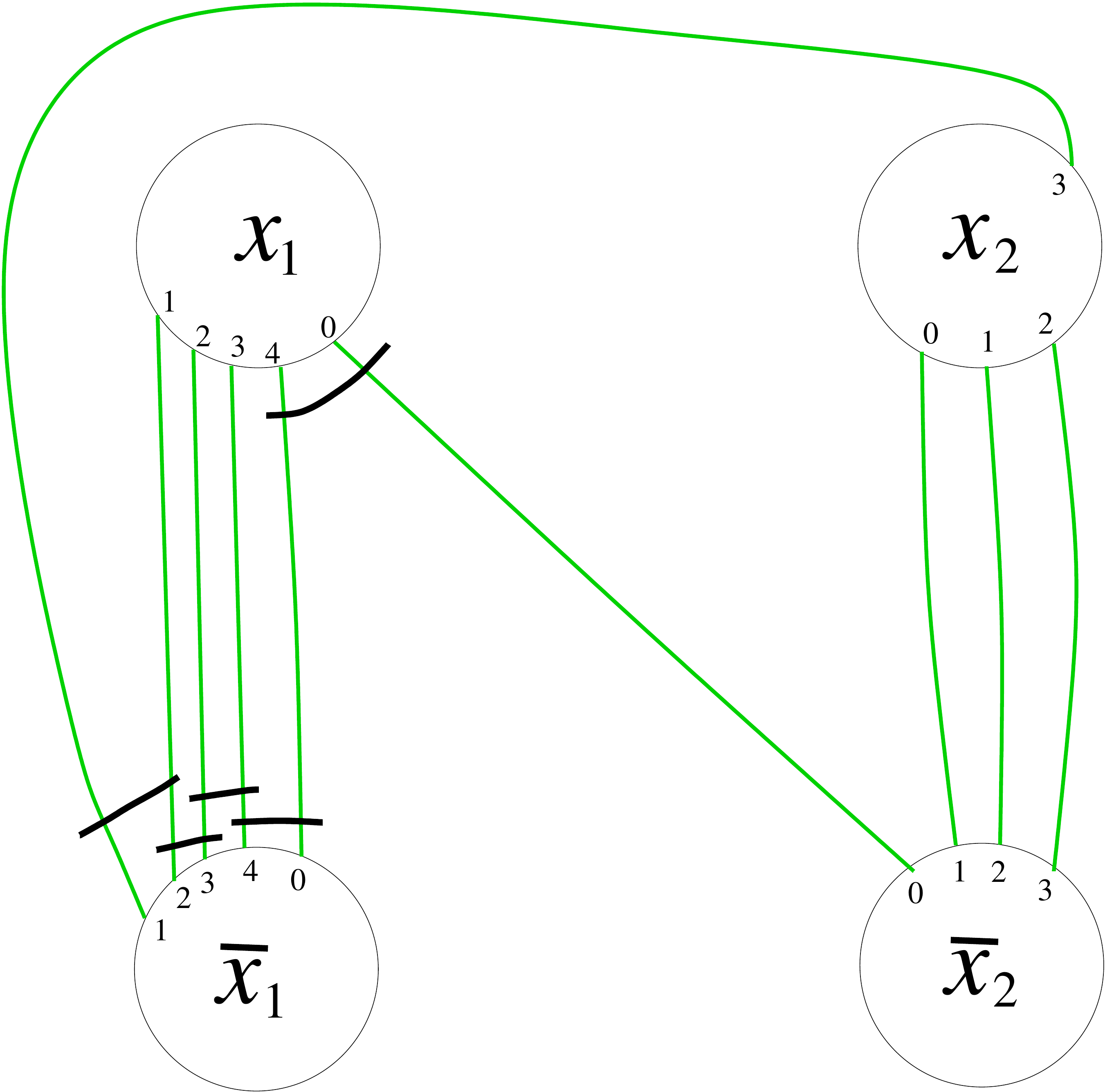}
\caption{}
\label{type2}
\end{figure}
\begin{figure}
\centering
\noindent
\includegraphics[height=5.5cm]{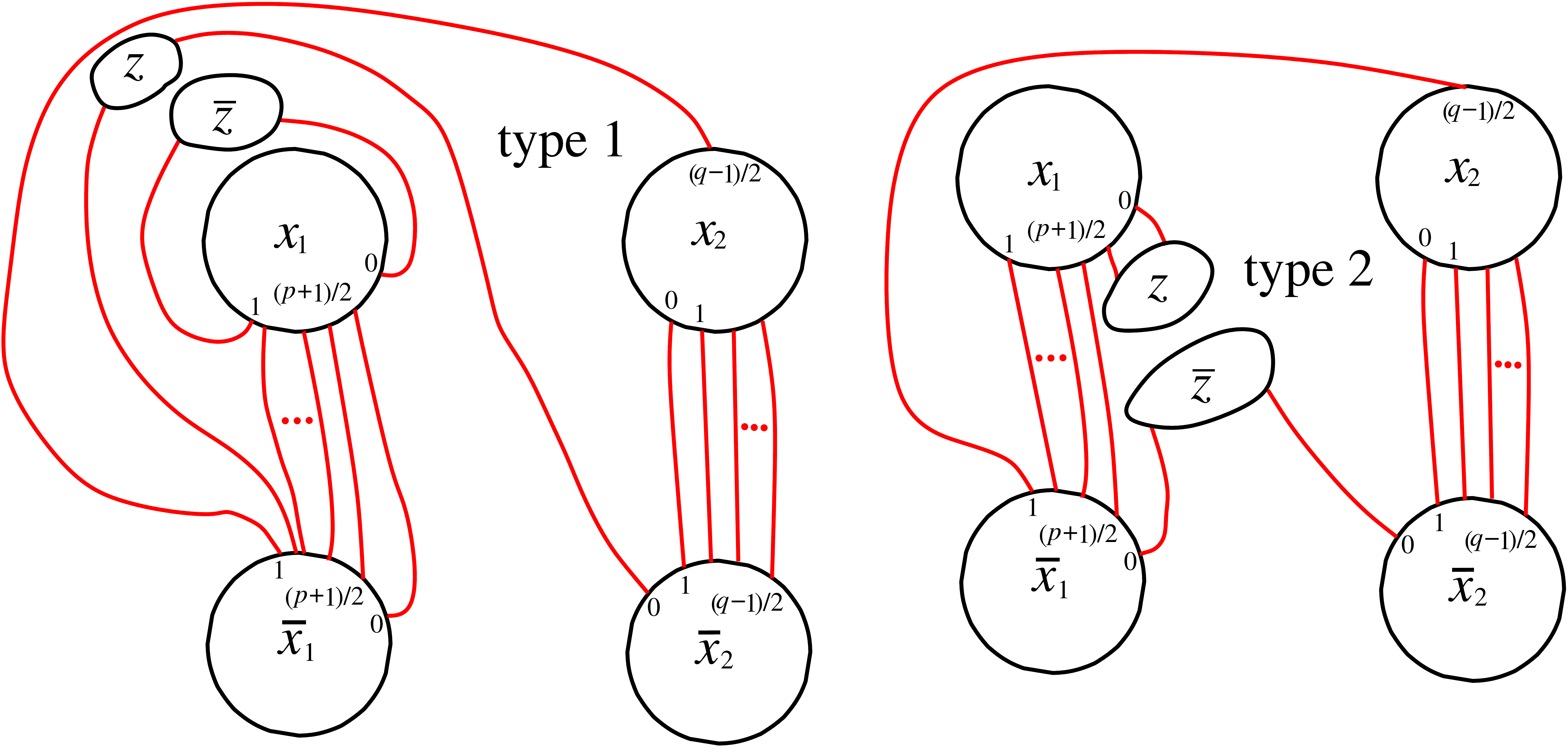}
\caption{}
\label{Ppqr4virotes}
\end{figure}

After drilling out the arc $\alpha$, if $\alpha$ is of type~1, or of
type~2, the new Whitehead graph contains a cut vertex (see
Figure~\ref{Ppqr4virotes}).  

After sliding handles, as in Section~\ref{sec27},
we end up with a graph $G_2'$ 
with its simple associated graph 
a cycle of six vertices and six
edges; that is, 
this simple graph contains no cut vertex.
Therefore, $G'_2$ contains no cut vertex, and by
Corollary~\ref{coro29},~$a_2$ intersects every essential disk
of~$E(\alpha)$.




If $q\geq7$, there are at least two vertical edges connecting
$x_2$ and $\bar x_2$. Then, by symmetry, the analysis of arcs of
length two around $x_2$ and $\bar x_2$ is the same as for arcs of
length two around $x_1$ and $\bar x_1$.

If $q=5$, there is a single vertical edge connecting
$x_2$ and $\bar x_2$, and, then, there are no arcs of length two
around $x_2$ or $\bar x_2$.

For arcs not around a vertex of $G_2$, there are two more types of arcs of length two as in Figure~\ref{Ppqr8},
\begin{figure}[htp]
\centering
\includegraphics[width=8cm]{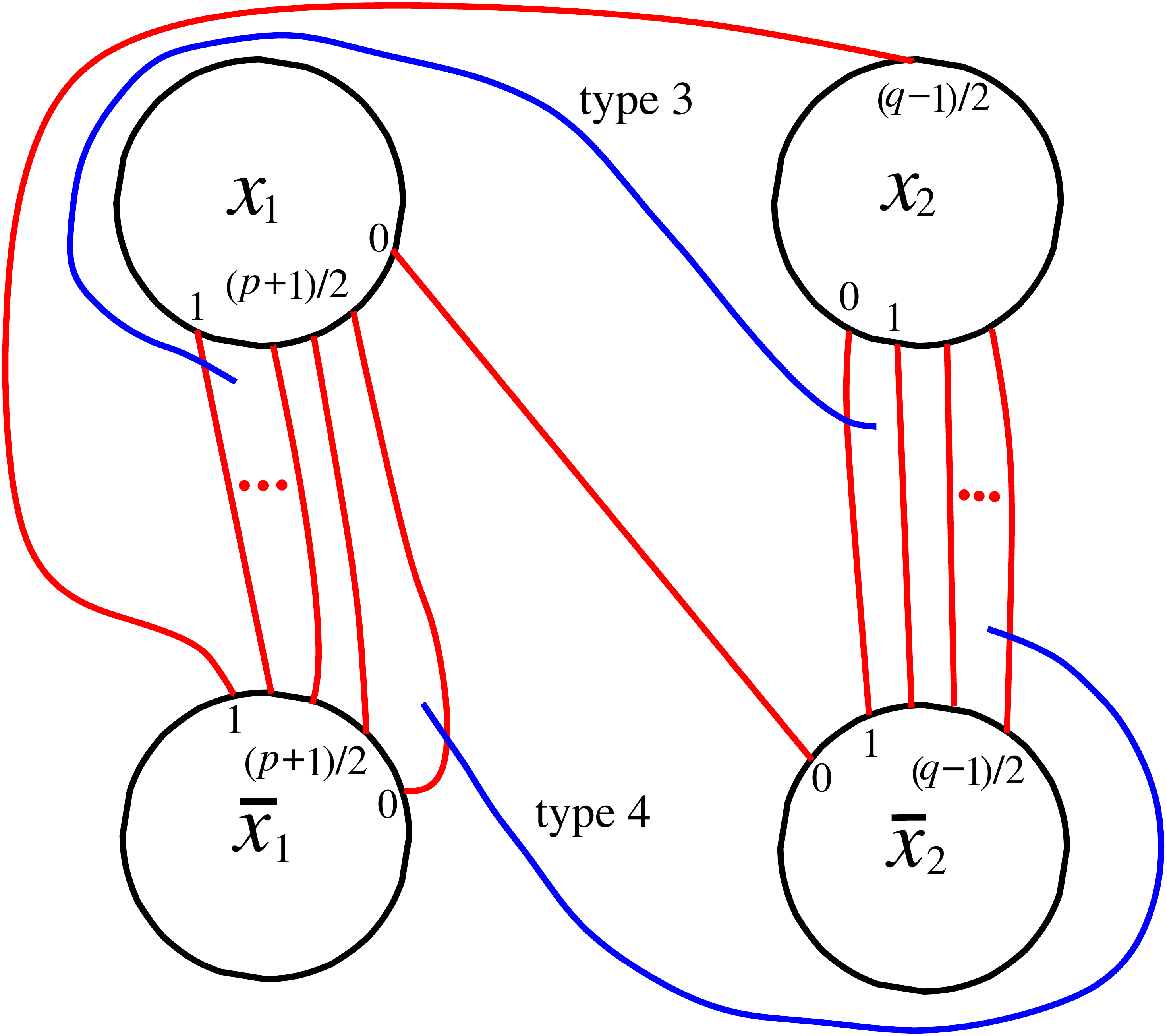}
\caption{}
\label{Ppqr8}
\end{figure}
but, after drilling out the arc $\alpha$ of type~3 or~4, the 
new Whitehead graph contains no cut vertex, and then, by
Corollary~\ref{coro29}, $a_2$ intersects every essential disk of~$E(\alpha)$.

\emph{Case ``Length of $\alpha\geq3$''}:
If $\alpha$ is an arc around $x_i$, we may assume that the length
of $\alpha$ in $G_2$ is between~3 and $degree(x_i)/2$ (see last
paragraph of Section~\ref{sec271}), and~$\alpha$ contains a 
sub-arc of type~1 or~2. After drilling out the
arc $\alpha$ and sliding, if there appear cut vertices, we end up
with a graph with its simple associated graph 
a cycle with six vertices and six edges.
Therefore, $a_2$ again intersects every essential disk
of~$E(\alpha)$. 

If $\alpha$ is of length at least 3, and $\alpha$ contains a sub-arc of
type~3 or~4, then, after drilling out the arc $\alpha$, the 
new Whitehead graph contains no cut vertex, and, by
Corollary~\ref{coro29}, we conclude that 
$a_2$ intersects every essential disk of~$E(\alpha)$. 

By the final remarks of
  Section~\ref{sec271}, the arcs of type 1-4 exhaust all arcs to be
  considered as arcs of a one-handled decomposition for $G_2$.

\end{proof}

\textbf{Claim 2:} 
Let $\alpha$ be a properly embedded arc in
$(E(F),a_1)$, disjoint with $a_1$, such that $\alpha$ is parallel
into $\partial E(F)$, and $\alpha$ has length at least two in
$G_1$. 
Then any essential disk in $E(\alpha)$ intersects $a_1$.

\begin{proof}

The Whitehead graph $G_1$ of $(E(F),a_1)$ has a shape as in
Figure~\ref{Ppqr4}, but with $(r-1)/2$ vertical edges
connecting $x_2$ with $\bar x_2$, one diagonal edge connecting~$x_2$
with~$\bar x_1$, one diagonal edge connecting $x_1$
with $\bar x_2$, and  there are $(p-3)/2$  vertical
arcs connecting $x_1$ with $\bar x_1$. 

A similar
(symmetric) analysis as in Claim~1, gives that $a_1$ intersects every
essential disk of~$E(\alpha)$. 

\end{proof}

We are assuming that, after drilling out
the arc $\gamma$, there is a set of handle slides of $E(\gamma)$ that 
disconnect the graph $G'$, giving an essential disk in $E(F)$ disjoint
with~$\Gamma$.

By Claims~1 and~2, $\gamma$ is of
length one in $G_1$, and of length one in $G_2$. If $\gamma$ is around
one fat vertex $\xi$ of $G$, it might happen that $\gamma$ encircles exactly
one edge of~$G_1$, and all but one edge of $G_2$, or vice versa. In this
case, $\gamma$ is around either $x_2$ or $\bar x_2$. There are
four arcs around $x_2$, and four arcs around $\bar x_2$ of this
kind. The four arcs with this property around $\bar x_2$ can be slid
in $E(F)$ and become equivalent to the four arcs around $x_2$ in
Figure~\ref{excepcion}; see Section~\ref{sec271}.
\begin{figure}[htp]
\centering
\includegraphics[width=8cm]{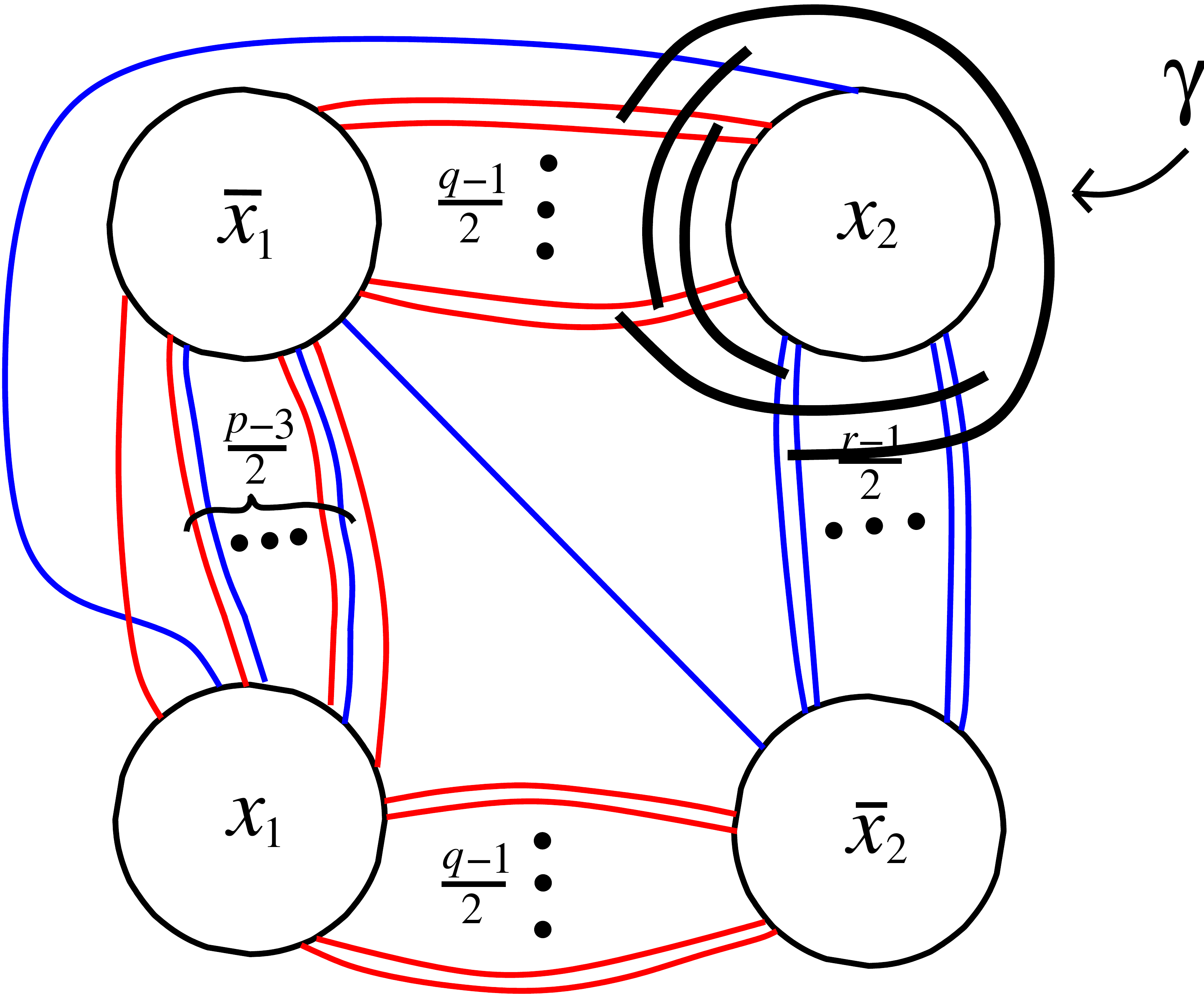}
\caption{}
\label{excepcion}
\end{figure}
After drilling
out $\gamma$, there is a cut vertex in the new Whitehead graph, and a
single handle slide produces a graph $G'$ with no cut 
vertices. By Corollary~\ref{coro29}, there are no essential disks
disjoint with $G$ in $E(\gamma)$. Another possibility is that $\gamma$
encircles all but one edge of $G_1$ and all but one edge of
$G_2$, but in this case, $\gamma$ also encircles exactly
one edge of $G_1$, and exactly one edge of $G_2$.

There are four types of arcs of length two encircling exactly
one edge of $G_1$ and exactly one edge of $G_2$ (see Figure
\ref{Ppqr10}). Again, any arc
encircling two edges of $G$, one of $G_1$ and one of $G_2$
can be slid in $E(F)$ into an arc of type~1, type~2, type~3, or type~4;
see Section~\ref{sec271}.  

\begin{figure}[htp]
\centering
\includegraphics[width=8cm]{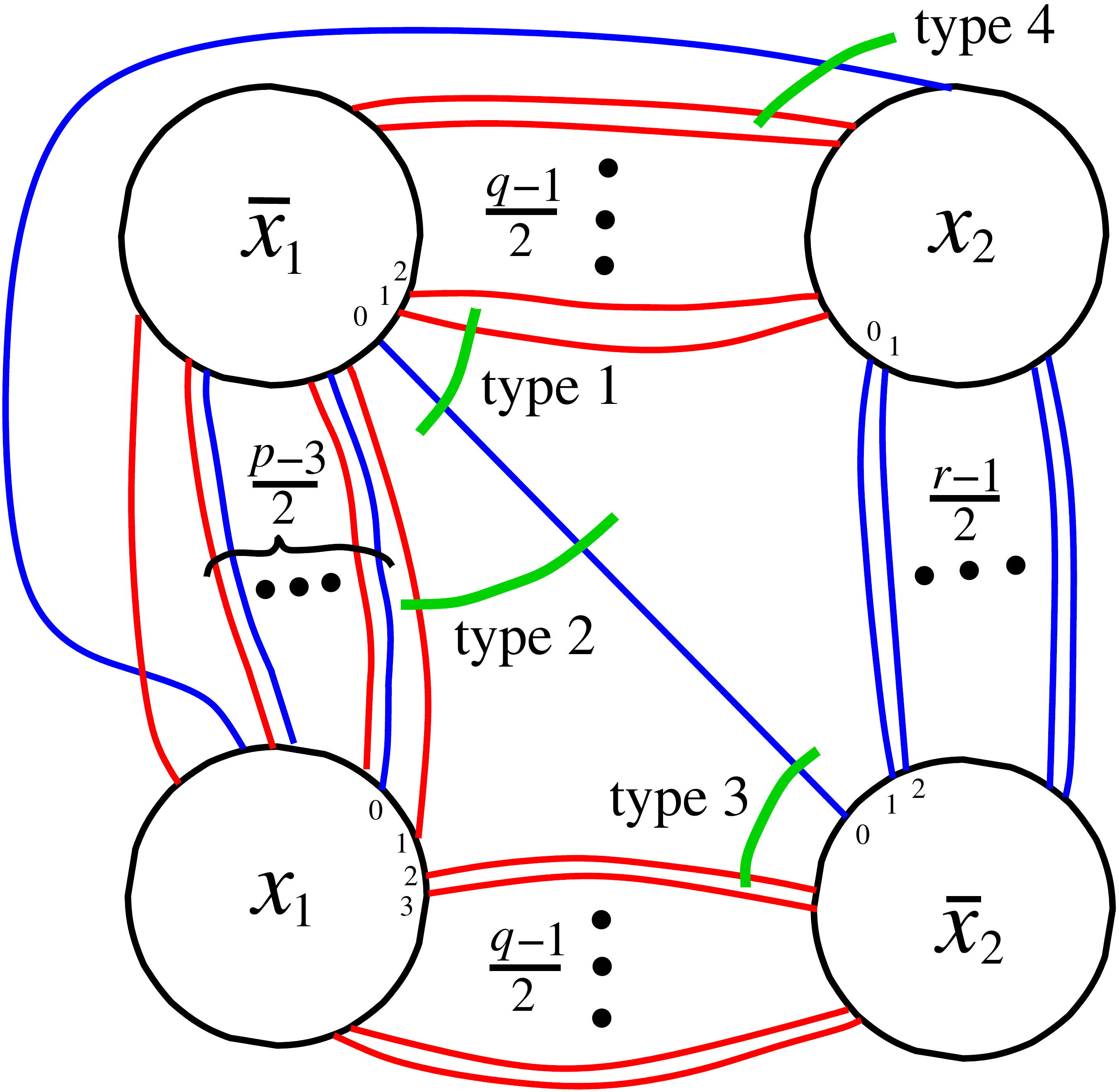}
\caption{}
\label{Ppqr10}
\end{figure}

After drilling out the arc $\gamma$, if $\gamma$ is of type~1, type~2,
type~3, or type~4, the 
new Whitehead graph contains a cut vertex. 
After sliding, we end up with a graph $G'$ 
with its simple associated graph as one of the drawings in Figure~\ref{Ppqr11}. Since
these graphs contain no cut vertex, by Corollary~\ref{coro29}, 
\begin{figure}[htp]
\centering
\includegraphics[height=4true cm]{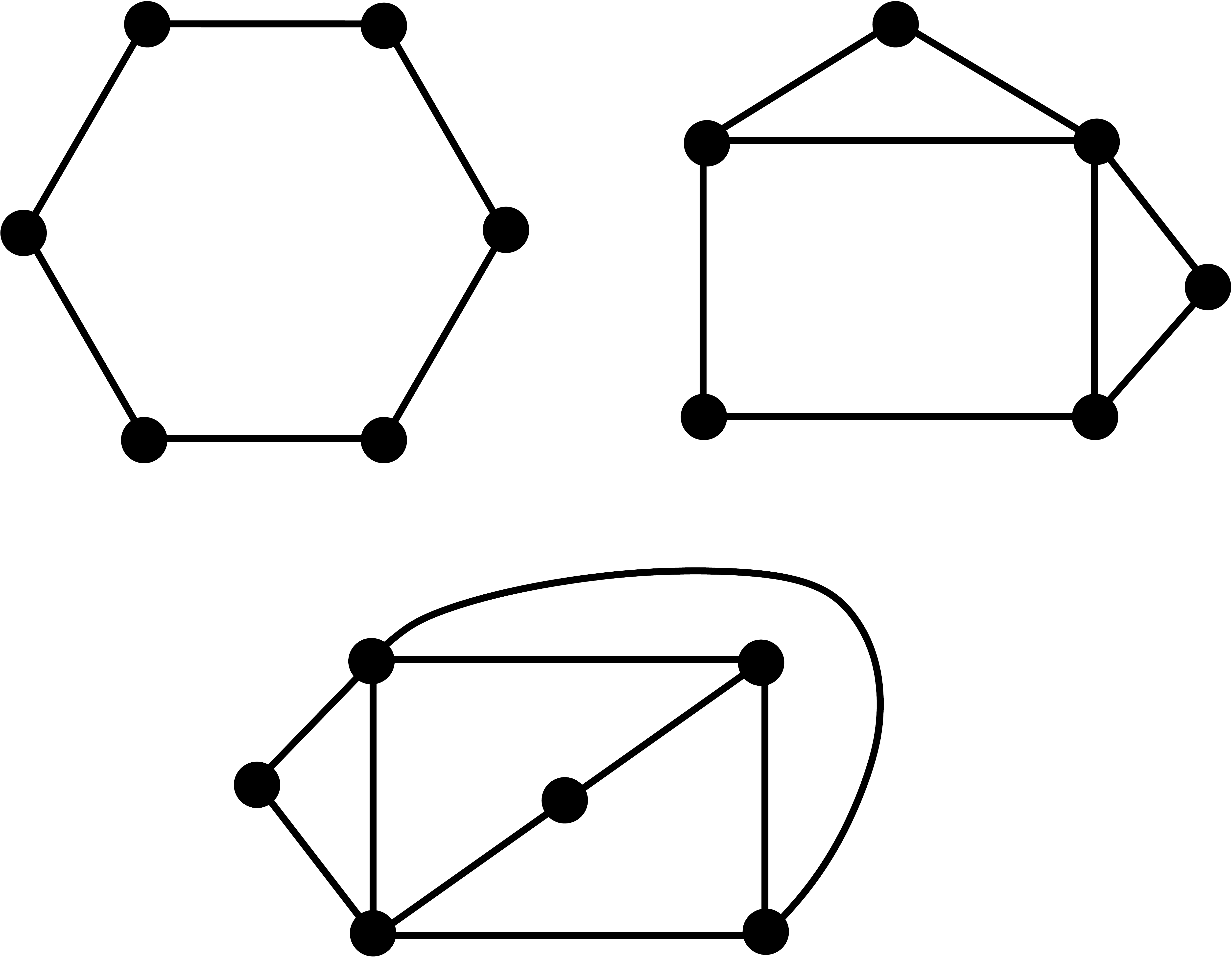}
\caption{}
\label{Ppqr11}
\end{figure}
%
%
we conclude that any
essential disk in~$E(\gamma)$ intersects~$G$, and, therefore,
intersects $\Gamma\subset F$. This contradiction shows that~$h(F)\neq1$. Since
$k=P(p,q,r)$ is not fibered, 
and
$h(F)\leq2$, by Corollary~\ref{coro36}, it follows that $h(F)=2$, 
when $p,q,r\geq5$.

This finishes Case~1.


\emph{Case 2:} ``$p<0$, and $q,r >0$''.
As in Example~\ref{exPretzel}, we construct a spine $\Gamma= a_1\vee
a_2 $ for $F$ starting with the 
spine
shown in Figure~\ref{fig31teta}, but now we slide the
middle edge of the $\theta$--graph  rightwards. The spine $\Gamma$ looks like
Figure~\ref{Ppqr1111}, and the Whitehead diagram for $(E(F),\Gamma)$
with respect to the system of disks $x_1,x_2$ is as in
Figure~\ref{Ppqr1112}. By Remark~\ref{remark26}, we only need to analyze the
handle decompositions of $(E(F),\Gamma)$.
\begin{figure}[htp]
\centering
\includegraphics[height=7true cm]{espinanegDer.pdf}
\caption{}
\label{Ppqr1111}
\end{figure}
\begin{figure}[htp]
\centering
\includegraphics[height=7true cm]{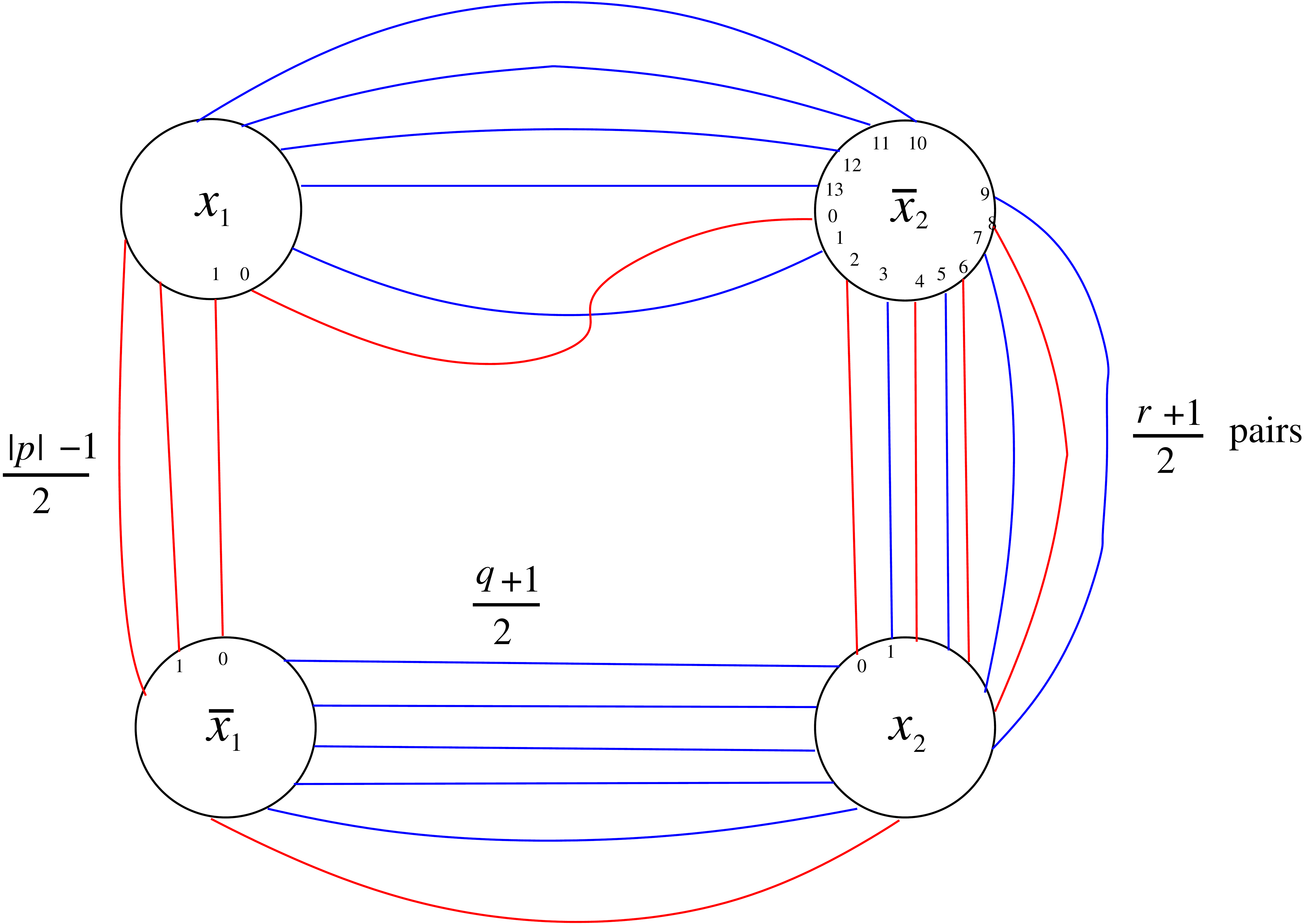}
\caption{}
\label{Ppqr1112}
\end{figure}

The Whitehead graphs $G_1$ and $G_2$ of the pairs $(E(F),a_1)$ and
$(E(F),a_2)$, respectively, are shown in
Figure~\ref{Ppqr1113}. Although these diagrams are similar to the
diagrams in Figure~\ref{a1a2} of Case~1, the configuration of the diagram for
$a_1$ here is not the same as the configuration of the positive case
(Case~1); that is, the corresponding Whitehead diagrams are not
isomorphic.
\begin{figure}[htp]
\centering
\includegraphics[height=12true cm]{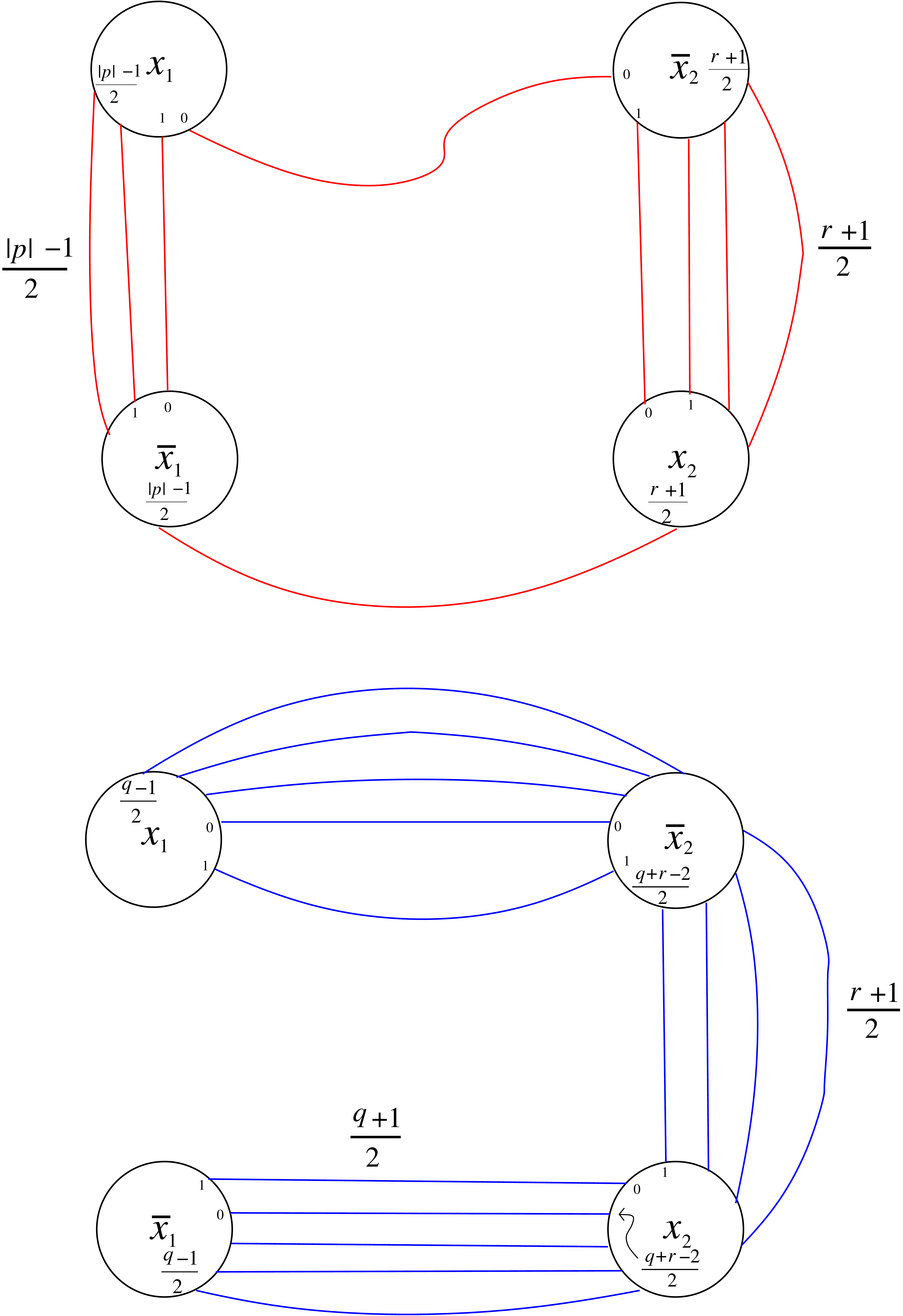}
\caption{}
\label{Ppqr1113}
\end{figure}
%


However, the analysis of the different properly embedded arcs in the
Whitehead diagrams of $(E(F),a_1)$, $(E(F),a_2)$, and $(E(F),\Gamma)$,
giving rise to a possible one-handled decomposition, is 
completely similar as in Case~1. 

The Whitehead diagram for $(E(F),a_2)$ is isomorphic to the
corresponding Whitehead diagram of Case~1. Then

\textbf{Claim 1:} 
Let $\alpha$ be a properly embedded arc in
$(E(F),a_2)$, disjoint with $a_2$, such that $\alpha$ is parallel
into $\partial E(F)$, and $\alpha$ has length at least two in
$G_2$. 
Then any essential disk in $E(\alpha)$ intersects $a_2$.
\newline\null\hfill$\Box$

\textbf{Claim 2:} 
Let $\alpha$ be a properly embedded arc in
$(E(F),a_1)$, disjoint with $a_1$, such that $\alpha$ is parallel
into $\partial E(F)$, and $\alpha$ has length at least two in
$G_1$. 
Then any essential disk in $E(\alpha)$ intersects $a_1$.
\begin{proof}
We first analyze arcs of length 2 in $G_1$. The arcs around vertices
$x_1$ and $\bar x_1$ are shown in Figure~\ref{nega1}. There are only
two types after sliding the arcs in $E(F)$.
\begin{figure}[htp]
\centering
\includegraphics[height=6true cm]{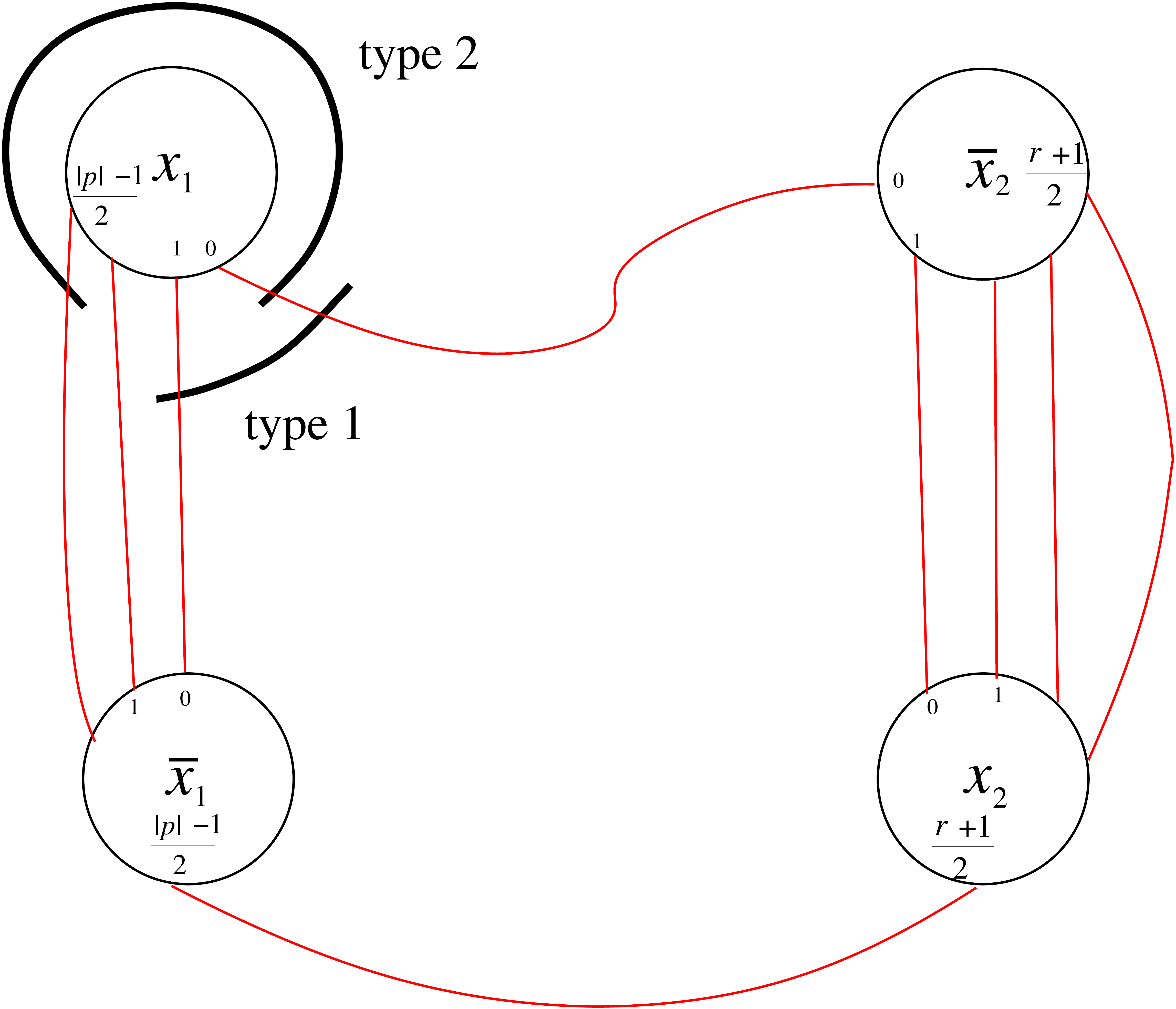}
\caption{}
\label{nega1}
\end{figure}
After drilling out the arc $\alpha$, if $\alpha$ is of type~1, or of
type~2, the new Whitehead graph contains a cut vertex, but
after sliding handles, as in Section~\ref{sec271},
we end up with a graph $G_1'$ 
with its simple associated graph a cycle of six vertices and six
edges; that is, 
this simple graph contains no cut vertex. Therefore, $G'_1$ contains no
cut vertex, and by Corollary~\ref{coro29}, $a_2$ intersects every
essential disk of~$E(\alpha)$. 

For arcs of length 2 around the vertices $x_2$ and $\bar x_2$, the
analysis is identical to Case~1.

For arcs not around a vertex of $G_1$, there are two more types of arcs of length two as in Figure~\ref{nega12},
\begin{figure}[htp]
\centering
\includegraphics[height=6true cm]{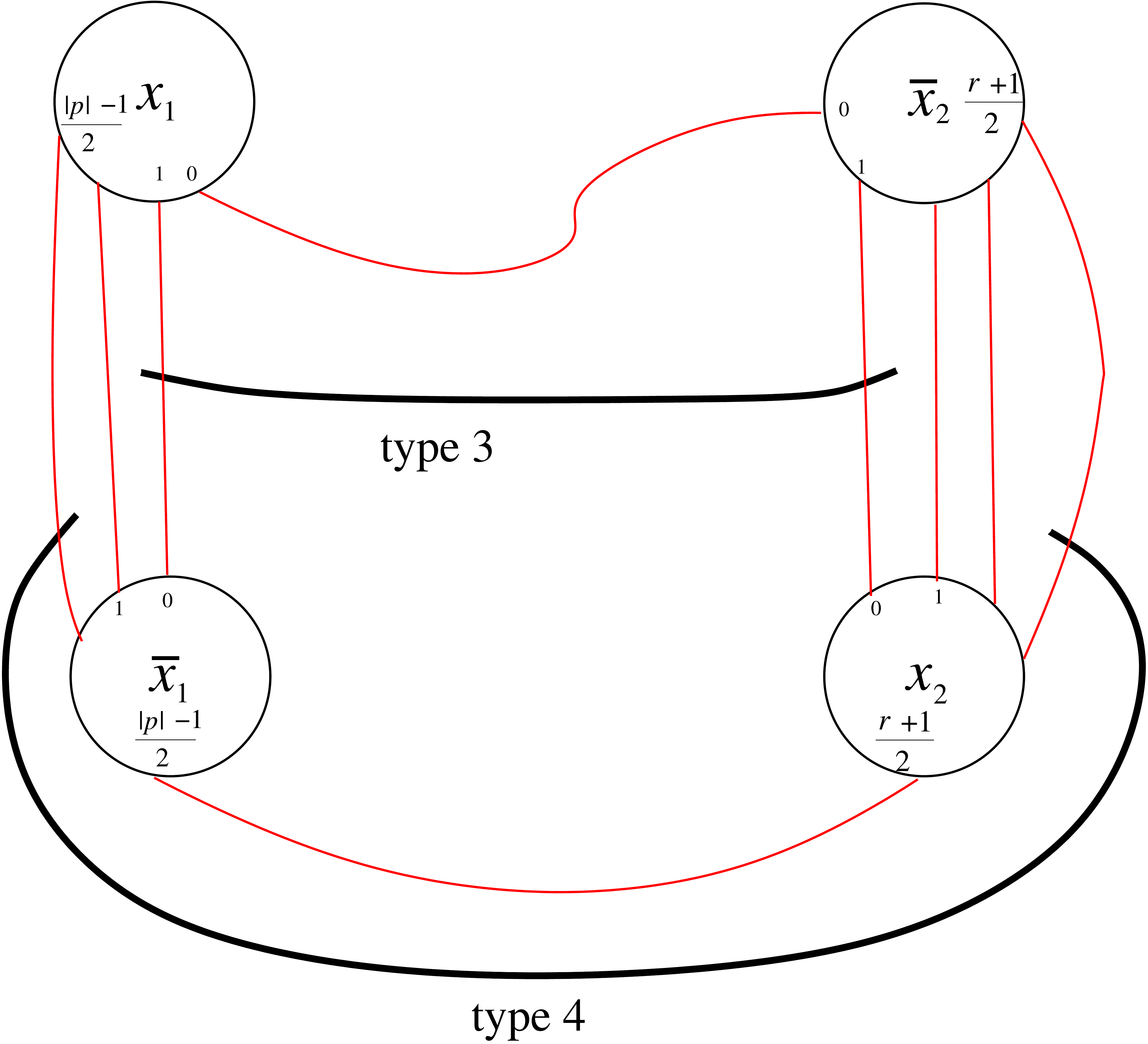}
\caption{}
\label{nega12}
\end{figure}
but, after drilling out the arc $\alpha$ of type~3 or~4, the 
new Whitehead graph contains no cut vertex, and then, by
Corollary~\ref{coro29}, $a_2$ intersects every essential disk
of~$E(\alpha)$.

For arcs of length at least three, we follow the same argument as in
Case~1, and conclude that $a_2$ intersects every essential disk
of~$E(\alpha)$.

\end{proof}

Recall that we are assuming that $\gamma$ is the core of a 1--handle
  of a one-handled circular decomposition of $E(k)$ based on $F$. In
  view of Claims~1 and~2,
as in Case~1, we see that the arc $\gamma$ encircles exactly one edge
of $G_1$, and exactly one edge of $G_2$.

There are four types of arcs of length two encircling exactly
one edge of $G_1$ and exactly one edge of $G_2$ (see Figure
\ref{neg2colors}). For, any arc
encircling two edges of $G$, one of $G_1$ and one of $G_2$
can be slid in $E(F)$ into an arc of type~1, type~2, type~3, or type~4
(Section~\ref{sec271}).
\begin{figure}[htp]
\centering
\includegraphics[height=6true cm]{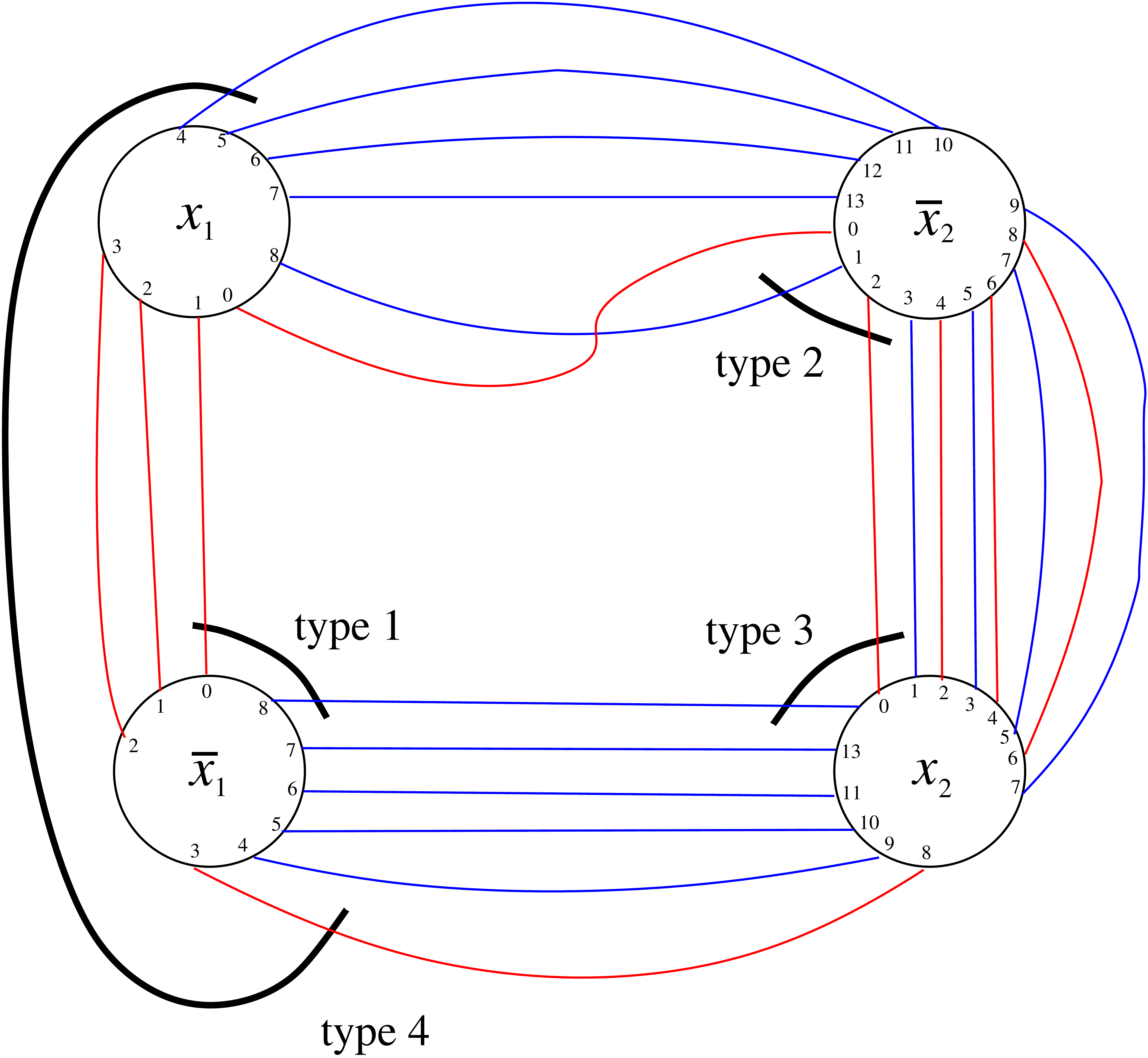}
\caption{}
\label{neg2colors}
\end{figure}

After drilling out the arc $\gamma$, if $\gamma$ is of type~1, type~2,
type~3, or type~4, the 
new Whitehead graph contains a cut vertex. 
After sliding, we end up with a graph $G'$ 
with its simple associated graph as one of the drawings in
Figure~\ref{negGraphs}. Since 
these graphs contain no cut vertex, by Corollary~\ref{coro29},
\begin{figure}[htp]
\centering
\includegraphics[height=3true cm]{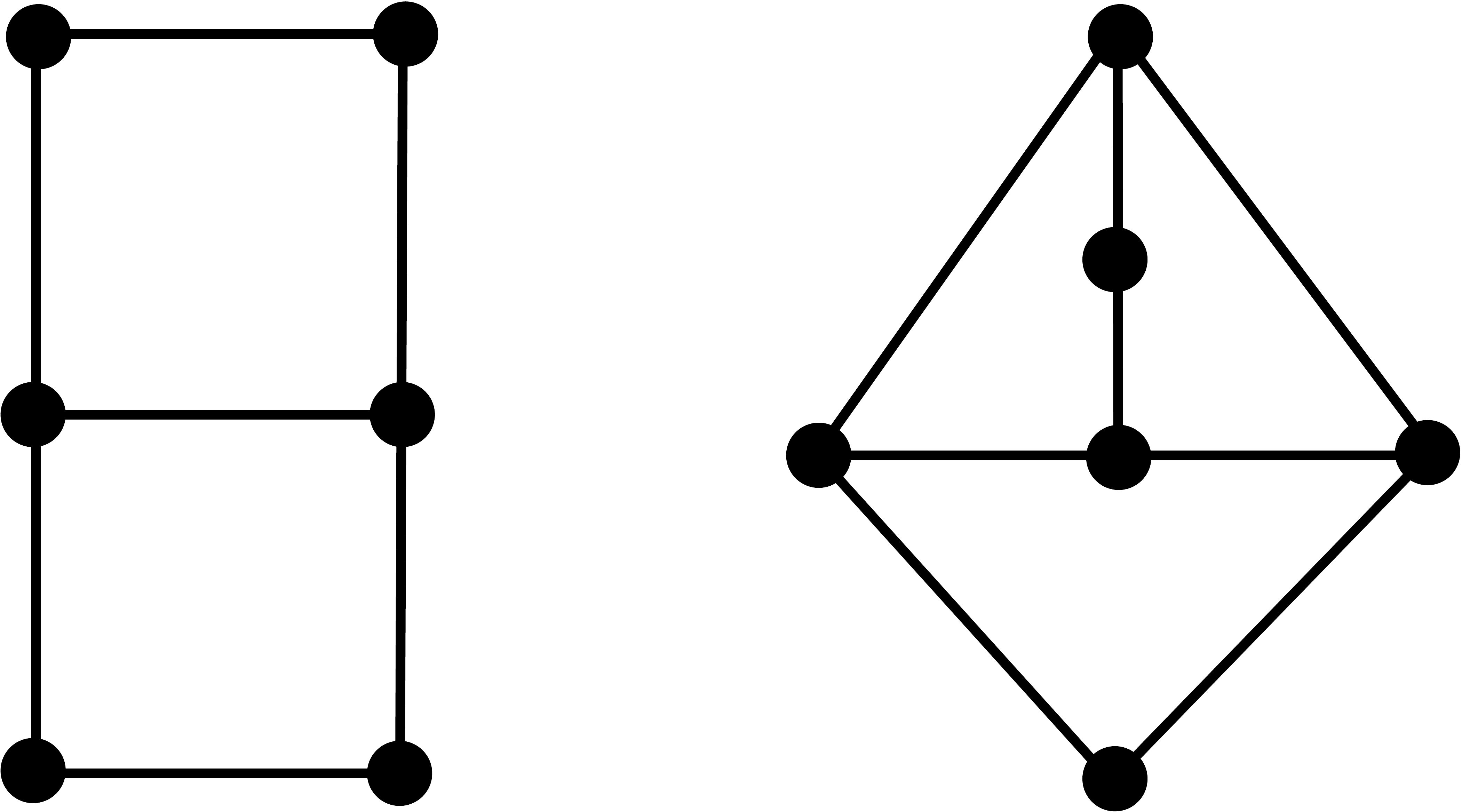}
\caption{}
\label{negGraphs}
\end{figure}
we conclude that any
essential disk in~$E(\gamma)$ intersects~$G$, and, therefore,
intersects $\Gamma\subset F$. Thus,~$h(F)\neq1$. Since
$k=P(p,q,r)$ is not fibered, 
and
$h(F)\leq2$, by Corollary~\ref{coro36}, it follows that $h(F)=2$, 
when $p\leq-5$ and $q,r\geq5$.

This finishes Case~2, and also the proof of Theorem~\ref{thm41}.

\end{proof}

\begin{corollary}
Let $k$ be the pretzel knot $P(p,q,r)$ with $|p|, |q|, |r|
\geq5$. Then $cw(k)=6$.
\end{corollary}
\begin{proof}
Since $k$ has a unique incompressible Seifert surface, by
Remark~\ref{remark24} it follows that $cw(k)\in \mathbb{Z}$. By
Theorem~\ref{thm41}, $cw(k)=6$.
\end{proof}

\begin{remark}
\label{remark45}
Theorem~\ref{thm41} gives a family of knots of genus one and handle
number two. This answers in the affirmative a question in
\cite{rurru}: Does there exist a knot $k$ with $h(k) > g(k)$?
\end{remark}

\section{Genus one essential surfaces and powers of primitive
  elements}
\label{sec5}
In this section we show that if $k$ is a free genus one knot with at
least two non-isotopic Seifert surfaces, then the free Seifert surface
of $k$ admits
a special type of spine.
This result is essential to prove the main
theorem of Section~\ref{sec6} (Theorem~\ref{thm67}).

\begin{lemma}
\label{lema51}
Let $H$ be a handlebody of genus $g\geq2$, and  let
$\alpha\subset \partial H$ be a 
simple closed curve. Assume that there is a primitive element
$p\in\pi_1(H)$ such that $\alpha$ represents an element conjugate with
$p^n$ for some 
$n\in\mathbb{Z}$, $n\neq0$. Then there is an
essential 2--disk~$D\subset H$ such that $D\cap\alpha=\emptyset$.
\end{lemma}
\begin{proof}
Consider a basis $\{p,q_2,\dots,q_g\}$ for $\pi_1(H)$. Then
$\pi_1(H)=\langle p\rangle*\langle q_2,\dots,q_g\rangle$ is a
non-trivial splitting, and 
$\alpha$ is conjugate with $p^n\in \langle p\rangle$. Then $\{\alpha\}$ is
separable in $\pi_1(H)$, and the disk $D$ is obtained by Theorem~3.2
of~\cite{stallings}. 
\end{proof}

Let~$\Gamma\cong a_1\vee a_2$ be a graph in the boundary of a genus
two handlebody $H$. We say that \emph{$a_2$ spoils disks for $a_1$}
if for any essential disk $D\subset H$ such that $D\cap
a_1=\emptyset$, the number of points $\#(D\cap a_2)\geq2$.

\begin{theorem}
\label{thm52}
Let $k\subset S^3$ be a non-trivial connected knot, and let $F\subset E(k)$ be a free genus one Seifert surface for $k$. 

There is another genus one
Seifert surface for $k$ which is not equivalent to $F$ if and only if there
exists a spine $\Gamma=a_1\vee a_2$ for $F$ in $\partial\mathcal{N}(F)$ such
that $a_1$ represents an element conjugate to $g^n$ with 
$n\geq2$ for some primitive element $g\in\pi_1(E(F))$,
and $a_2$ spoils disks for $a_1$.
\end{theorem}


\begin{proof}
Let $\Gamma=a_1\vee a_2$ be a spine for $F$ such
that $a_1$ represents an element conjugate to $g^n$ with 
$n\geq2$ for some primitive element $g\in\pi_1(E(F))$,
and $a_2$ spoils disks for $a_1$

Let $D\subset E(F)$ be an essential properly embedded disk such that
$a_1\cap D=\emptyset$, which is given by Lemma~\ref{lema51}. We may
assume that $H_1=\overline{E(F)-\mathcal{N}(D)}$ is a
solid torus. Let $A_1$ be a regular neighbourhood of $a_1$ in
$\partial E(F)$; then $A_1\subset \partial H_1$. Write~$B_1=\overline{\partial
  H_1-A_1}$. Since $|n|\geq2$, the annuli $A_1$ and $B_1$ are non-parallel
in~$H_1$. We push $Int(B_1)$ into $H_1$ to obtain $B_1'$, a properly
embedded annulus in $H_1$.

Let $\mathcal{N}(a_2)\subset\partial E(F)$ be a regular neighbourhood
of $a_2$ such that $A_1\cap \mathcal{N}(a_2)$ is a rectangle; then
$B_2=\overline{\mathcal{N}(a_2)-A_1}$ is a `band'
(that is, a 2--disk) such that $B_2\cap A_1=B_2\cap B_1'$ is a pair of arcs in
$\partial B_1'$. Then $G=B_2\cup B_1'$ is a genus one Seifert surface
for $k$ (we push $Int(G)$ slightly into $E(F)$ to get a properly embedded
surface in $E(F)$).

Now, $\widehat G=G\cap H_1$ is the union of the annulus $B_1'$ with the
disk components of~$\widehat B_2=B_2\cap H_1$. Notice that $\partial
\widehat B_2\subset B_1\subset\partial H_1$. 

By hypothesis $\#(a_2\cap D)\geq2$; thus,
$\widehat G$ is disconnected, and the components of $\widehat G$ are
$B_1'\cup (\text{two 2--disks of } \widehat B_2)$, and at least one
sub-disk $z\subset \widehat B_2$ with $\partial z\subset Int(B_1)$.

Since $|n|\geq2$, we cannot push $B_1'$ onto
$A_1$ in $H_1$.
Then a
  $\partial$--parallelism for $\widehat G$ in $H_1$ contains a
  $\partial$--parallelism $W$ for $B_1'$ onto $B_1$, but then $W$
contains the 
2--disk $z\subset \widehat G$. Therefore,~$\widehat G$ is not
parallel into $\partial H_1$. We conclude that $G$ is not boundary parallel in
$E(F)$ for, a $\partial$--parallelism for $G$ induces a
$\partial$--parallelism for $\widehat G$. 
It follows that $G$ and $F$
are not equivalent. This finishes sufficiency.



Now, if there is another genus one
Seifert surface for $k$ which is not equivalent to~$F$, we can find
still another non-equivalent genus one Seifert surface~$G\subset E(k)$
for $k$ such 
that $G$ and $F$ have disjoint interiors; see \cite{Charle-Tonson}. 
We write $k=G\cap\partial E(F)$.

The surface $G$ splits $E(F)$ into two handlebodies, $H_0\cup
H_1=\overline{E(F)-\mathcal{N}(G)}$, of genus two 
for $H_0$ and~$H_1$ are
irreducible and, since $G$ is $\pi_1$--injective into $H_0$ and $H_1$,
it follows that  $H_0$ and $H_1$ are $\pi_1$--injective into $E(F)$;
therefore, $H_0$ and $H_1$ have free fundamental groups. We assume $\partial
H_i=G\cup (F\times\{i\})$ plus a neighbourhood of~$k$, $i=0,1$. By
considering a system of disks for 
the handlebody~$E(F)$, we see that there is a disk $D\subset E(F)$ that
$\partial$--compresses $G$ in $E(F)$, and $D$ is contained in, say $
H_0$, and
is properly embedded in $H_0$.

Then $k$ is a $((1,0),(n,m))$--curve in $\partial H_0$ (Lemma~4.3 of
\cite{tsutsu}) with $|k\cap D|=2$.
See Figure~\ref{fig3}.
\begin{figure}
\centering
\includegraphics[height=50mm]{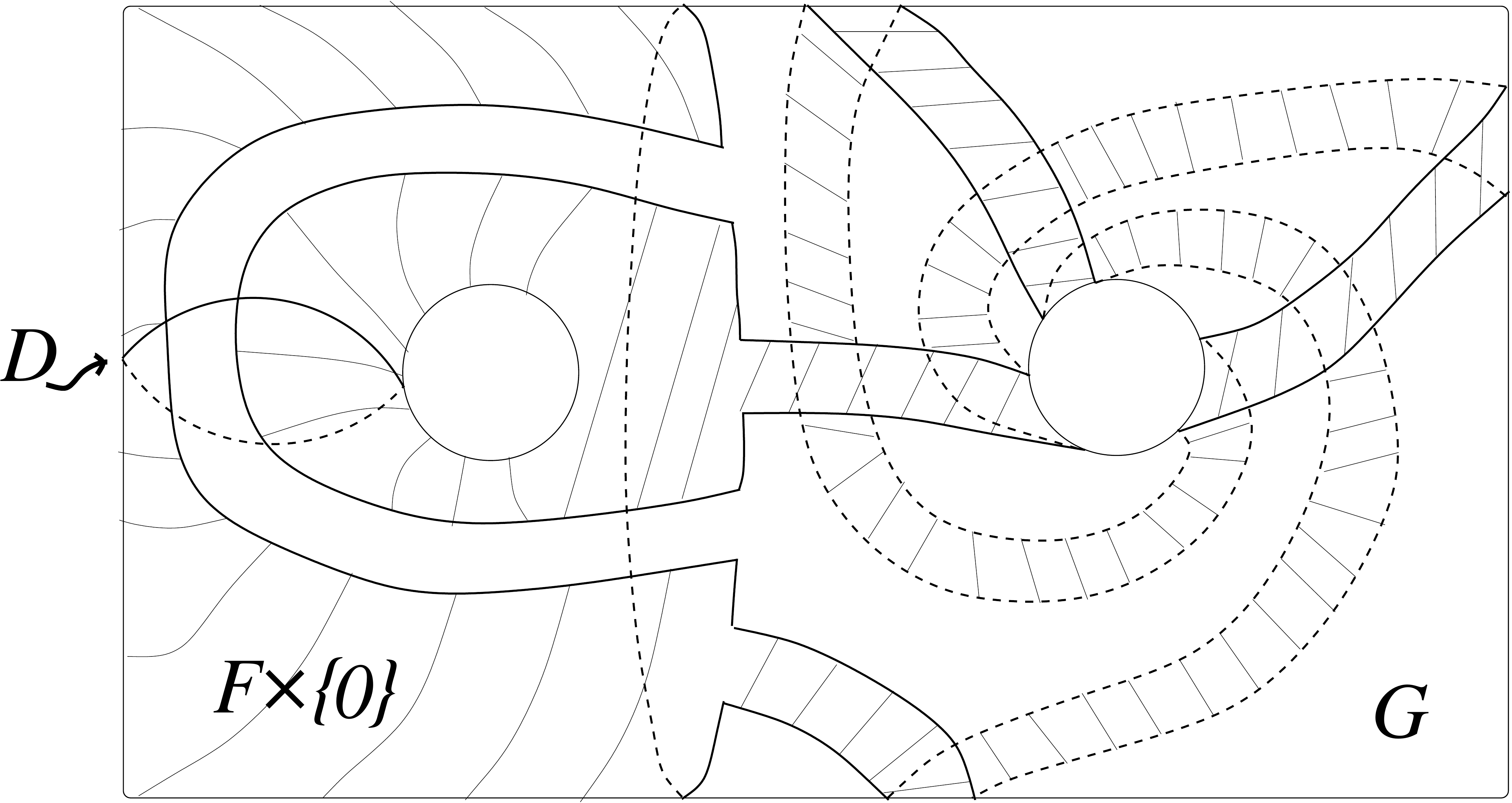}
\caption{Surfaces $G$ and $F\times\{0\}$ in $H_0$.}
\label{fig3}
\end{figure}

Cutting $H_0$ along $D$ we obtain a solid torus $V\subset H_0$ such
that $\widehat G=G\cap V$ is an $(n,m)$--torus annulus in $\partial
V$; and the complementary annulus $\widehat F=\overline{\partial V-\widehat
G}$ contains, and is isotopic, to $(F\times\{0\})\cap V$ in $\partial
V$ with an isotopy fixed outside a regular neighbourhood of~$D$.

Let $a_1\subset F\times\{0\}$ be the core of the annulus $\widehat F$, and let
$b_1\subset \widehat G$ be the core of the annulus $\widehat G$.
\begin{figure}
\centering
\centerline{\includegraphics[height=50mm]{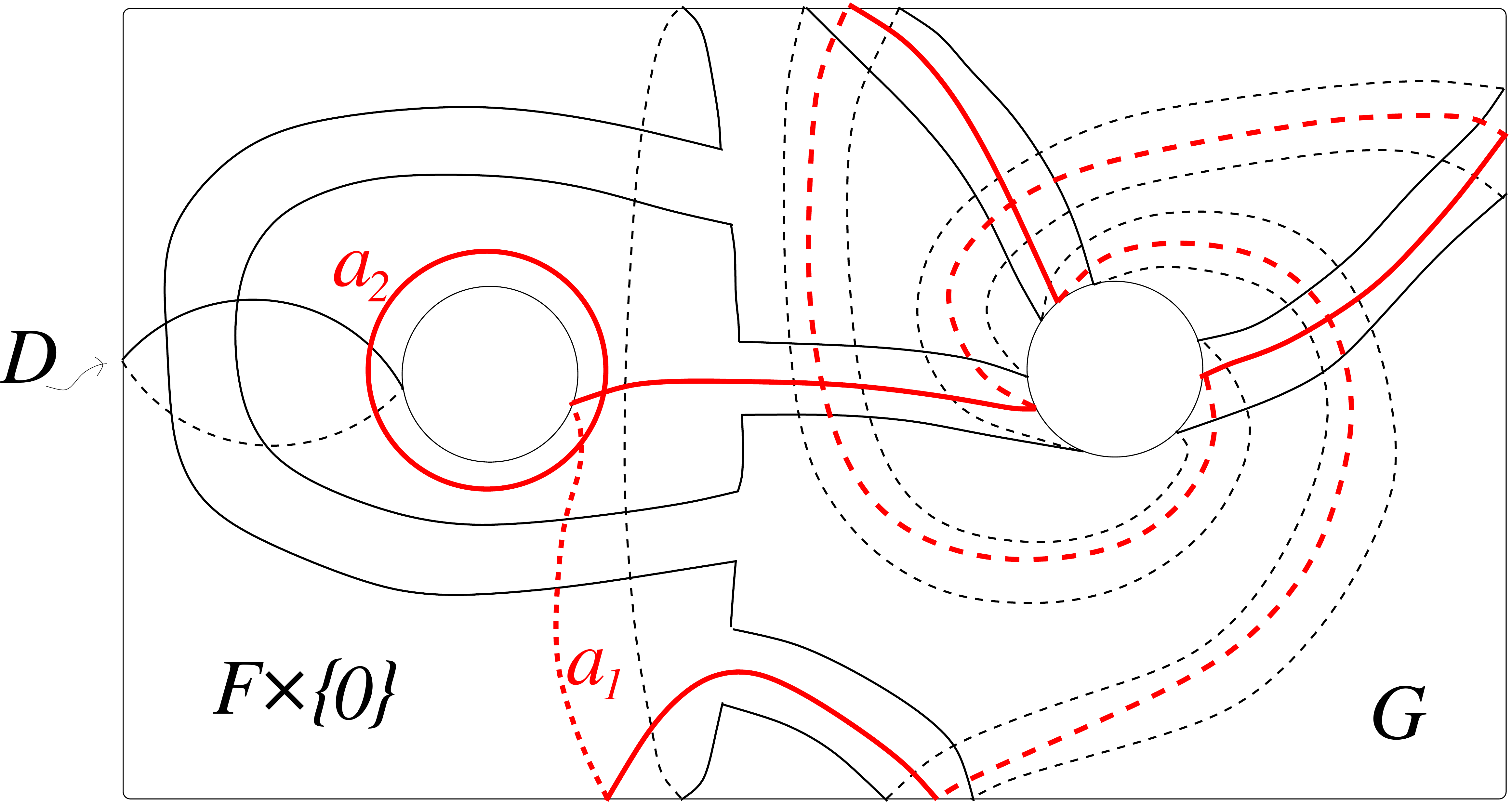}
\includegraphics[height=50mm]{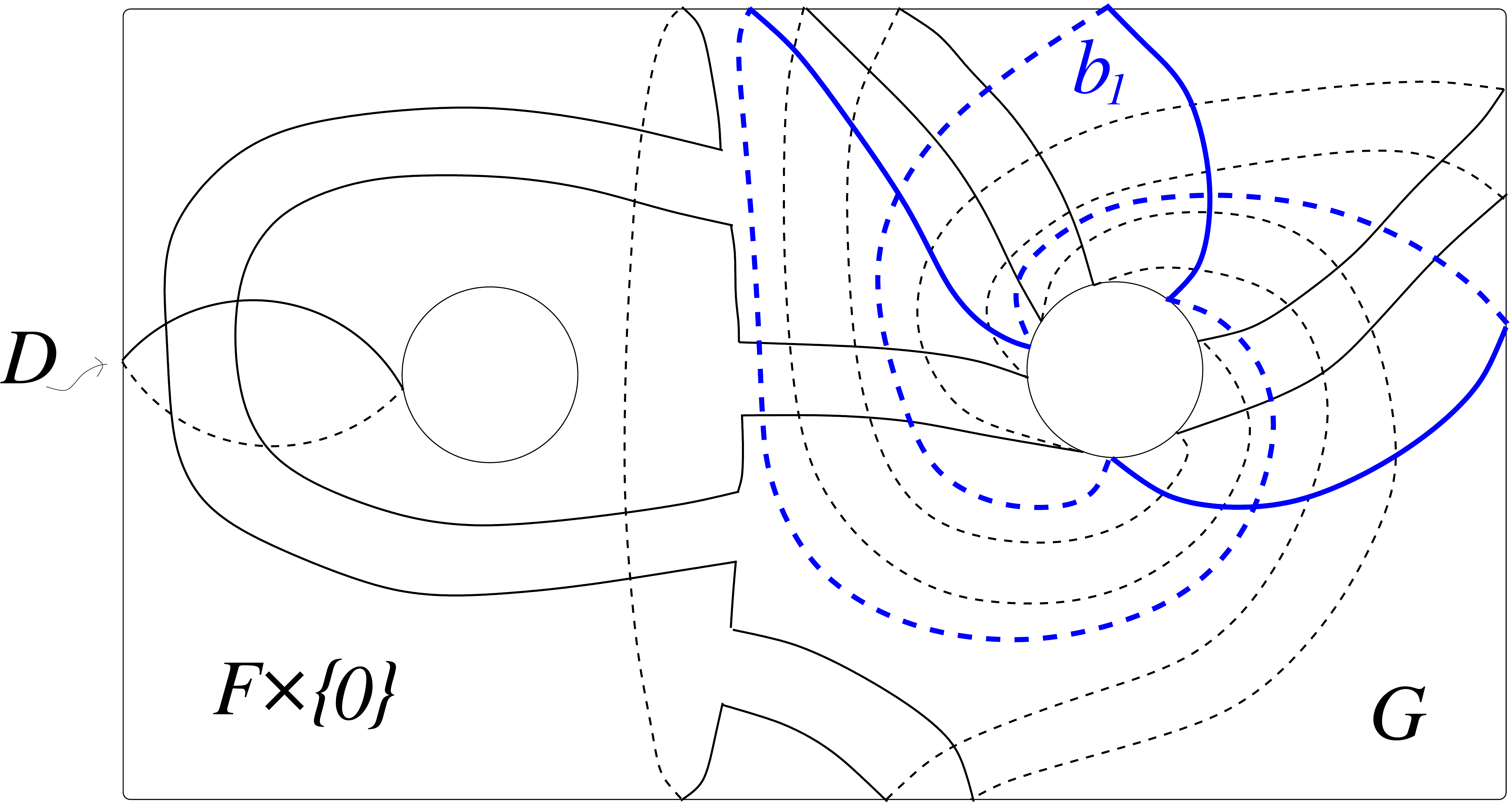}}
\caption{$\Gamma=a_1\vee a_2$ and $b_1$.}
\label{fig4}
\end{figure}

Let $C'\subset\partial V$ be a 2--disk that contains the pair of disks
$\partial V\cap\mathcal{N}(D)$, and let $C\subset H_0$ be a properly embedded
  disk with $\partial C=\partial C'$. Now let $Z\subset H_0$ be a
  meridional disk such that $Z\cap C=\emptyset$.
Then $\widetilde
F=(F\times\{0\})\cap(\overline{H_0-\mathcal{N}(Z)})$ contains a
(1,0)--annulus $A$
in the solid torus $\overline{H_0-\mathcal{N}(Z)}$. Let $a_2\subset
Int(F\times\{0\})$ be the core of $A$,
where we can arrange
that $a_1\cap a_2$ is just one point. Then $\Gamma=a_1\vee a_2$ is a
spine for $F$. See Figure~\ref{fig4}.

The curve $a_2$ spoils disks for $a_1$ in $E(F)$ for, otherwise, there is an
essential disk $D\subset E(F)$ such that $D\cap a_1=\emptyset$, and the
number of points $\#(D\cap a_2)<2$. If~$D\cap a_2=\emptyset$, since
$\Gamma$ is a spine for $F$, the surface $F$ is contained in the solid
torus~$E(D)\subset E(F)$; it follows that $F$ is compressible in
$E(D)$, and, 
thus, $F$ is compressible in $E(F)$. But, since $k$ is non-trivial, and
$g(F)=1$, $F$ is incompressible in $E(k)$. 
Then $D\cap a_2$ is just one point, and $D\cap\partial F$ is a set of two
points. We may assume that $D$ intersects $k=\partial G$ in exactly
two points. Since $G$ is incompressible, we may arrange that $D\cap G$
is just one arc. Now, this arc is essential in $G$ for, otherwise, we
can slide~$G$ along $D$, and obtain $G'$ homotopic to $G$ in $E(F)$ such
that $G'$ is  contained in
the solid torus $E(D)$; then $G'$ is not $\pi_1$--injective, and,
since~$G$ and $G'$ are homotopic embeddings, thus, $G$ is not 
$\pi_1$--injective; but that makes $G$ compressible. Then~$\widehat G=G\cap
E(D)$ is an annulus, therefore, $\widehat G$ is parallel into~$\partial
E(D)$. Using the disk~$D$ we can extend this parallelism to a
parallelism of $G$ into~$\partial E(F)$,
contradicting that $G$ is essential in $(E(F),k)$.

Now, $a_1\subset F\times\{0\}$ represents, up to conjugacy, the same
element as $b_1\subset G$ in~$\pi_1(H_0)$ for, they are disjoint curves
on a torus, and therefore, parallel.

Observe that, since $G$ is not parallel to $F\times\{0\}$, we have
$|n|\geq2$. In particular~$\widehat G$ and $\widehat F$ are not parallel 
in $V$.

We now explore $H_1$. 

Recall $D$ is a $\partial$--compression disk for $G$ in $E(F)$; in
particular $D\cap\partial E(F)$ is an arc. It follows that, to recover
$E(F)$ from $\overline{E(F)-\mathcal{N}(D)}$, we attach to
$\overline{E(F)-\mathcal{N}(D)}$ the 3--ball $\mathcal{N}(D)$ along a disk. Then
$\overline{E(F)-\mathcal{N}(D)}$ is a genus 2 handlebody. In
fact~$E(F)$ is a regular neighbourhood of
$\overline{E(F)-\mathcal{N}(D)}$. In particular, the inclusion induces
an isomorphism $\pi_1(\overline{E(F)-\mathcal{N}(D)})\rightarrow\pi_1(E(F))$.

Since $\overline{E(F)-\mathcal{N}(D)}=H_1\cup_{\widehat G}V$, then
$H_1\cup_{\widehat G}V$ is a genus two handlebody. Therefore, the core
$b_1$ of $\widehat G$ represents a primitive element
$\beta_1\in\pi_1(H_1)$ for, if~$\pi_1(V)=\langle v;-\rangle$, then
$b_1$ represents $v^n$, which is not primitive in $V$. The element~$\beta_1$
is part of a basis, say, $\pi_1(H_1)=\langle w,\beta_1:-\rangle$. By
Seifert-van Kampen, 
$\pi_1(E(F))\cong \pi_1(H_1\cup_{\widehat G}V) =\langle
w,\beta_1,v:\beta_1=v^n\rangle\cong \langle w,v:-\rangle$. That is, $v$ is
primitive in $\pi_1(E(F))$, and $b_1$ represents $v^n$.

\end{proof}


\section{Free genus one knots are almost fibered}
\label{sec6}
In this section we show that all free genus one knots are almost
fibered. We give here an outline of the plan of the proof:

Start with a non-fibered free genus one knot $k$ with a genus one free
Seifert surface~$F\subset E(k)$. If $k$ has a unique
Seifert surface, then $k$ is almost fibered
(Remark~\ref{remark24}). If $k$ were not almost fibered, then
as in Remark~\ref{remark37}, $k$ has a genus one Seifert surface not
isotopic to $F$. By Theorem~\ref{thm52}, there is a
spine~$\Gamma=a_1\vee a_2$ for $F$ in $\partial\mathcal{N}(F)$ such
that 
$a_1$ represents an element conjugate to~$g^p$ with $p\geq2$ for some
primitive element $g\in\pi_1(E(F))$, and $a_2$ spoils disks
for~$a_1$. By Lemma~\ref{lema51}, we can find an essential disk 
$\Delta\subset E(F)$ with $\Delta\cap a_1=\emptyset$, and
the exterior
$E(\Delta)=\overline{E(F)-\mathcal{N}(\Delta)}$ is the disjoint union of two
  solid tori, $V_0,V_1$ with, say, $a_1\subset\partial V_0$. We regard
  $\Delta\subset \partial V_0$. Then $\Gamma\cap V_0$ consists of the
  curve $a_1$, which is a $(p,q)$--curve in $V_0$, and an arc with
  endpoints on $\partial\Delta$ intersecting $a_1$ in exactly one
  point, and a set of parallel arcs with endpoints on $\partial\Delta$
  which are disjoint with $a_1$. See Figure~\ref{fig41}.

In Section~\ref{subsec61} we show how to find a properly embedded arc
in $V_0$ disjoint with~$\Gamma$ which, in Section~\ref{subsec62}, is shown to  be the core of the 1--handle
of a one-handled circular decomposition for $E(k)$ based on $F$.
 In 
this analysis, the disk $\Delta$ is regarded as `unreachable', and
should be thought as very near the point at infinity. That is, all
homeomorphisms in this subsection will fix point-wise the disk
$\Delta$.

\subsection{Handles for torus manifolds} 
\label{subsec61}
Let $p$ and $q$ be a pair of coprime integers. 
Consider the points~$\{s_\ell\}_{\ell=1}^{p} \subset S^1$ with
$s_\ell=e^{2\pi i\ell/p}$; also let $\widetilde V$ be 
the cylinder~$D^2\times I$, and write $s_\ell^I=s_\ell\times I\subset
\widetilde V$. The 
rotation~$\rho_q$ of angle $2\pi q/p$ on~$D^2$ gives a
quotient~$P:(\widetilde V,\cup_{\ell=1}^{p} s_\ell^I)\rightarrow (V,\alpha)$, 
where $V$ is the solid torus obtained from $\widetilde V$ by
identifying~$(z,0)$ with $(\rho_q(z),1)$ for each $z\in D^2$,
and~$\alpha$ is the simple closed curve on~$\partial V$ obtained as
the image of the union $\cup_{\ell=1}^{p}s_\ell^I$ in this quotient.  
The rotation~$\rho_q$ acts on $\{s_\ell\}_{\ell=1}^{p}$ as the cyclic
permutation of order $p$ such that $\rho_q(s_i)=s_{i+q}$ where
subindices are taken~$\mod p$.
We consider also a
fixed point~$\infty\in\alpha$, the `point at infinity'. The
homeomorphism type of the pair $(V,\alpha)$ is
called \emph{the $(p,q)$--torus sutured manifold}, or simply
\emph{the~$(p,q)$--manifold}.
Throughout this section we assume $0<q<p$. 
Notice that the $(p,q)$--torus sutured
manifold $(V,\alpha)$ is not a sutured manifold, but~$\alpha$ is a
spine of a small regular neighbourhood
$\mathcal{N}(\alpha)\subset\partial V$, and the pair
$(V,\mathcal{N}(\alpha))$ is a true sutured manifold with suture
$\alpha$.

In the following, we perform several operations on the
$(p,q)$--manifold (drilling of arcs, homeomorphisms, etc.), and it will
be done in such a way that the
point at infinity of the manifold will remain fixed.

Let $x\subset V$ be the meridional disk $P(D^2\times\{0\})$. From
the pair $(\widetilde V,\cup_{\ell=1}^{p} s_\ell^I)$ we give a
Whitehead diagram for the $(p,q)$--manifold $(V,\alpha)$
associated to $x$ as follows:

We regard $\tilde V=D^2\times I$ as the exterior $E(x)\subset V$, and
write $x$ and $\bar x$ for $D^2\times\{0\}$ and $D^2\times\{1\}$,
respectively. The arcs $s_1^I,\dots,s_p^I$ are the edges of $G$, the
corresponding Whitehead graph with fat vertices $x$ and $\bar x$. To
obtain a Whitehead diagram, we have to number the endpoints of
$s_1^I,\dots,s_p^I$. In a plane projection of the graph $G$, we assume
that the unbounded face of $G$ contains the edges $s_q^I$ and $s_{q+1}^I$.
See Figure~\ref{fig42}. The point at infinity is either the middle
point of $s_q^I$, or the middle point of $s_{q+1}^I$. If~$\infty\in
s_q^I$, then we rename $v_j=(s_j,0)$ and $\bar
v_j=(\rho_q(s_j),1)=(s_{j+q},1)$; if~$\infty\in s_{q+1}^I$, we rename
$v_j=(s_{j+q},0)$ and $\bar v_j=(\rho_q(s_{j+q}),1)=(s_{j+2q},1)$
where subindices are taken $\mod{p}$. In
any case, we number the point $v_i$ with the number~$i$, and the point
$\bar v_j$ with the number~$j$ ($i,j=1,\dots,p$). Also, we write
$\alpha_i$ for the edge of~$G$ such that $v_i\in\alpha_i$.
\begin{figure}
\centering
\centerline{\includegraphics[height=6true cm]{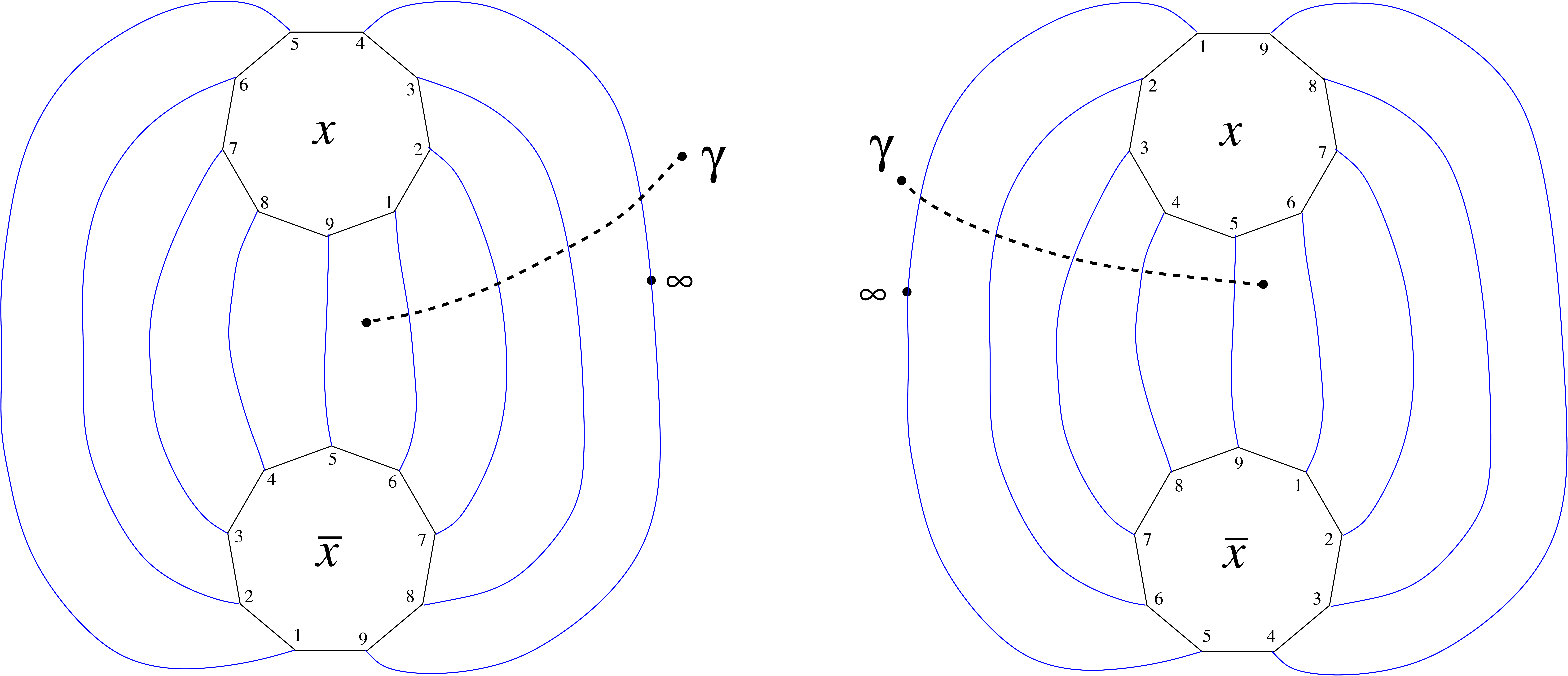}}
\caption{Whitehead diagrams for the (9,4)--manifold and the (9,5)--manifold}
\label{fig42}
\end{figure}
This diagram and the corresponding Whitehead graph are called
\emph{the~$(p,q)$--diagram} and \emph{the~$(p,q)$--graph},
respectively. Notice that the edge~$\alpha_1$ connecting $x$ with
$\bar x$ starting at the point numbered $1\in x$
ends at the point numbered $p-q+1\in\bar x$.

\begin{remark}
\label{remark61}
Consider a Whitehead diagram of a pair $(V,\alpha)$ associated to $x$
where~$V$ is a solid torus, $\alpha$ is a simple closed curve on
$\partial V$, and $x$ is a meridional disk of $V$. If in the
fat vertices of the Whitehead diagram of $(V,\alpha)$, the points
corresponding to 
ends of edges are numbered with elements of the set
$\{1,\dots,p\}$ consecutively in the positive (negative) direction on
$x$ (on $\bar x$), in a compatible way with the gluing homeomorphism
to recover the $V$, then if the edge connecting $x$ with $\bar x$
starting at the point 
numbered $1\in x$
ends at the point numbered $t\in\bar x$,
then~$t=p-q+1$; that is, the Whitehead diagram corresponds to the $(p,q)$--torus
sutured manifold with $q=p-t+1$. 
\end{remark}

Let $(V,\alpha)$ be the $(p,q)$--torus sutured manifold, and let $G$ be
the Whitehead graph of $(V,\alpha)$ with respect to a meridional disk
$x\subset V$.  
Let $\gamma$ be a properly embedded
arc in $V$, such that $\gamma$ is around the vertex $x$ in the
Whitehead diagram of $(V,\alpha)$ with respect to $x$, and $\gamma$
encircles the edges $\alpha_1,\dots,\alpha_q$.
Also, assume that~$\gamma$ lies `above' the point~$\infty\in\alpha$, that is,
$\gamma$ is between $\infty$ and $x$. See
Figure~\ref{fig42}. The arc $\gamma$ is called \emph{the canonical
  2--handle of length $q$} for the $(p,q)$--manifold. Note that
the arc $\gamma$ is the co-core of a 2--handle in $V$.

If we drill out the canonical 2--handle of length $q$, we obtain a
Whitehead diagram with respect to the system of disks $x,z\subset
E(\gamma)\subset V$ where $z$ is the obvious $\partial$--parallelism
disk for $\gamma$. See Figure~\ref{fig43}.  
\begin{figure}
\centering
\centerline{\includegraphics[height=9true cm]{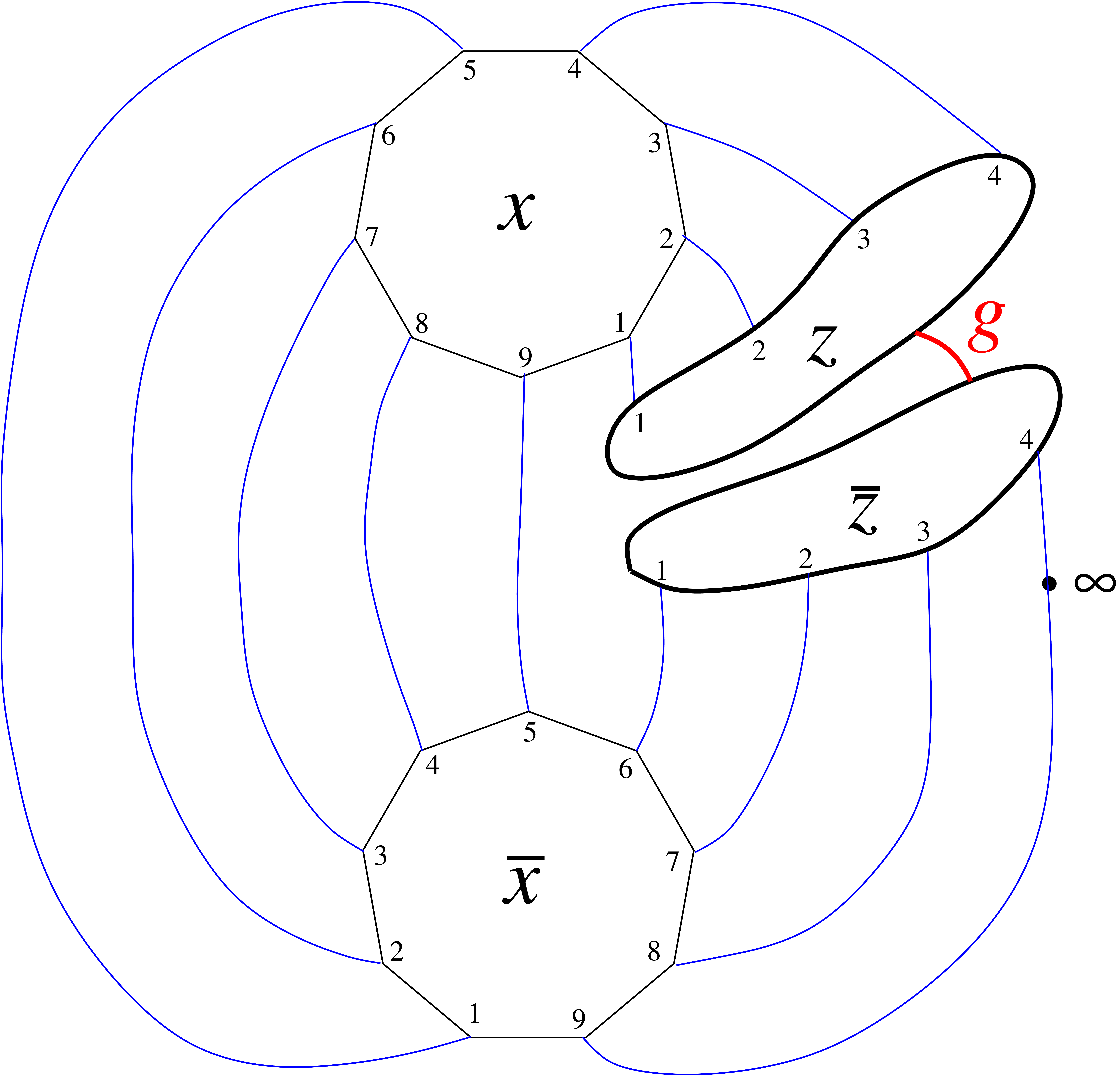}}
\caption{}
\label{fig43}
\end{figure}
We refer to this Whitehead
diagram as the \emph{Whitehead diagram obtained by drilling out the
  canonical 2--handle of length $q$} of the $(p,q)$--manifold. Notice that the arc $g$ in
Figure~\ref{fig43} is a `longitude' for the handle defined by
$z$. That is, if we glue back the disks $z$ and $\bar z$ and kill the
longitude $g$ with a 2--handle, we recover the Whitehead diagram of the
$(p,q)$--manifold. In practice, we just join the ends of the edges
in~$z$ with the 
ends of the edges in $\bar z$ with parallel arcs on the diagram, and
delete the disks $z$ and~$\bar z$ from the picture, and we  get
the Whitehead diagram of the $(p,q)$--manifold back.

Let $G$ be the graph of the Whitehead diagram obtained by drilling out the
canonical 2--handle of length $q$ of the $(p,q)$--manifold. Then $G$ is a
graph with four fat vertices $x,\bar x,z$, and $\bar z$; there are
$q$ edges connecting $z$ and $x$; there are $q$ edges 
connecting~$\bar z$ and $\bar x$; and there are $p-q$ edges connecting $x$
with $\bar x$. Compare with Figure~\ref{fig43}. Note that $x$ is a cut vertex
of $G$ (and $z$ and $\bar z$ are \emph{not} cut vertices); then we can slide the handle corresponding to $z$ along
the handle defined by $x$ 

After sliding, if the new disk $x$ is still a cut vertex,
we can again slide the new disk $z$ along the new disk $x$, and
so on. Let $G'$ be the image of the graph~$G$ after~$\kappa$ handle
slides of $z$ along $x$. The graph $G'$ is called \emph{the
  $\kappa$--slid graph obtained from the $(p,q)$--graph~$G$}. 

\begin{lemma}
Let $p,q$ be a pair of coprime integers, $0<q<p$,
and assume that
$$p=\kappa_1 q+r_1, \textrm{ with } 0\leq r_1<q, \textrm{ and } \kappa_1\geq 1$$ 
Let $G$ be the graph of the Whitehead diagram obtained by 
drilling out the canonical 2--handle of length~$q$ of the
$(p,q)$--manifold, and let $G'$ be the $\kappa_1$--slid graph obtained
from the $(p,q)$--graph~$G$. Then $G'$ is the graph of 
 the Whitehead diagram
obtained by drilling out the canonical 2--handle of length $r_1$ of the
$(q,r_1)$--manifold. The point at infinity is a fixed point of these
handle slides.
\end{lemma}
\begin{figure}
\centering
\centerline{\includegraphics[height=9true cm]{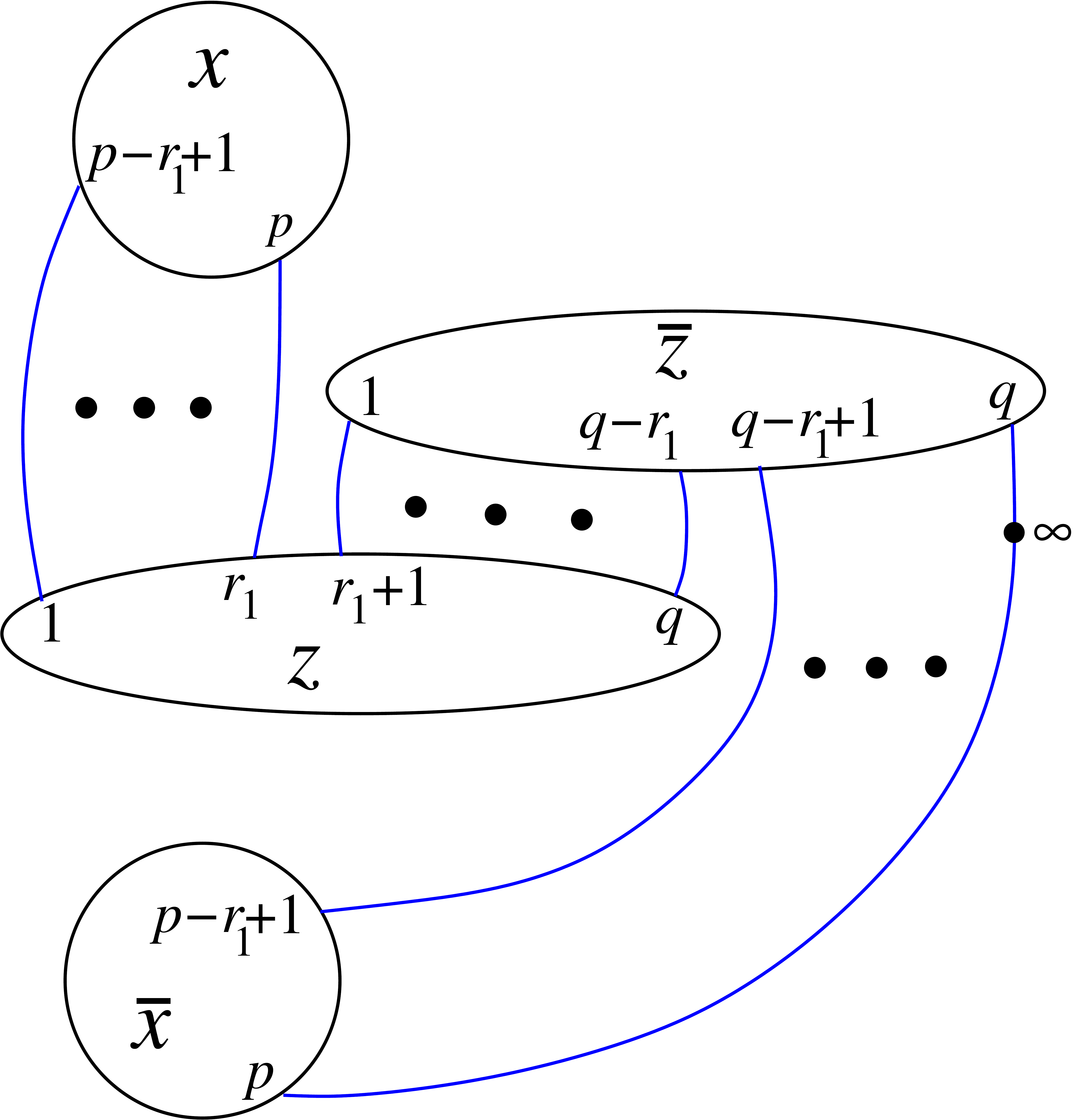}}
\caption{After sliding $z$ along $x$}
\label{fig44}
\end{figure}
\begin{proof}
In the Whitehead
graph $G$, the ends of the edges connecting the disk $z$ with the disk
$x$, are numbered $1,2,\dots,q$ in the disk~$x$; these ends are the
points $v_1,v_2,\dots,v_{q}$ in $\partial x$. Then, after
sliding $z$ along~$x$, the new disk $z$ 
carries the edges with ends that were numbered $1,2,\dots,q$ in $\bar
x$. Thus, now the ends of the edges connecting $z$ and $x$, after the
slide, have ends which are
the image of the rotation $\rho_q$ of angle $2\pi q/p$ of the
points~$v_1,v_2,\dots,v_{q}$; that is, the ends are the
points~$v_{q+1},v_{q+2},\dots,v_{2q}$ which are numbered
$q+1,q+2,\dots,2q$ in $x$. 

We see that after sliding $\kappa_1-1$ times~$z$ along $x$, the
ends of the edges connecting~$z$ and $x$ are numbered
$(\kappa_1-1)q+1,(\kappa_1-1)q+2,\dots,\kappa_1q$ in $x$. Then after
sliding~$\kappa_1$ times $z$ along $x$, the points still connected
by edges in $x$
are numbered~$\kappa_1q+1,\kappa_1q+2,\dots,p$. Now, by hypothesis
$p=\kappa_1 q+r_1$, 
then $\kappa_1q+1=p-r_1+1$, which means that there are $r_1$ points left in
$x$. That is, see Figure~\ref{fig44}, we have a graph, the image of
$G$ after the slides, with fat vertices 
$x,\bar x, z, \bar z$; there are $r_1$ edges connecting~$x$ with $z$;
there are $r_1$ 
edges connecting $\bar x$ with with $\bar z$; and there are~$q-r_1$ edges
connecting $z$ with $\bar z$. Now, the edge with one end in $z$ numbered
with 1 has the other end numbered with $p-r_1+1\in x$; and the edge
with one end in $\bar x$ numbered with $p-r_1+1$ has the other end in
$\bar z$ numbered with $q-r_1+1$. 

Therefore, the new diagram is the
Whitehead diagram obtained by drilling out the canonical 2--handle of
length $r_1$ of the $(q,r_1)$--manifold. Since the disks $\bar x$,
and~$\bar z$ were never touched, the point at infinity is a  fixed
point 
of the handle slides.

Notice that if $q=1$, then $\kappa_1=p$, and $r_1=0$, and everything
is easier: The image
graph $G$ above, in this case, replacing the values of $q$ and $r_1$, has
four fat vertices~$x,\bar x, z, \bar z$; there are~$0$ edges
connecting $x$ with $z$; 
there are $0$ 
edges connecting~$\bar x$ with with $\bar z$; and there is $1$ edge
connecting $z$ with $\bar z$. That is, after canceling the handle
defined by~$x$, we obtain the (1,0)--manifold.
\end{proof}


\begin{corollary}
\label{coro63}
Let $r_1,r_2$ be a pair of coprime integers, $0<r_2<r_1$. Assume
\begin{equation*}
   \begin{array}{lll}
r_1=&\kappa_1 r_2+r_3, &   0< r_3<r_2\\
r_2=&\kappa_2 r_3+r_4,  &  0<r_4<r_3\\
&\hfill\vdots \hfill& \null\hfill\vdots\hfill\\
r_{n-1}=&\kappa_{n-1}r_{n}+1,& 0<1<r_{n}\\
r_{n}=&\kappa_n,
    \end{array} 
\end{equation*}
with $\kappa_i\geq1$, $i=1,\dots,n$.

Let $G$ be the graph of the Whitehead diagram obtained by 
drilling out the canonical 2--handle of length~$r_2$ of the
$(r_1,r_2)$--manifold. 
Let $G_1$ be the $\kappa_1$--slid graph obtained
from the $(r_1,r_2)$--graph~$G$. 
For $i=1,\dots,n-1$, let $G_{i+1}$ be the~$\kappa_{i+1}$--slid graph obtained 
from the $(r_{i},r_{i+1})$--graph~$G_i$.

Then $G_n$ is the graph of the Whitehead diagram
obtained by drilling out the canonical 2--handle of length $0$ of the
$(1,0)$--manifold $(V,\alpha)$. 

The point at infinity is a fixed point of these
handle slides.

\end{corollary}
\hfill $\Box$

\begin{remark}
\label{remark64}
The graph $G_i$ in the statement of Corollary~\ref{coro63} is the
graph of the Whitehead diagram obtained by drilling out the 
canonical 2--handle of length $r_{i+2}$ of the
$(r_{i+1},r_{i+2})$--manifold. Then $G_i$ is a 
graph with four fat vertices $\xi,\bar\xi,\zeta$, and~$\bar\zeta$. The
symbols $\xi$ and $\zeta$ stand for the symbols $x$ and $z$ in some
order (that is, the sets $\{\xi,\zeta\}$ and $\{x,z\}$ are equal, but
just as unordered sets). There are
$r_{i+2}$ edges connecting $\zeta$ and $\xi$; there are $r_{i+2}$
edges connecting~$\bar\zeta$ and $\bar\xi$; and there
are~$r_{i+1}-r_{i+2}$ edges connecting $\xi$ 
with $\bar\xi$.
\end{remark}

\begin{remark}
\label{remark65}
Let $p,q$ be a pair of coprime integers, 
and assume that $p/q=[\kappa_1,\dots,\kappa_n]$, as a continued
fraction, with $\kappa_i\geq1$ for each $i$. 
\begin{enumerate}
\item \label{remark651} Write $p_i/q_i=[\kappa_1,\dots,\kappa_i]$ with $p_i,q_i$
  coprimes. Write $p_0=1,
  p_{-1}=0$, and $q_0=0, q_{-1}=1$. It is well known that
  $p_i=\kappa_i p_{i-1}+p_{i-2}$, and $q_i=\kappa_i q_{i-1}+q_{i-2}$;
  also
  $p_iq_{i-1}-p_{i-1}q_i=(-1)^i$
  for $i\geq1$. (Article~337 and~338 of~\cite{HK}). Since
  $\kappa_i\geq1$, one easily shows $p_i>q_i>0$ for $i\geq1$. In
  particular, $p>q>0$. Note also that~$p_{i+1}>p_{i}$.  

\item Let $r,s$ be the two coprime integers $p_{n-1},q_{n-1}$,
  respectively, and let $(V,\alpha)$ be the $(p,q)$--manifold. Then the
  $(r,s)$--torus curve can be drawn on~$\partial 
  V$ as a simple closed curve, $\beta$, which intersects~$\alpha$ exactly at the
  point at infinity for
  $ps-qr=\pm1$. 
  Note that if $n$ is even, then the point at infinity
  is at the right in the Whitehead diagram, and if $n$ is odd, it is
  at the left, as in 
  Figure~\ref{fig42}.
  The curve $\beta$ can be visualized on the Whitehead
  diagram of the $(p,q)$--manifold as a set of new edges connecting the fat
  vertices, and disjoint with the Whitehead graph, and a single new edge
  intersecting the Whitehead graph at the point at infinity. 
  Conversely,
  the curve $\alpha$ can be visualized in a similar way on the Whitehead
  diagram of the $(r,s)$--manifold.
  
  Notice that between two edges of $\alpha$, there is at most one edge
  of $\beta$ for,~$p>r$.
\end{enumerate}
\end{remark}

\begin{theorem}
\label{theorem66}
Assume $p/q=[\kappa_1,\dots,\kappa_n]$ with $p,q$ coprime, and $\kappa_i\geq1$
for each~$i$. Let $r,s$ be the pair of coprime integers such that
$r/s=[\kappa_1,\dots,\kappa_{n-1}]$.
Let $(V,\alpha)$ be the $(p,q)$--manifold, and let
$\beta\subset \partial V$ be the $(r,s)$--torus curve such that
$\alpha$ intersects~$\beta$ exactly at the point at infinity. 

If
$\gamma\subset V$ is the canonical 2--handle of length $q$ of
the $(p,q)$--manifold, then the
exterior $E(\gamma)$ is a regular neighbourhood of $\alpha\cup\beta$.
\end{theorem}
\begin{proof}
Let $G$ be the graph of the Whitehead
diagram obtained by drilling out the canonical 2--handle of length~$q$ of
the $(p,q)$--manifold, but including the arcs of the curve~$\beta$.
Call \emph{$\alpha$--edges} the edges of $G$ corresponding to the
$(p,q)$--torus curve $\alpha$, and~\emph{$\beta$--edges} the edges of
$G$ corresponding 
to the $(r,s)$--torus curve~$\beta$.

Writing $r_1=p$, and $r_2=q$, the statement
$p/q=[\kappa_1,\dots,\kappa_n]$ with $\kappa_i\geq1$
means: there are integers $r_3,\dots,r_n$
such that
\begin{equation*}
   \begin{array}{lll}
r_1=&\kappa_1 r_2+r_3, &   0< r_3<r_2\\
r_2=&\kappa_2 r_3+r_4,  &  0<r_4<r_3\\
&\hfill\vdots \hfill& \null\hfill\vdots\hfill\\
r_{n-1}=&\kappa_{n-1}r_{n}+1,& 0<1<r_{n}\\
r_{n}=&\kappa_n.
    \end{array} 
\end{equation*}
See Remark~\ref{remark65},~(\ref{remark651}). Writing $\rho_1=r$ and
$\rho_2=s$, the statement $r/s=[\kappa_1,\dots,\kappa_{n-1}]$ means:
there are integers~$\rho_3,\dots,\rho_{n-1}$ 
such that
\begin{equation*}
   \begin{array}{lll}
\rho_1=&\kappa_1 \rho_2+\rho_3, &   0< \rho_3<\rho_2\\
\rho_2=&\kappa_2 \rho_3+\rho_4,  &  0<\rho_4<\rho_3\\
&\hfill\vdots \hfill& \null\hfill\vdots\hfill\\
\rho_{n-2}=&\kappa_{n-2}\rho_{n-1}+1,& 0<1<\rho_{n-1}\\
\rho_{n-1}=&\kappa_{n-1}.
    \end{array} 
\end{equation*}

Notice that the canonical 2--handle of length $q$ for the $(p,q)$--manifold
is the canonical 2--handle of length $q$ for the $\alpha$--edges of
$G$, but it is 
also the canonical 2--handle of length~$s$ for the $\beta$--edges of $G$. Then
the graph $G_{n-1}$ of Corollary~\ref{coro63} (Remark~\ref{remark64})
contains four fat 
vertices 
$\xi,\bar\xi,\zeta$, and $\bar\zeta$. Note $r_{n+1}=1$;
then there is
a single $\alpha$--edge connecting~$\zeta$ and $\xi$; there is a
single $\alpha$--edge connecting~$\bar\zeta$ and $\bar\xi$; and there
are $r_{n}-1$ $\alpha$-edges connecting $\xi$ 
with $\bar\xi$.
Note that $\rho_n=1$ and $\rho_{n+1}=0$; then there is a single $\beta$--edge
connecting $\xi$ with $\bar \xi$ intersecting the
$\alpha$--edge connecting $\bar\zeta$ and $\bar\xi$ at the point at
infinity; and there are no more $\beta$--edges. The graph $G_n$ is
obtained by sliding $\zeta$ through $\xi$ the number $\kappa_n=r_n$ of
times. Then $G_n$ has a single $\alpha$--edge connecting $\xi$ 
with $\bar\xi$ and a single $\beta$--edge connecting $\zeta$ with $\bar
\zeta$ intersecting at the point at infinity. The theorem follows.

Notice that when $q=1$, $n=1$, the graph $G_{n-1}=G$.

\end{proof}


\subsection{One-handledness of knots}
\label{subsec62}
\begin{theorem}
\label{thm67}
If $k$ is a non-fibered free genus one knot in $S^3$, then $k$ is
almost fibered.
\end{theorem}
\begin{proof}
Let $k\subset S^3$ be a knot, and let $F\subset E(k)$ be a genus one
free Seifert surface for~$k$. Assume $k$ is not almost fibered. Then,
as in
Remark~\ref{remark37}, $k$ has another genus
one Seifert surface disjoint and not equivalent to $F$. By
Theorem~\ref{thm52} there is a spine $\Gamma=a_1\vee a_2$ for $F$
in $\partial\mathcal{N}(F)$ such that~$a_1$ represents an element
conjugate to $g^p$ with $p\geq2$, for some primitive
element~$g\in\pi_1(E(F))$, and $a_2$ spoils the disks of $a_1$. 
We shall show that the existence of such graph $\Gamma$ implies
$h(F)=1$, and, since $F$ is of minimal genus,  therefore,
$cw(k)=4$. This 
contradiction gives the theorem.

By 
Corollary~\ref{lema51}, there is an essential 2--disk $\Delta\subset E(F)$
such that~$\Delta\cap a_1=\emptyset$. We may assume that the exterior
$E(\Delta)\subset E(F)$ is not
connected, and is the union of two solid tori $H_0$ and $H_1$ with
$a_1\subset H_0$. There
is a copy of~$\Delta$ in~$\partial H_0$; then~$a_1\subset\partial
H_0-\Delta$. Write $T=\overline{\partial H_0-\Delta}$; $T$ is a once
punctured torus. A properly embedded arc
$\alpha\subset T$ is called a rel~$\Delta$ curve in $\partial H_0$, and is
visualized as the arc~$\alpha$ union a properly embedded arc in $\Delta$
with the same ends as $\alpha$. Or rather, we may regard~$\Delta$ as a
point at infinity of the torus $T/\partial \Delta$.

We have that $a_1$ is a $(p,q)$--torus curve
in $H_0$ for some $q$ (this implies that we have fixed a
longitude-meridian pair in $\partial H_0$; by changing the
longitude-meridian pair, we may assume that $0<q<p$). The intersection
$a_2\cap H_0=a_2\cap\partial H_0$ is a set
of disjoint arcs $c\cup b_1\cup\cdots\cup b_m\subset\partial H_0$ with
ends in $\partial \Delta$ and such 
that $b_i\cap a_1=\emptyset$ for each $i$, and the set $c\cap a_1$ is a single
point, the base point of $\Gamma$. 

Regarding $c$ as a rel~$\Delta$ curve, $c$ is an $(r,s)$--torus rel~$\Delta$
curve in~$H_0$ with $ps-qr=\pm1$. 
Since
$ps-qr=\pm1$, any other pair $(r',s')$ such that $ps'-qr'=\pm1$ is of
the form~$(r',s')=(r+\ell p,s+\ell q)$ for some integer $\ell$.
Then by sliding $a_2$ along~$a_1^{\pm1}$ several times, we obtain a
new spine for $F$. By Remark~\ref{remark26}, we may
assume that the arc $c$ is an 
$(r,s)$--torus rel~$\Delta$ curve in~$H_0$ where, if
$p/q=[\kappa_1,\dots,\kappa_n]$ as a continued fraction with terms
$\kappa_i\geq1$, then~$r/s=[\kappa_1,\dots,\kappa_{n-1}]$.

Since $b_1,\dots,b_m\subset\partial H_0-(Int(\Delta)\cup a_1\cup
c)\cong D^2$, then each of
$b_1,\dots,b_m$ are rel~$\Delta$ curves parallel to $a_1$. 

%
\begin{figure}
\centering
\centerline{\includegraphics[height=13true cm]{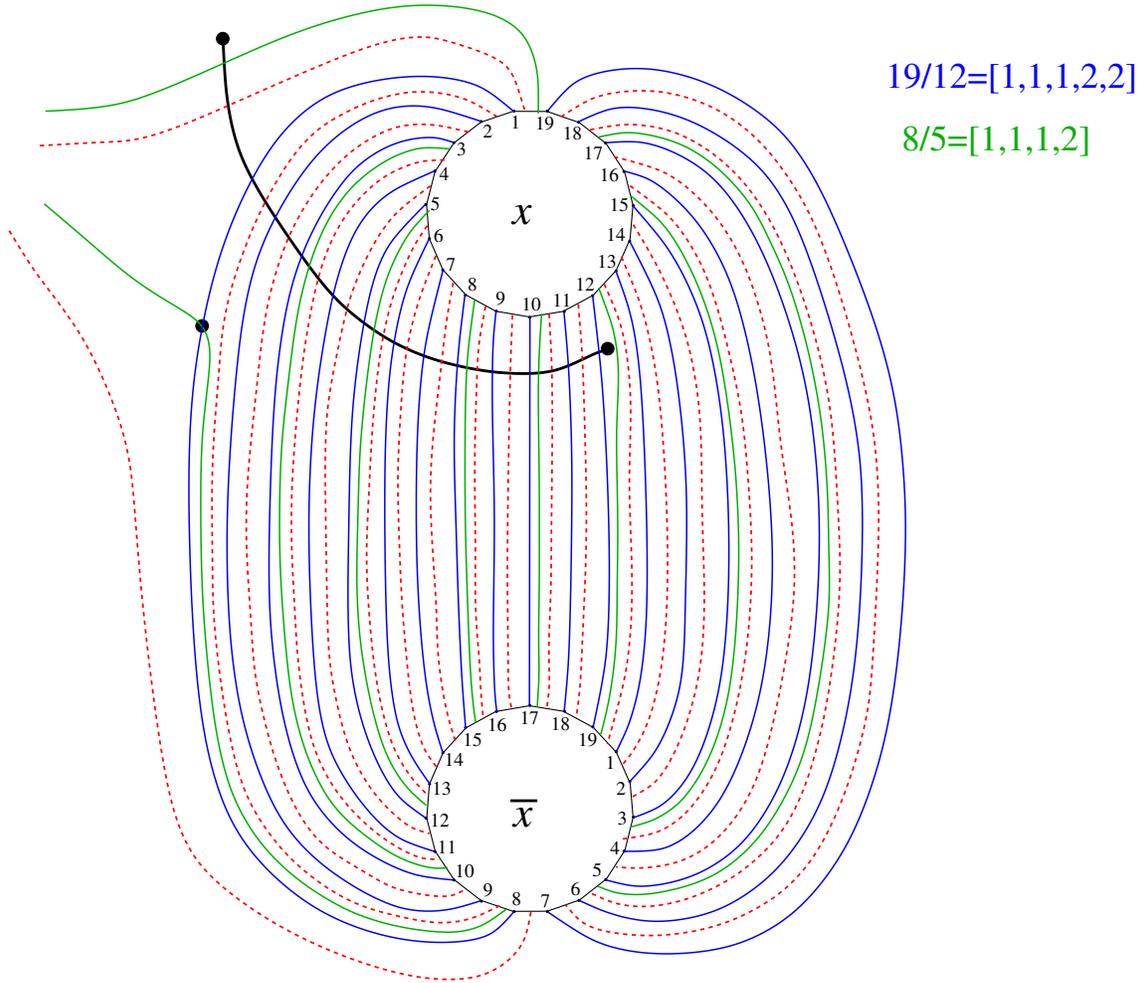}}
\caption{The (19,12) and (8,5)--torus curves}
\label{fig41}
\end{figure}

Now, consider the graph $G$ of the Whitehead diagram of the
$(p,q)$--ma\-ni\-fold $(H_0,a_1)$, and include in $G$ the edges induced by
the rel~$\partial$ curves $c,b_1,\dots,b_m$. By deforming the diagram,
we may assume that $\Delta$ is contained in a small neighbourhood of
the point at infinity which is the base point of $\Gamma$,
the point of intersection of $c$ and $a_1$. Let~$\gamma$ be the
canonical 2--handle of length $q$ for $(H_0,a_1)$.  In
the Whitehead diagram, we place $\gamma$ in
such a way that it starts by encircling the arc $c$ coming from
infinity,
and then encircles the $q$ edges belonging to $a_1$ and whatever is in
the middle, and nothing more (that is, after encircling the last edge
belonging to $a_1$, the arc $\gamma$ does not encircle any arc belonging to
$c$ or $b_1,\dots,b_m$). See Figure~\ref{fig41} where the dotted line
is a set of parallel arcs.
We drill out
$\gamma$ and, by Theorem~\ref{theorem66}, if we slide handles in the
Whitehead diagram obtained by drilling $\gamma$ out of $H_0$, we
obtain a sequence of diagrams as in
Figures~\ref{fig141}-\ref{fig146}. All handle slides fix 
point-wise the small neighbourhood of the point at infinity, and, thus, also
the disk~$\Delta$. 

The resulting Whitehead graph on $\partial H_0$ consists of four fat
vertices $\xi,\bar\xi,\zeta,\bar\zeta$; there is a single $a_1$--edge
connecting $\xi$ and $\bar\xi$, and a single $c$--edge connecting
$\zeta$ with~$\bar\zeta$ intersecting in the base point of $\Gamma$
(In Figure~\ref{fig41}, $\xi=z$ and $\zeta=x$). Notice that the
$c$--arc is actually two arcs, one connecting $\zeta$ with $\partial
\Delta$, and the other connecting $\partial\Delta$ with
$\bar\zeta$. Without lost of generality, this last arc contains the
base-point of $\Gamma$.

Let $v$ be a meridional disk for $H_1$ disjoint with $\Delta$. Then
$\xi,\zeta$ and $v$ is a system of meridional disks for the handlebody
$E(\gamma)$. Write $\pi_1(E(\gamma))=\langle
\xi,\zeta,v:-\rangle$. Then~$a_1$ represents the element $\xi$, and $a_2$
represents an element $\bar\zeta\cdot W(\xi,v)$ where $W(\xi,v)$ is a word
in the letters $\xi$ and $v$. Since $\{\xi, \bar\zeta\cdot
W(\xi,v),\zeta\}$ is a basis for $\pi_1(E(\gamma))$, it follows that
$a_1$ and $a_2$ 
represent associated primitive elements. Then we can find a system of
disks $D_1,D_2,D_3$ for $E(\gamma)$ such that $a_i\cap D_i$ is exactly
one point, and~$a_i\cap D_j=\emptyset$ for $i\neq j$, $i=1,2$, and
$j=1,2,3$. Therefore, $\overline{E(\gamma)-\mathcal{N}(D_3)}$ is a
regular neighbourhood of $\Gamma=a_1\vee a_2$. We conclude that $D_3$
is the co-core of a 1--handle that, together with $\gamma$, gives a
one-handled circular decomposition for~$E(k)$ as in
Remark~\ref{remark23}~(\ref{dos}). Since $k$ is not fibered, it
follows that $h(k)=1$, and that $k$ is almost fibered. This
contradiction finishes the proof of the theorem.

\def\hei{8}
\begin{figure}
\centering
\centerline{\includegraphics[height=\hei true cm]{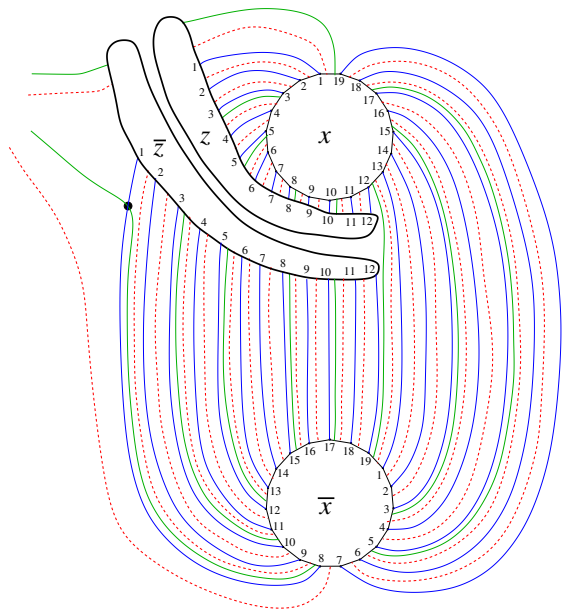}}
\caption{Slide $z$ along $x$}
\label{fig141}
\end{figure}
\begin{figure}
\centering
\centerline{\includegraphics[height=\hei true cm]{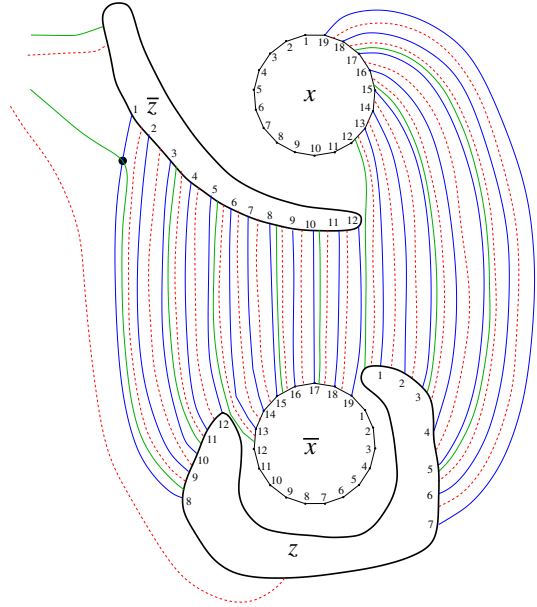}}
\caption{Slide $x$ along $z$}
\label{fig142}
\end{figure}
\begin{figure}
\centering
\centerline{\includegraphics[height=\hei true cm]{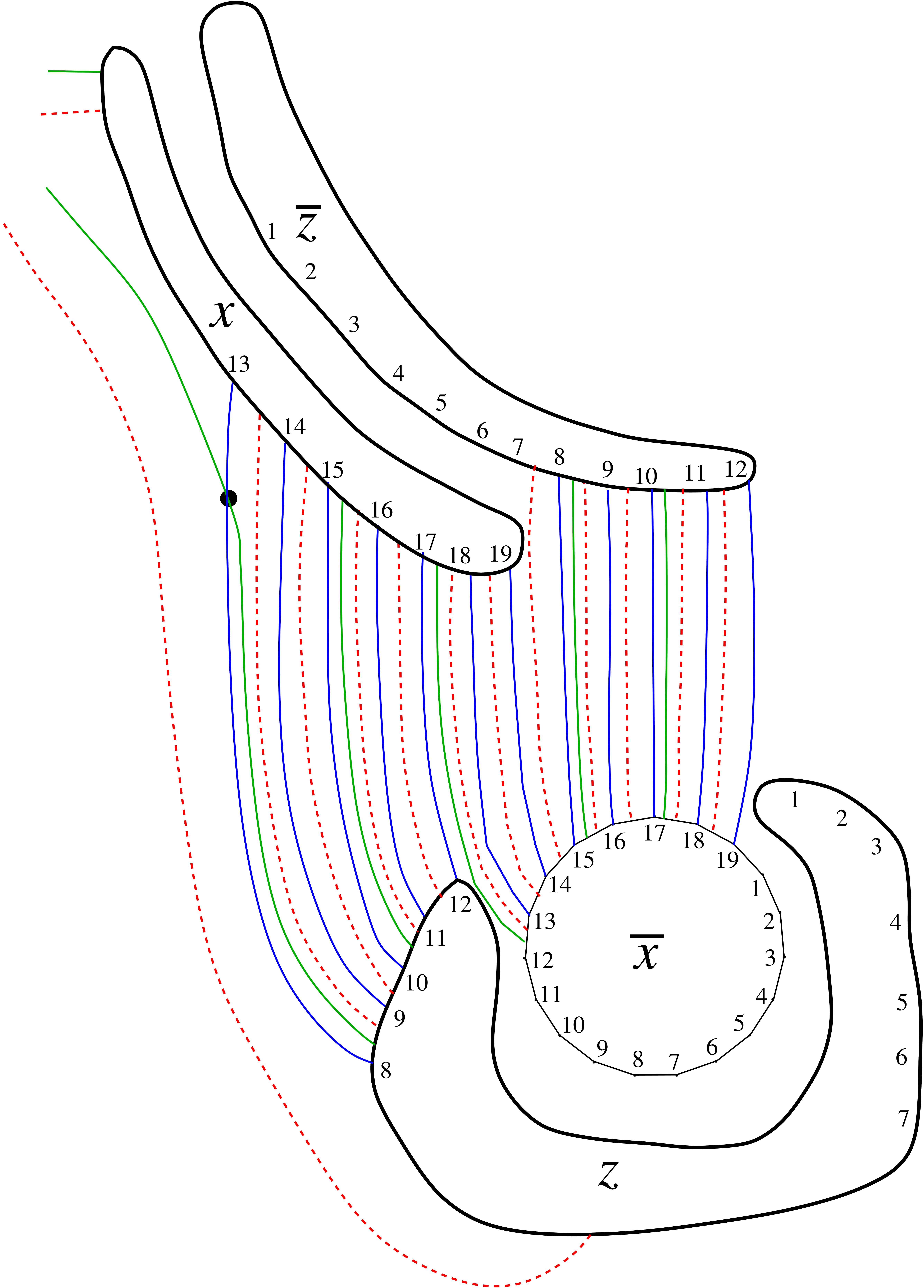}}
\caption{Slide $\bar z$ along $\bar x$}
\label{fig143}
\end{figure}
\begin{figure}
\centering
\centerline{\includegraphics[height=\hei true cm]{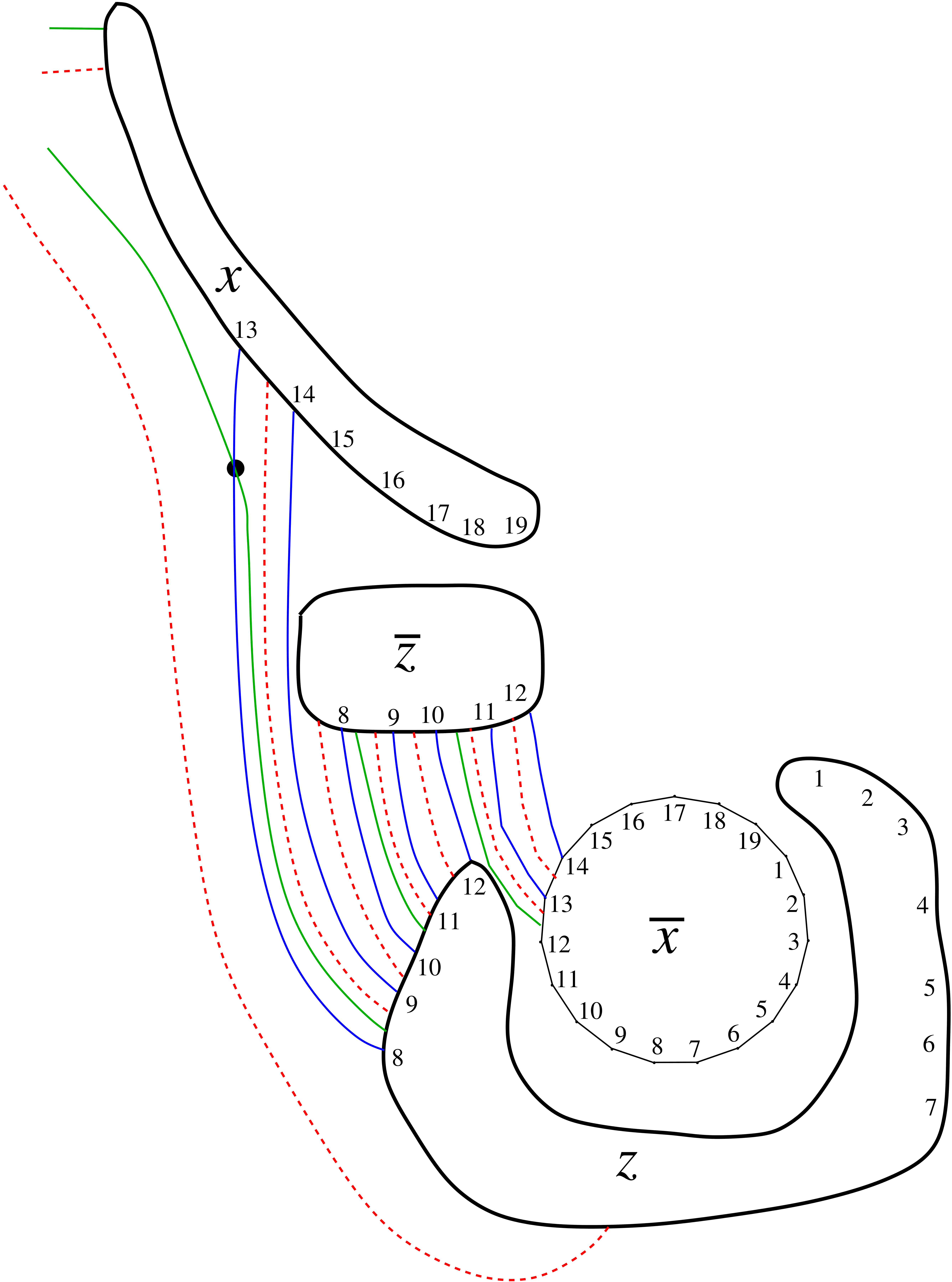}}
\caption{Slide twice $\bar x$ along $\bar z$}
\label{fig144}
\end{figure}
\begin{figure}
\centering
\centerline{\includegraphics[height=\hei true cm]{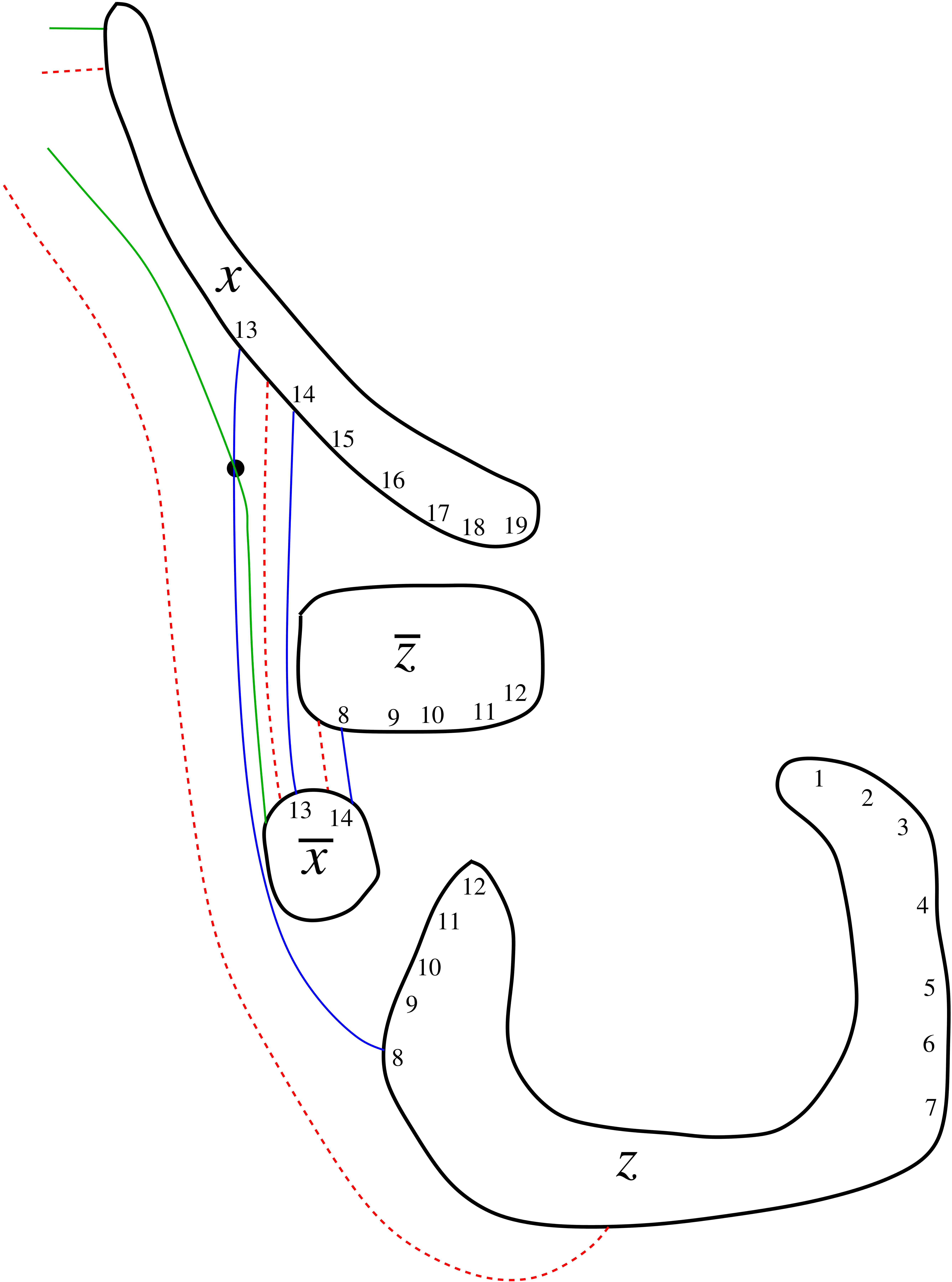}}
\caption{Slide twice $\bar z$ along $\bar x$}
\label{fig145}
\end{figure}
\begin{figure}
\centering
\centerline{\includegraphics[height=\hei true cm]{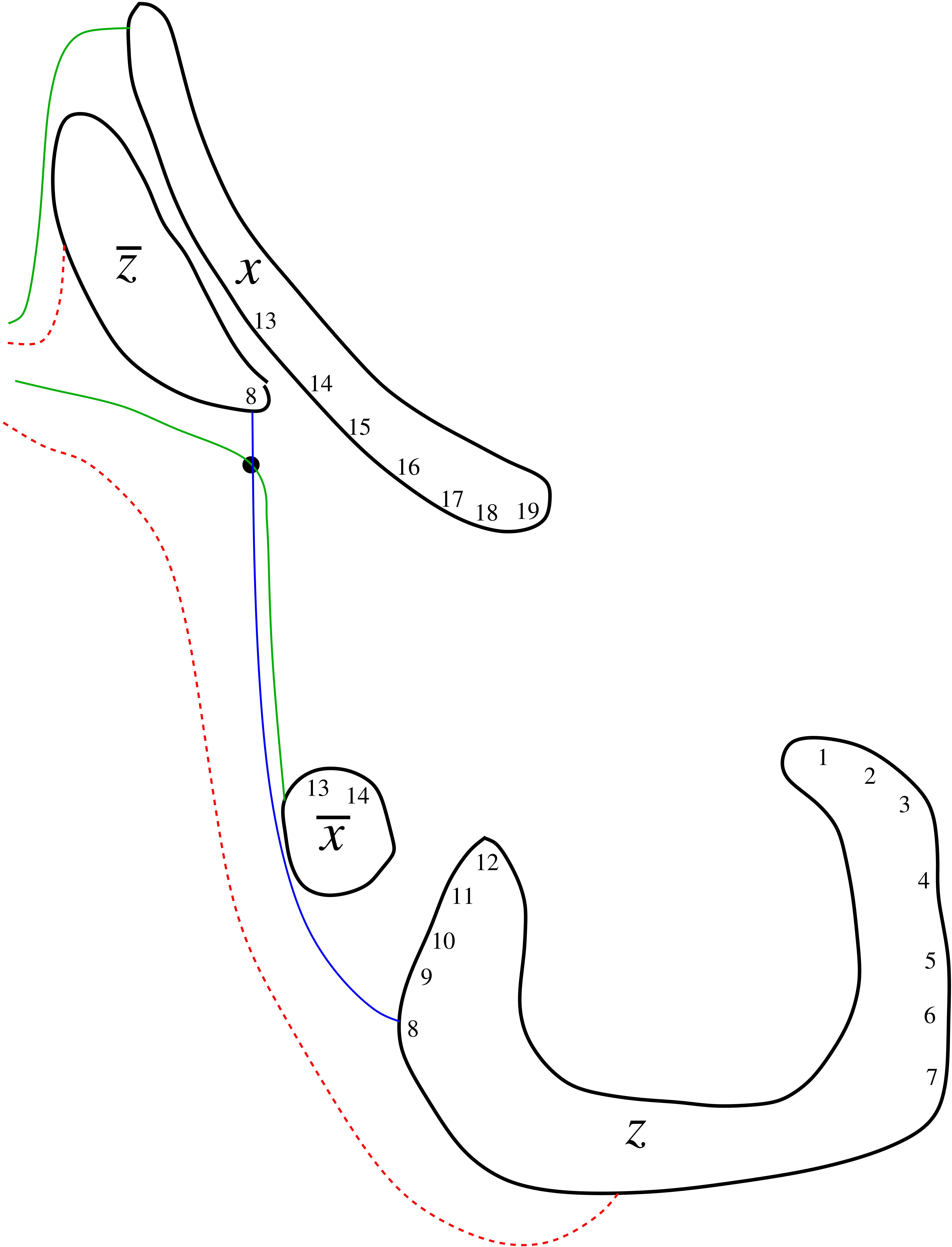}}
\caption{A long slide of $x$ deletes curve}
\label{fig146}
\end{figure}

\end{proof}

\begin{remark}
By \cite{pajilon}, a tunnel number one knot admits a one-handled
circular decomposition based on some not specified surface. In
\cite{charle} genus one  knots with tunnel number one were classified,
and it turns out that these knots are free genus one knots. Let $k$
be a non fibered genus one  knot with tunnel number one. In
Example~\ref{ex38}, we 
considered the case that $k$ is simple, and in the proof of
Theorem~\ref{thm67}, we considered the case that $k$ is not simple. It
follows that for these knots, their circular width is realized with a
one-handled circular decomposition based on a minimal (genus one) free
Seifert surface.

\end{remark}

\end{document}